\documentclass[onefignum,onetabnum]{siamart190516}  

\usepackage{amsfonts}
\usepackage{amsopn}
\usepackage{amsmath}
\usepackage{mathtools}
\usepackage{multirow}
\usepackage{rotating,tabularx}
\usepackage{hhline}
\usepackage{scalerel,stackengine}
\usepackage{arydshln}
\usepackage[normalem]{ulem}
\usepackage{soul}
\usepackage{tikz}
\usepackage{pgf}
\usepackage{collcell}
\usepackage[utf8]{inputenc}
\usepackage{geometry}
\usepackage{xfp}
\usepackage{siunitx}
\usepackage[export]{adjustbox}
\usepackage{numprint}
\newcolumntype{E}{>{\collectcell\usermacro}c<{\endcollectcell}}
\newcommand\usermacro[1]{\fpeval{1000*#1}\pgfmathprintnumber\pgfmathresult}

\xdefinecolor{samsi_green}{rgb}{0.04, 0.85, 0.32}
\xdefinecolor{samsi_darkgreen}{rgb}{0.24, 0.7, 0.44}
\xdefinecolor{samsi_mint}{rgb}{0.24, 0.71, 0.54}
\xdefinecolor{samsi_officegreen}{rgb}{0.0, 0.5, 0.0}
\xdefinecolor{samsi_napiergreen}{rgb}{0.16, 0.5, 0.0}
\xdefinecolor{samsi_maroon}{rgb}{0.65,0.06,0.17}
\xdefinecolor{samsi_blue}{rgb}{0,0.2,0.6}
\xdefinecolor{samsi_phthalogreen}{rgb}{0.07,0.21,0.14}
\definecolor{grassgreen}{RGB}{92,135,39}
\newcommand{\Oh}[1]{\mathcal{O}{\left(#1\right)}}
\newcommand{\brexp}[1]{\exp{\left({#1}\right)}}                                      
\newcommand{\logdet}[1]{\log\det\!\left(#1\right)}                    
\newcommand{\diag}[1]{\mathsf{diag}\left(#1\right)}                 
\newcommand{\Diag}[1]{\mathsf{Diag}\left(#1\right)}                 
\newcommand{\del}[2]{\frac{\partial{#1}}{\partial{#2}}}             
\newcommand{\mat}[1]{\mathbf{{#1}}}                                 
\renewcommand{\vec}[1]{\mathbf{{#1}}}                                 

\DeclareMathOperator*{\argmin}{arg\,min}  

\newcommand*{\tran}{^{\mkern-1.5mu\mathsf{T}}}                
\newcommand*{\opt}{^{\mathrm{{opt}}}}                
\newcommand*{\adj}{^{\mkern-1.5mu\mathsf{*}}}                 
\newcommand*{\inv}{^{\mkern-1.5mu\mathsf{-1}}}                
\newcommand{\pseudoinv}[1]{\left(#1\right)^{\dagger}}         
\newcommand{\domain}{\mathcal{D}}                             
\newcommand{\trace}{\mathrm{Tr}}                              
\newcommand{\Trace}[1]{\trace\!\left(#1\right)}                
\newcommand{\norm}[1]{\left\| {#1} \right\|}                  
\newcommand{\wnorm}[2]{\left\| {#1} \right\|_{#2}}            
\newcommand{\sqwnorm}[2]{\left\| {#1} \right\|^2_{#2}}        
\newcommand{\ip}[2]{{\left\langle {#1}, {#2} \right\rangle}}  
\newcommand{\wip}[3]{\left\langle{#1},                        
    {#2}\right\rangle_{\!\scriptscriptstyle{\mathrm{#3}}}}
\newcommand\restr[2]{{ \left.\kern-\nulldelimiterspace        
                     {#1}\vphantom{\big|} \right|_{#2}}}

\newcommand{\Rnum}{\mathbb{R}}  
\newcommand{\xcont}{u}                             

\newcommand{\x}{\vec{x}}                           

\renewcommand{\O}{\mathcal{B}}                                   
\newcommand{\ObsOperCont}{\O} 

\newcommand{\ObsOperind}[1]{\ObsOperCont_{#1}}

\newcommand{\y}{\mathbf{y}}                        
\newcommand{\obs}{\y}                              

\newcommand{\param}{\vec{\theta}}                  
\newcommand{\iparam}{\param}                       
\newcommand{\iparprior}{\iparam_{\rm pr}}          
\newcommand{\iparb}{\iparprior}                    
\newcommand{\iparpost}{\iparam_{\rm post}^\obs}    
\newcommand{\ipara}{\iparpost}                     

\newcommand{\ipartrue}{\iparam_{\rm true}}
\newcommand{\Nstate}{\textsc{N}_{\rm state}}                    
\newcommand{\Nparam}{{N_\mathrm{\theta}}}
\newcommand{\Nobs}{\textsc{N}_{\rm obs}}                        
\newcommand{\Npred}{\textsc{N}_{\rm goal}}                      
\newcommand{\nobs}{n_{t}}                                  
\newcommand{\nobstimes}{\nobs}                             
\newcommand{\Nsens}{n_{\rm s}}                         
\newcommand{\budget}{\lambda}                         
\newcommand{\Cparamprior}{\mat\Gamma_{{\rm pr}}}                
\newcommand{\Cparampost}{\mat\Gamma_{{\rm post}}}               
\newcommand{\Cobsnoise}{\mat{\Gamma}_{{\rm noise}}}            
\newcommand{\Cobsnoiseinv}{\mat{\Gamma}\inv_{{\rm noise}}}            
\newcommand{\Cpred}{\mat{\Sigma}}                               
\newcommand{\Cpredprior}{\Cpred_{{\rm pr}}}                    
\newcommand{\Cpredpost}{\Cpred_{\rm post}}                     
\newcommand{\Cparampriormat}{\Cparamprior}                      
\newcommand{\Cparampostmat}{\Cparampost}                        

\newcommand{\Cobsnoisemat}{\mat{R}}                             

\newcommand{\Cpredpostmat}{\Cpredpost}                          
\newcommand{\wCpredpostmat}{\Cpredpostmat(\design)}             
\newcommand{\Fcont}{\mathcal{F}}                                 
\newcommand{\FF}{\Fcont}                                 
\newcommand{\F}{\mathbf{F}}                                      
\newcommand{\Ftind}[2]{\F_{t_{#1}\rightarrow t_{#2}}}             
\newcommand{\Find}[2]{\F_{#1,#2}}
\newcommand{\Fadj}{\F\adj}
\newcommand{\Fadjind}[2]{\F\adj_{#2,#1}}
\newcommand{\Sol}{\mathcal{S}}  
\newcommand{\Soltind}[2]{\Sol_{t_{#1}\rightarrow t_{#2}}}           %
\newcommand{\Solind}[2]{\Sol_{#1,#2}}

\newcommand{\Hessmat}{\mat{H}}                                   

\newcommand{\Predmat}{\mat{P}}                                   

\newcommand{\T}{\mathcal{T}}                                     
\newcommand{\obj}{\Psi}

\newcommand{\GA}{\mathrm{GA}}
\newcommand{\GD}{\mathrm{GD}}

\newcommand{\GM}[2]{\mathcal{N}\!\left( {#1}, {#2}\right)}       

\newcommand{\like}{\mathcal{L}}                          
\newcommand{\Like}[2]{\like{\left(#1|#2\right)}}                                  

\newcommand{\Expect}[2]{\mathbb{E}_{#1}{\left[ #2 \right]} }     
\newcommand{\D}{\mathcal{D}}

\newcommand{\design}{\vec{\zeta}}                                       
\newcommand{\binarydesign}{\design^{\rm b}}        
\newcommand{\weightfunc}{\omega}                                    
\newcommand{\varweightfunc}{\varpi}    
\newcommand{\designmat}{\mat{W}}                                    
\newcommand{\wdesignmat}{\designmat_{\Gamma}}                       

\newcommand{\pred}{\vec{\gamma}}      

\newcommand{\ipred}{\pred}
\newcommand{\predprior}{\ipred_{\mathrm{pr}}}
\newcommand{\ipredprior}{\predprior}
\newcommand{\ipredpost}{\ipred_{\mathrm{post}}}

\newcommand{\wHessmat}{\Hessmat(\design)}                           
\newcommand{\directsum}[3]{\bigoplus\limits_{#1}^{#2}\!{\left(#3\right)}}

\stackMath
\newcommand\reallywidehat[1]{%
\savestack{\tmpbox}{\stretchto{%
  \scaleto{%
    \scalerel*[\widthof{\ensuremath{#1}}]{\kern-.6pt\bigwedge\kern-.6pt}%
    {\rule[-\textheight/2]{1ex}{\textheight}}
  }{\textheight}%
}{0.5ex}}%
\stackon[1pt]{#1}{\tmpbox}%
}
\newcommand{\proj}{\mat{L}}     
     
\newcommand{\decorrcoeff}{\rho}                     

\newcommand{\floor}[1]{\lfloor #1 \rfloor}
\newcommand{\Floor}[1]{\Bigl\lfloor #1 \Bigr\rfloor}

\newcommand{\commentout}[1]{\iffalse {#1} \fi}

\ifpdf
  \DeclareGraphicsExtensions{.eps,.pdf,.png,.jpg}
\else
  \DeclareGraphicsExtensions{.eps}
\fi

\newsiamremark{remark}{Remark}
\newsiamremark{hypothesis}{Hypothesis}
\crefname{hypothesis}{Hypothesis}{Hypotheses}
\newsiamthm{claim}{Claim}
\crefname{lemma}{Lemma}{Lemmas}

\headers{Bayesian OED with Correlated Observations}{A. Attia and E. Constantinescu}

\title{Optimal Experimental Design for Inverse Problems in the Presence of Observation Correlations\thanks{Submitted to the editors \today.
\funding{This work was partially supported by the U.S. Department of Energy, Office of Science, Advanced Scientific Computing Research Program under contract DE-AC02-06CH11357 and Laboratory Directed Research and Development (LDRD) funding from Argonne National Laboratory.}}}

\author{Ahmed Attia\thanks{Mathematics and Computer Science Division,
                   Argonne National Laboratory, Lemont, IL
                  ( \email{attia@mcs.anl.gov} and \email{emconsta@mcs.anl.gov}).}
\and Emil Constantinescu\footnotemark[2]~\thanks{The University of Chicago, IL.}}

\ifpdf
\hypersetup{
  pdftitle={Bayesian OED with Correlated Observations},
  pdfauthor={A. Attia and E. Constantinescu}
}
\fi

\begin{document}

\maketitle

\begin{abstract}
Optimal experimental design (OED) is the general formalism of sensor
placement and decisions about the data collection strategy for
engineered or natural experiments. This approach is prevalent in many
critical fields such as battery design, numerical weather prediction,
geosciences, and environmental and urban studies. State-of-the-art
computational methods for experimental design, however, do not
accommodate correlation structure in observational errors produced by
many expensive-to-operate devices such as X-ray machines or radar and
satellite retrievals. Discarding evident data correlations leads to
biased results, poor data collection decisions, and waste of valuable
resources. We present a general formulation of the OED formalism for
model-constrained large-scale Bayesian linear inverse problems, where
measurement errors are generally correlated. The proposed approach
utilizes the Hadamard product of matrices to formulate the weighted
likelihood and is valid for both finite- and infinite-dimensional
Bayesian inverse problems. 
We also discuss widely used approaches for relaxation of the binary OED
problem, in
light of the proposed pointwise weighting approach,
and present a clear interpretation of the relaxed design and its effect on the
observational error covariance.
Extensive numerical experiments are carried
out for empirical verification of the proposed approach by using an
advection-diffusion model, where the objective is to optimally place a
small set of sensors, under a limited budget, to predict the
concentration of a contaminant in a bounded
domain.
\end{abstract}

\begin{keywords}
  Optimal experimental design (OED),
    Inverse problems,
    Correlated observations,
    Data assimilation
\end{keywords}

\begin{AMS}
  62K05, 35Q62, 62F15, 35R30, 35Q93, 65C60
\end{AMS}

\section{Introduction}\label{sec:introduction}
  Sensor placement is the problem of determining the optimal positions of a given number of sensors from a set 
  of candidate locations. 
  This problem is widely formulated as an optimal experimental design (OED) 
  problem~\cite{Pukelsheim2006optimal,FedorovLee00,Pazman86,Ucinski00}, where the design 
  determines whether to activate a sensor or not. 
  Bayesian OED for data acquisition and sensor placement has been addressed in
  the context of ill-posed linear inverse problems 
  \cite{HaberHoreshTenorio08,HaberMagnantLuceroEtAl12} and nonlinear
  inverse problems~\cite{HaberHoreshTenorio10,HuanMarzouk13} and has also been
  considered in infinite-dimensional settings~\cite{bui2013computational,alexanderian2016bayesian,AlexanderianSaibaba17,AlexanderianPetraStadlerEtAl16,AlexanderianPetraStadlerEtAl14,alexanderian2021optimal:review}.
  In our work and the works cited above, linearity is associated with the parameter-to-observable map, 
  which describes the relation between the quantity of interest (QoI) and the
  observational data.

  Scalable data assimilation and uncertainty quantification
  methodologies have gained  significant interest, especially when
  applied to  large- to extreme-scale models, such as the atmosphere and power
  grid simulations~\cite{bannister2017review,daley1993atmospheric,navon2009data,attia2016reducedhmcsmoother,attia2015hmcfilter}.
  The underlying principle of these methods is that information collected from 
  observational systems is fused into computational models to produce accurate 
  forecasts of a QoI.
  This QoI could, for example, refer to the model parameters, initial condition, 
  or the model state.
  In Bayesian inversion, the goal is to infer the value of the QoI through
  the posterior distribution of that QoI modeled as a random variable, conditioned by the
  noisy observational data.
  The quality of such assimilation systems depends heavily on the extent to which 
  the mathematical assumptions reflect reality, as well as on the quality of the collected 
  measurements, which in turn depends on the data acquisition schedule and the sensor
  placement strategy.

  An optimal design improves the quality of the observational grid and
  the collected data and hence improves the model-based forecast  made by an
  assimilation system.  Deploying and operating observational instruments 
  can be  expensive, however, thus mandating the development of optimal sensor placement
  strategies and optimal data acquisition schedules. 
  An optimal design is generally found by solving a binary optimization problem,
  or a relaxation thereof, with the 
  objective of minimizing the uncertainty of the posterior QoI under 
  the constraints forced by the computational model.
  The objective function of such an optimization problem is formulated based on the 
  choice of the optimality criterion, which is a scalar summary of the uncertainty
  of the Bayesian inversion QoI.
  For example, an A-optimal design minimizes the average variance, 
  and a D-optimal design minimizes the generalized variance, which is equivalent to 
  maximizing the differential Shannon information  content of the inversion 
  QoI~\cite{Pukelsheim2006optimal}.
  
  The presence of observation spatiotemporal error correlations in OED
  for inverse problems poses a significant challenge~\cite{ucinski2020d}.
  A common approach, which is followed for simplicity in solving model-constrained OED problems, 
  is to drop spatiotemporal correlations and assume that observation errors are temporally and 
  spatially uncorrelated.
  This approach can be restrictive, however, and is even violated 
  in many applications where a single instrument is used to collect 
  measurements at various orientations, such as X-ray machines, satellites, 
  and light detection and ranging radars. 

  The statistics literature provides a wealth of treatments of correlated
  observations in several OED scenarios and
  settings~\cite{cressie2015statistics,muller2007collecting,nather1985effective,pilz1991bayesian,ChalonerVerdinelli95}.
  For example, the ordinary least-squares approach for inversion and the D-optimality criterion is adopted
  in~\cite{ucinski2020d} to avoid inaccuracies or lack of knowledge of the
  observation correlation structure.
  Additionally, OED for correlated observations has been considered in the context of linear
  regression models, for example in~\cite{dette2013optimal,ucinski2004experimental}. 
  Nevertheless, in the context of Bayesian inverse problems,
  more work is needed to provide simple algorithmic OEDs that are  suitable for
  multiple choices of the optimality criterion and are capable of
  properly handling correlations in observational errors. 

  In this work we present a generalized OED formulation based on
  the Hadamard product of the observation error covariance matrix with a symmetric
  weighting kernel, where the observation errors are generally correlated.
  In this case the covariances of the observation errors, manifested in the 
  weighted data likelihood, could form a diagonal, a block diagonal, or 
  even a full space-time covariance matrix.
  This approach provides a generalized formulation of the OED optimality
  criterion and the associated gradients, where it takes into account the
  correlation structure of observational errors.
  Moreover, it enables controlling the influence of observational correlations 
  on the experimental design, for example, by ignoring the effect of spurious 
  correlations, using a space-time weighting kernel. 
  In this  paper
  we extensively analyze the effect of the binary design and the relaxation on the
  observational error variances/covariances and on the OED objective function. 
  We also discuss the limitation and potential inaccuracy of existing weighting
  approaches, which have been  used in the case of uncorrelated observational errors.
  The new formulation resolves a common misunderstanding about the effect of the
  relaxation on the OED optimization objective and on the corresponding optimal design. 
  Moreover, this formulation adds a layer of flexibility to the standard 
  OED formulation in the context of Bayesian inversion and is suitable 
  for both finite- and infinite-dimensional Bayesian linear inversion frameworks.
  Additionally, the presented approach enables converting the OED constrained optimization 
  problem into an unconstrained optimization problem, which enables utilizing 
  a plethora of efficient numerical optimization routines.

  The rest of the paper is structured as follows.
  \Cref{sec:background} provides the mathematical background and
  describes the standard formulation of optimal design of experiments for
  Bayesian linear inverse problems.
  The proposed approach for handling errors in observation correlations is
  described in~\Cref{sec:OED_Correlated_OBS}, with detailed derivation of
  formulae given in the appendices included in the supplementary
  material. 
  The setup and results of numerical experiments are given in
  \Cref{sec:numerical_experiments}.
  Concluding remarks are presented in~\Cref{sec:Conclusions}.

\section{Background}
  \label{sec:background}
Here we review the elements of Bayesian inversion and OED 
 for Bayesian inverse problems governed by computational models,
such as partial differential equations (PDEs).

\subsection{Bayesian inverse problem}
  \label{subsec:Bayesian_inversion}
  A forward problem describes the 
  relationship between model parameters and 
  observational data. The vast majority of large-scale simulation models follow the 
  forward model described by
  $ 
    \obs  \!=\! \FF(\iparam) \!+\!\vec{\delta},
  $
  where $\iparam \!\in\! \Rnum^{\Nstate}$ is the model parameter and 
  $\obs \!\in \!\Rnum^{\Nobs}$ is the observational data. 
  The forward map $\FF$ is a discretized parameter-to-observable operator that 
  maps a model parameter into the observation space,
  and $\vec{\delta} \!\in\! \Rnum^{\Nobs}$ accounts for measurement noise. 
  We restrict the discussion in this work to the case of linear models, that is,
  $\FF(\iparam)\!=\!\F\iparam$: 
  \begin{equation}\label{eqn:forward_problem}
    \obs  = \F\iparam + \vec{\delta} \,.
  \end{equation}
  The parameter $\iparam$ is generally modeled as a random variable with a Gaussian 
  prior $\GM{\iparb}{\Cparampriormat}$.
  In the Gaussian framework, the observational noise is assumed to be Gaussian, 
  that is, $\vec{\delta} \sim \GM{\vec{0}}{\Cobsnoise}$, 
  where $\Cobsnoise$ is the observation error covariance matrix.
  In this case, the data likelihood is 
  \begin{equation} \label{eqn:Gaussian_likelihood}
    \Like{\obs}{ \iparam } \propto 
      \exp{\left(-\frac{1}{2} \sqwnorm{\F\iparam-\obs}{\Cobsnoise\inv }\right) } \,,
  \end{equation}
  where the weighted norm is defined as 
  $\sqwnorm{\vec{x}}{\mat{A}} = \vec{x}\tran \mat{A} \vec{x} $.

  The inverse problem involves retrieving the underlying model parameter $\iparam$ from
  noisy measurements $\obs$, given the specification of prior and observation noise uncertainties.
  Specifically, in Bayesian inversion we seek the posterior distribution of $\iparam$ conditioned by
  observational data $\obs$.
  For a linear forward operator $\F$, the posterior is Gaussian $\GM{\ipara}{\Cparampostmat}$ with
  \begin{equation}\label{eqn:Posterior_Params}
    \ipara         = \Cparampostmat \left( \Cparampriormat\inv \iparb + \Fadj \Cobsnoise\inv\, \obs \right)  \,, \quad 
   \Cparampostmat = \left( \Fadj \Cobsnoise\inv \F  + \Cparampriormat\inv
    \right)\inv = \Hessmat\inv \,,
  \end{equation}
  where $\Fadj$ is the adjoint of the forward operator $\F$.
  Here the posterior mean $\ipara$ provides a posterior estimate of the true
  value of the parameter $\ipartrue$, given the prior mean $\iparprior$, the
  data $\obs$, and the associated uncertainties.

  We note that the Hessian $\Hessmat$ of the negative log of the posterior
  probability density function (PDF) is equal to the posterior 
  precision matrix $\Cparampostmat\inv$ and is independent from the data.
  Thus, in the linear Gaussian case one can completely describe the posterior covariances,
  given the forward operator and both prior and observation noise covariances.
  Unlike the linear case, however, if the forward operator $\FF$ is
  nonlinear, the posterior covariance depends on the observational
  data, and the OED problem becomes more difficult and is beyond the scope of
  this work.
  
\subsection{Goal-oriented Bayesian inverse problem}
\label{subsec:GoalOriented_Inversion}
  Solving an inverse problem is often an intermediate step, where the inversion
  parameter, such as the inferred model parameter, is then used to make further
  prediction of a goal QoI. 
  We consider the QoI of the form $\pred=\Predmat
  \iparam\in\Rnum^{\Npred}$, where $\Predmat$ is 
  a linear \textit{goal operator}~\cite{attia2018goal}. 
  A simple example of $\Predmat$ is an integral operator that evaluates the 
  expected value of the inferred parameter.
  We follow this general formulation hereafter since the standard 
  Bayesian inverse problem is a special case, 
  where $\Predmat$ is set to the identity operator $\mathcal{I}$.
  Given the distribution of the model parameter, the prior of $\pred$ is  
  Gaussian $\GM{\ipredprior}{\Cpredprior}$, with $\ipredprior = \Predmat \iparb$, 
  and  $\Cpredprior = \Predmat \Cparamprior \Predmat\adj$.
  The posterior of $\pred$ is also Gaussian $\GM{\ipredpost}{\Cpredpost}$, with
  \begin{equation}\label{eqn:pred_post}
    \ipredpost 
      = \Predmat \Cparampostmat \left( \Cparampriormat\inv \iparb 
        + \Fadj \Cobsnoise\inv\, \obs \right) 
      \,,\quad
      \Cpredpost 
      = \Predmat \left( \Fadj \Cobsnoise\inv \F  
        + \Cparampriormat\inv \right)\inv \Predmat\adj
    \,.
  \end{equation}

  The exact form of the adjoint of the forward operator $\Fadj$ and that of the adjoint 
  of the goal operator $\Predmat\adj$  depend on the problem at
  hand. Specific instances will be discussed in~\Cref{sec:numerical_experiments} in the context of numerical 
  experiments.

\subsection{Bayesian OED}\label{subsec:OED}
  In Bayesian OED, we seek a design for the data acquisition process.
  The specific definition of an experimental design, in general, 
  is problem dependent. 
  For example, an experimental design $\design\in\Xi$ can be associated with 
  the configuration of an observational grid; that is, the experimental design 
  describes the sensor placement strategy. 
  In this case an experimental design is said to be optimal 
  if it minimizes the posterior uncertainty of the solution of an inverse problem in some sense.
  The optimal design $\design^{\rm opt}$ in the context of an inverse
  problem is defined by an \textit{optimality criterion}, which in general
  is a scalar functional $\Psi(\cdot)$ that depends on the posterior uncertainty
  of the inversion QoI.
  In the linear Gaussian settings, the optimality criterion is generally set as 
  a scalar summary of the posterior covariance matrix.
  The traditional alphabetic criteria including A- and D-optimality are the most 
  popular choices \cite{alexanderian2021optimal:review}.
  An A-optimal design $\design^{\rm A-opt}$ minimizes the trace of the posterior 
  covariance, while a D-optimal design $\design^{\rm D-opt}$  
  minimizes its determinant (or equivalently the log-determinant): 
  \begin{equation}\label{eqn:A_and_D_optimality}
    \design^{\rm A\!-\!opt}\! = \argmin_{\design\in\Xi} \Psi^\GA(\design) :=\! \Trace{\Cpredpost(\design)}  \,; \quad
    \design^{\rm D\!-\!opt}\! = \argmin_{\design\in\Xi} \Psi^\GD(\design) :=\! \logdet{\Cpredpost(\design)} \,.
  \end{equation}

  These optimality criteria are typically augmented by regularization and 
  sparsification terms that are discussed in later sections.
  Note that the optimization problems~\eqref{eqn:A_and_D_optimality} can be rewritten
  as maximization by replacing the posterior covariance with the Fisher
  information matrix.

\subsection{Bayesian OED for sensor placement}
  \label{subsec:OED_sensor_placement}
  In the sensor placement problems we seek the \textit{optimal} 
  subset of sensors of size $\budget$ from $\Nsens$ candidate locations.
  In this case the design $\design=\binarydesign$ is a vector of binary entries, 
  whose components can be interpreted as sensors being \textit{active} or
  \textit{inactive}; that is, $\binarydesign\in \Xi := \{0,1\}^{\Nsens}$.
  When a sensor is deactivated, the row/column in the observation
  error covariance matrix corresponding to that observation is eliminated from the formulation of the Bayesian
  inverse problem and is kept when the sensor is active.
  This can be formulated as a binary optimization problem over the binary design
  space.
  Solving a binary OED problem for sensor placement, however, is computationally infeasible for 
  large-scale problems. 
  In practice, the design weights, that is, the integrality of entries $\design_i; i=1,2, \ldots,\Nsens$, 
  are \textit{relaxed} to take values in the interval $[0, 1]$; 
  that is, $\design\in \Xi := [0,1]^{\Nsens}$. 
  Then a sparsification is enforced by adding a suitable regularization term to 
  the optimality objective~\eqref{eqn:A_and_D_optimality}. 

  The design enters the inverse problem formulation through the 
  data likelihood as a set of observation weights. 
  Specifically, the observation covariance $\Cobsnoise$ is replaced with a weighted version 
  $\wdesignmat(\design)$, resulting in the weighted data likelihood
  \begin{equation}\label{eqn:weighted_joint_likelihood}
    \Like{\obs}{ \iparam; \design}
      \propto \exp{\left( - \frac{1}{2} \sqwnorm{ \F \iparam - \obs}{\wdesignmat(\design)} \right) } \,.   
  \end{equation}

  The weighted posterior covariance of the inferred QoI in this case is
  \begin{equation}\label{eqn:weighed_post_cov}
    \wCpredpostmat = \Predmat \Hessmat\inv(\design) \Predmat\adj 
      = \Predmat \left( \Fadj \wdesignmat(\design) \F  
        + \Cparampriormat\inv \right)\inv \Predmat\adj \,.
  \end{equation}

  A common approach to define the weighted inverse covariance matrix
  $\wdesignmat(\design)$, in the case of uncorrelated observation 
  errors, is to introduce  a weighting matrix $\designmat(\design)$ and use the form 
  $\wdesignmat(\design) := \designmat^{\frac{1}{2}}(\design) \Cobsnoise\inv
  \designmat^{\frac{1}{2}}(\design) $. 
  The design matrix $\designmat \in \Rnum^{\Nobs \times \Nobs}$ is a symmetric weighting 
  matrix parameterized by the design $\design$.
  In sensor placement problems, the design is a vector containing weights assigned to 
  each candidate sensor location 
  $\design=(\design_1, \design_2, \ldots, \design_{\Nsens})\tran$,
  and the weighting matrix is a diagonal matrix with design weights on the diagonal; 
  that is, $\designmat = \Diag{\design}$.
  An alternative approach is to use the form
  $\wdesignmat := \Cobsnoise^{-\frac{1}{2}} \designmat \Cobsnoise^{-\frac{1}{2}}$. 
  Little attention is given in the literature to the difference between these two forms,
  mainly because in many applications the observations are assumed to be uncorrelated, 
  resulting in a diagonal covariance matrix $\Cobsnoise$.
  In this case these two forms are equivalent. 
  In general, however, they are not.
  The latter form simplifies the derivation of the gradient of the optimality criterion
  with respect to the design; however, the interpretability of its effect is not intuitive.
  On the other hand, the former form weighs the sensor observations, based on their 
  contributed information gain, resulting in a reduction in the QoI posterior uncertainty. 
  Further discussion of the validity of these forms of the weighted precision
  matrix $\wdesignmat(\design)$ is in~\Cref{subsec:OED_weighting}.
 
  In time-dependent settings, the dynamical model is simulated over a window
  containing observation time instances $t_1, t_2, \ldots t_{\nobstimes}$.
  For simplicity we assume the matrix representation of the forward map $\F$ 
  has rows that form consecutive blocks, each corresponding to all spatial measurements at one time moment. 
  Assuming that the design weights associated with candidate observational gridpoints are 
  fixed over time, the weighting matrix $\designmat \in \Rnum^{\Nobs \times \Nobs}$ 
  can be defined as $\designmat = \mathbf{I}_{\nobstimes} \otimes \Diag{\design}$ with $\mathbf{I}_{\nobstimes} \in \Rnum^{\nobstimes\times\nobstimes}$ 
  an identity matrix and $\Nobs = \Nsens \nobstimes$.
  Here $\otimes$ is the matrix Kronecker product.
  If the observational noise is temporally uncorrelated, 
  then the covariance $\Cobsnoise$ is generally 
  an $\Nobs \times \Nobs$ block diagonal
  matrix $\Cobsnoise = \directsum{m=1}{\nobstimes}{\mat{R}_m}$, where
  $\bigoplus$ is the matrix direct sum and  
  $\mat{R}_m \in \Rnum^{\Nsens \times \Nsens}$ models the spatial 
  covariances between observation errors prescribed at time instance $t_m$.  
  Note that the matrix direct sum is equivalent to the matrix Kronecker product with 
  an identity matrix only if the entries of the direct sum are identical.

  If the observation correlations are time independent, namely, 
  $\mat{R}_m=\mat{R},\, \forall m$, then
  $\Cobsnoise = \mat{I}_{\nobstimes}\otimes \mat{R}$. 
  In the presence of spatiotemporal correlations, the covariance matrix $\Cobsnoise$
  becomes a dense symmetric block matrix with the $mn$th block $\mat{R}_{mn}$
  describing covariances between observation gridpoints at time instances $t_m$ 
  and $t_n$, respectively.

  The standard Bayesian OED formulation, discussed above, is ideal for spatially and 
  temporally uncorrelated observation errors. 
  However, it may not properly account for spatiotemporal correlations, 
  and it is harder to apply when the design is allowed to vary over time.
  Specifically, the role of the relaxed design matrix is to weight covariances
  between candidate sensor locations---a  maximum weight of $1$ means activating
  the corresponding sensors, and a minimum weight of $0$ means deactivating
  them---based on their contribution to the OED optimality objective.
  Thus, the weighting  matrix $\designmat$ needs to be applied to
  the observation error covariance matrix, and not to the precision
  matrix~\cite{liu2016sensor}, which is initially obtained based on the assumption that all
  sensors are activated.
  Moreover, the formulation does not provide enough flexibility for handling the effect of 
  spurious observation correlations that might be introduced, for example, by
  misspecification of the observational error covariances.
  Furthermore, it does not provide any control over the values of design, besides 
  the imposed bound constraints. For example, no inherent property
  enables imposing preference of specific values of the design variables
  such as observation cost constraints.
  In~\Cref{sec:OED_Correlated_OBS} we further discuss the issue of relaxation
  and introduce a generalized formulation of the OED problem capable of handling these limitations.
  %

\subsection{The OED optimization problem}
\label{subsec:OED_optimization}
  The relaxed OED problem for sensor placement is described by the 
  following constrained optimization problem,
  \begin{equation}\label{eqn:oed_relaxed_optimization_problem}
    \design^{\rm opt} 
      =\argmin_{\design \in [0,\,1]^{\Nsens}}{
        \T(\design):=\Psi (\design) + \alpha \, \Phi(\design) 
      }  \,,
  \end{equation}
  where $\Psi$ is the design criterion,  
  $\Phi(\design)\!:\![0,\,1]^{\Nsens}\! \mapsto\! [0, \infty)$ 
  is a penalty function that enforces 
  regularization or sparsity on the design, and $\alpha \!>\! 0$
  is a user-defined penalty parameter.
  The optimality criterion $\Psi$ is set to $\Psi^\GA$ or $\Psi^\GD$, defined by
  ~\eqref{eqn:A_and_D_optimality}, for obtaining an A- or D-optimal design, respectively.

  A gradient-based optimization approach is followed to numerically 
  solve~\eqref{eqn:oed_relaxed_optimization_problem}. 
  A central piece of this procedure is the gradient of the optimality criterion $\Psi$.
  For $i=1, 2, \ldots, \Nsens$, the derivatives of the  A- and D-optimality 
  criteria~\eqref{eqn:A_and_D_optimality}, respectively, are
  \begin{subequations}\label{eqn:general_gradient_entries}
  \begin{align}
    \del{\Psi^\GA}{\design_i} &  
      = -\,\Trace{\Predmat \Hessmat\inv(\design) \Fadj\del{\wdesignmat(\design)}{\design_i}\F \Hessmat\inv(\design) \Predmat\adj}  \,,  \\
    \del{\Psi^\GD}{\design_i} & 
      = \Trace{\Cpredpostmat\inv(\design) 
          \Predmat \Hessmat\inv(\design) \Fadj\del{\wdesignmat(\design)}{\design_i}\F \Hessmat\inv(\design) \Predmat\adj
        } \,,
  \end{align}
  \end{subequations}
  where $\wHessmat$ is the weighted Hessian of the negative log of the posterior
  PDF,
  $\wHessmat:=\Cparampriormat\inv + \Fadj \wdesignmat(\design) \F $.
  The gradient of the objective~\eqref{eqn:oed_relaxed_optimization_problem} is obtained by combining the gradient 
  of the penalty term $\nabla_{\design}\Phi(\design)$ with the optimality criterion~\eqref{eqn:general_gradient_entries},
  that is, $\nabla_{\design}\T(\design)=\nabla_{\design}\Psi (\design) + \alpha \, \nabla_{\design}\Phi(\design)$, 
  and is provided to the numerical optimization routine.
  
  The general form of the derivative~\eqref{eqn:general_gradient_entries} shows that, 
  to formulate the gradient, we need the derivative of the weighted observation
  precision matrix $\wdesignmat(\design)$.
  The final form of the gradient depends on the specific form of the weighted
  observation precision matrix $\wdesignmat(\design)$.
  Most studies assume uncorrelated observations and derive the gradient accordingly for 
  simplicity.
  In the next section we provide a generalized formulation of the weighted noise matrix
  and the OED problem, with emphasis on the effect of the design on the observations and 
  observation spatiotemporal correlations.
  
  Deploying observational sensors can be expensive, and in general we seek a small number of sensors
  to accurately solve the inverse problem with minimum uncertainty levels. 
  To achieve this goal, we choose the penalty $\Phi$ function to enforce sparsity on the optimal design.
  Sparsification can be achieved by using the $\ell_0$ ``norm''~\cite{AlexanderianPetraStadlerEtAl14};
  however, $\ell_0$ is not a valid norm and is nondifferentiable,
  which leads to complications 
  in the optimization procedure.
  An acceptable level of sparsification can be achieved by utilizing the $\ell_1$
  norm in the regularization term~\cite{alexanderian2016bayesian}.

  The solution of~\eqref{eqn:oed_relaxed_optimization_problem} is expected to be 
  a sparse design; however, it is not necessarily binary. 
  A binary design can be obtained from the solution of the relaxed optimization 
  problem~\eqref{eqn:oed_relaxed_optimization_problem}, for example, by 
  applying thresholding~\cite{alexanderian2016bayesian,attia2018goal} 
  or by applying a continuation procedure or by reweighting the regularization 
  term~\cite{herman2019randomization,koval2020optimal}.
  Another approach to ensure a binary design is to partition the domain and use 
  the sum-up rounding procedure; see, for example,
  \cite{sager2013sampling,YuZavalaAnitescu17}.
  The simplest approach is to truncate the relaxed optimal design. This can be
  done by activating sensors corresponding to the nonvanishing weight (e.g.,
  greater than a small-enough value,) resulting
  from solving the relaxed OED problem. 
  Alternatively, given a specific budget
  $\budget$, one can activate the sensors corresponding to the highest
  $\budget$ nonzero weights in the relaxed optimal design.

  Following the approach in~\cite{alexanderian2016bayesian,attia2018goal},
  in the numerical experiments  reported in \Cref{sec:numerical_experiments} we promote 
  sparsification by setting $\Phi$  
  to the $\ell_1$ norm.
  Once a sparse optimal design is obtained, we choose the sensors
  corresponding to the highest $\budget$ 
  optimal weights.
  Note that the discussion here is not limited by the choice of $\Phi$ 
  and can be used with other sparsification methods.

\section{OED for Correlated Observations}
\label{sec:OED_Correlated_OBS}
  With observational noise being uncorrelated in space and in time, 
  the mathematical analysis of the traditional formulation of the PDE-constrained OED problem 
  is simple; however, this places limitations on the applicability of the 
  formulated framework.
  In this section we address this issue by formulating the OED problem where the design 
  choice is sensitive  not only to the variances of the involved observational errors 
  but also to the spatiotemporal correlations.
  Specifically, in this section we introduce a generalization of the traditional OED framework, 
  following a Hadamard product approach~\cite{Horn_1990_Hadamard} 
  for likelihood weighting. 
  This will add another dimension of flexibility that will help us formulate 
  a general OED framework capable of handling observation noise 
  with spatiotemporal correlation.
  We start with a discussion of the effect of design relaxation
  in~\Cref{subsec:OED_weighting}, followed by the proposed
  approach starting from~\Cref{subsec:Schur_OED}.
  
  \subsection{Insight into the relaxation of binary OED problem}
  \label{subsec:OED_weighting}
    As discussed in~\Cref{subsec:OED_sensor_placement}, one can 
    define the weighted precision matrix by weighting the precision matrix
    $\Cobsnoiseinv$ by the design weights.

    \noindent\paragraph{Precision matrix weighting}
    In this case we form the weighted precision matrix as 
    $\wdesignmat(\design):=\designmat^{\frac{1}{2}}(\design) \Cobsnoiseinv
      \designmat^{\frac{1}{2}}$ 
    and then the optimization
    problem~\eqref{eqn:oed_relaxed_optimization_problem}. To utilize the
    derivative~\eqref{eqn:general_gradient_entries} in a gradient-based
    optimization approach for
    solving~\eqref{eqn:oed_relaxed_optimization_problem}, one needs to
    calculate  $\del{\wdesignmat{(\design)}}{\design_i}$. 
    However, an issue is that in general the derivative would not be defined at
    $\design_i=0$. 
    A possible remedy would be to redefine 
    the weighted precision matrix by pre- and postmultiplication with the
    relaxed design:
    \begin{equation}\label{eqn:standard_prepost_fixed}
      \wdesignmat(\design):=\designmat(\design) \Cobsnoiseinv \designmat(\design).
    \end{equation}
    In this case one can show that 
    \begin{subequations}\label{eqn:general_gradient_entries_prepost}
    \begin{equation}\label{eqn:standard_deriv_vec}
      \del{\wdesignmat(\design)}{\design_i} 
        = \vec{v}_i(\design) \vec{e}_i\tran + \vec{e}_i \vec{v}_i\tran; \qquad 
        \vec{v}_i(\design) = \left(\Cobsnoiseinv \vec{e}_i\right) \odot\design\,; i=1,
        \ldots,\Nsens \,,
    \end{equation}
    where $\odot$ is the pointwise product and $\vec{e}_i$ is the $i$th
    versor of $\Rnum^{\Nsens}$.
    By substituting~\eqref{eqn:standard_deriv_vec}
    in~\eqref{eqn:general_gradient_entries}, the derivatives (for A- and
    D-optimality) are
      \begin{align}
        \del{\Psi^\GA}{\design_i} 
          &= -2 
          \vec{e}_i\tran
          \F \Hessmat\inv(\design) \Predmat\adj 
          \Predmat \Hessmat\inv(\design) \Fadj
          \vec{v}_i(\design)
          \,,  \\
        \del{\Psi^\GD}{\design_i} 
          &= -2 
          \vec{e}_i\tran
          \F \Hessmat\inv(\design) \Predmat\adj
            \Cpredpostmat\inv(\design) 
          \Predmat \Hessmat\inv(\design) \Fadj
          \vec{v}_i(\design)
          \,,
      \end{align}
    \end{subequations}
    where we used the circular and transposition properties of the matrix trace.
    This can be thought of as an \textit{ad hoc} fix of the standard OED formulation
    where $\designmat$ pre- and postmultiply the precision matrix
    (instead of using $\designmat^{\frac{1}{2}}$).
    This raises another problem, however: Even with the ad hoc fix~\eqref{eqn:standard_prepost_fixed}, this
    formulation yields a relaxed solution that is not
    guaranteed to match the solution of the original binary OED problem.
    The reason is that the weighting is carried out after
    the precision matrix $\Cobsnoiseinv$ is formulated  assuming all sensors are
    active (even if a design variable is set to $0$), and thus inactive sensors 
    contribute to the elements in the precision (inverse covariance) matrix corresponding to active sensors. 
    More formally, it is not guaranteed---unless $\Cobsnoise$ is diagonal---that 
    $\lim_{\design\rightarrow\binarydesign}{\wdesignmat(\design)}$ is equal to
    the degenerate precision matrix $\Cobsnoiseinv(\binarydesign)$ obtained by 
    eliminating rows/columns from $\Cobsnoise$ corresponding to zero entries of $\binarydesign$,
    where $\binarydesign\in\{0,1\}^{\Nsens}$. In other words, 
    the matrix $\wdesignmat(\design)$ may not have the correct limit as
    $\design$ approaches points on the boundary of the relaxed domain $[0,
    1]^{\Nsens}$. This issue will be discussed further below.

    \noindent\paragraph{Effect of relaxing the design}
    To properly formulate and solve sensor placement problems, we start by
    discussing the roles played by the binary design and the associated relaxation.
    The original OED binary optimization problem for sensor placement takes the
    form
    \begin{equation}\label{eqn:binary_OED_optimization}
      \min_{\binarydesign \in \{0, 1\}^{\Nsens}} \obj(\binarydesign) \,,
    \end{equation}
    where we dropped the regularization term for simplicity.
    Here, $\obj$ depends on $\binarydesign$ through the weighted precision 
    matrix $\wdesignmat$. 
    The binary design $\binarydesign:=(\binarydesign_1, \ldots,
    \binarydesign_{\Nsens})\tran$ characterizes the active sensors and thus
    defines which entries of the observation vector to keep and which to
    remove.
    This corresponds to keeping/removing rows and columns from the observation
    error covariance matrix $\Cobsnoise$.
    As described by~\eqref{eqn:weighted_joint_likelihood} and~\eqref{eqn:weighed_post_cov}, 
    this is encoded in the inverse problem by modifying the likelihood function.
    Since we generally assume observation errors are Gaussian, that is,
    $\vec{\delta} := \obs - \F (\iparam) \sim \GM{\vec{0}}{\Cobsnoise}$, one can
    encode the design in the likelihood as 
    $
      \proj( \obs - \F (\iparam) ) \sim \GM{\vec{0}}{\proj \Cobsnoise
      \proj\tran} \,,
    $
    where $\proj:=\proj(\binarydesign)$ is a sparse matrix that extracts 
    rows and columns of $\Cobsnoise$ that correspond to active sensors; that is, 
    $\proj:=\proj(\binarydesign) \in \mathbb{R}^{n_a \times
      \Nsens}$, $n_a=\sum \binarydesign$, $\proj_{ij} = 1$ if $\binarydesign_j$ is
      nonzero.
    In other words, $\proj$ is a binary matrix with only one entry equal to $1$ on
    each row, and the number of rows is equal to the number of active
    sensors. 
    If we define a binary weight matrix
    $\designmat(\binarydesign):=\Diag{\binarydesign}$, 
    then 
    $
    \proj(\binarydesign) [ \obs - \F (\iparam) ]
      \sim \GM{\vec{0}}{\proj(\binarydesign) \designmat(\binarydesign) 
        \Cobsnoise \designmat\tran(\binarydesign) 
      \proj\tran(\binarydesign)} \,,
    $
    and thus we can write the weighted data likelihood on the form (introducing
    $\designmat$ here does not make a difference given the definition of
    $\proj$, but it will be crucial when the design is relaxed):
    \begin{equation}\label{eqn:binary-mod-likelihood}
      \Like{\obs}{ \iparam; \design}
      \propto \exp{\left( - \frac{1}{2} 
        \sqwnorm{ \F \iparam - \obs}{\mathbf{Q}} 
        \right) } \,;\,\, 
        \mathbf{Q}:=
           \proj\tran(\binarydesign)
             \left( \proj(\binarydesign) \designmat(\binarydesign) 
               \Cobsnoise \designmat\tran(\binarydesign) \proj\tran(\binarydesign)
             \right)\inv \proj(\binarydesign)
        \,.   
    \end{equation}

    Note that pre- and postmultiplication of $\Cobsnoise$ by the diagonal
    matrix $\designmat$ means that the $(i, j)$th entry of $\Cobsnoise$ is
    scaled by $\binarydesign_i \binarydesign_j$, where $i,j=1,\ldots,\Nsens$.
    Loosely speaking, the effect can be explained as follows. 
    When $\binarydesign_i=0$, the $i$th row and column of
    $\Cobsnoise$ are set to $0$ and are then eliminated by the effect of $\proj$.
    Inversion of the covariance matrix is carried out in the projected space (by
    removing data corresponding to that inactive sensor then evaluating the inverse) and the precision matrix
    is then projected
    back to the original space by applying the transpose of $\proj$.
    \begin{lemma}\label{lemma:pseudo_inverse_prepost}
      Let $\designmat:=\Diag{\binarydesign}$, 
      $\binarydesign\in\{0, 1\}^{\Nsens}$.
      $\Cobsnoise\in\Rnum^{\Nsens \times \Nsens}$ is an admissible covariance matrix, and
      $\proj$ is as defined above. Then
        \begin{equation}
           \proj\tran(\binarydesign)
             \left( \proj(\binarydesign) \designmat(\binarydesign) 
               \Cobsnoise \designmat\tran(\binarydesign) \proj\tran(\binarydesign)
             \right)\inv \proj(\binarydesign)
           = \pseudoinv{\designmat(\binarydesign) 
               \Cobsnoise \designmat\tran(\binarydesign) } \,,
        \end{equation}
        where $\dagger$ denotes the Moore--Penrose (pseudo) inverse.
    \end{lemma}
    \begin{proof}
      See~\Cref{app:Proofs}.
    \end{proof}

    Despite being elementary, the significance
    of~\Cref{lemma:pseudo_inverse_prepost} is that it 
    enables rewriting the 
    weighted data likelihood~\eqref{eqn:binary-mod-likelihood} in the following equivalent
    form:
    \begin{equation}\label{eqn:binary-mod-likelihood-pseudo}
      \Like{\obs}{ \iparam; \design}
      \propto \exp{\left( - \frac{1}{2} 
        \sqwnorm{ \F \iparam - \obs}{ 
            \pseudoinv{\designmat(\binarydesign) 
               \Cobsnoise \designmat\tran(\binarydesign) 
             }
        } 
        \right) } \,.   
    \end{equation}

    Let us relax the design variable $\design$ to take any value in
    the domain $[0, 1]^{\Nsens}$, and let the observation errors be
    uncorrelated, that is, diagonal $\Cobsnoise$. 
    When the relaxed design attains a corner point
    of the domain, that is, $\design \in \{0, 1\}^{\Nsens}$, the argument and
    formulation of the likelihood~\eqref{eqn:binary-mod-likelihood-pseudo} hold.
    However, this formulation becomes inconsistent for any value of the design in
    the interior of the relaxation domain, that is, $\left(0, 1\right)^{\Nsens}$.
    Intuitively speaking, we aim to discard a sensor that does not provide
    valuable information (e.g., the sensor provides negligible information or 
    is associated with a very high uncertainty level).
    Now, consider a design value $\design=\epsilon \in(0, 1)$; by
    pre- and postmultiplying the variance of the $i$th sensor with $\epsilon$.
    By letting $\epsilon\rightarrow 0$, the posterior becomes sharper (the
    variance is reduced), which is the exact opposite of the desired
    behavior.
    
    \paragraph{Naive solution}
    In general, we need to formulate the relaxed weighted likelihood such that
    the weighted variance of the $i$th sensor increases as the associated design
    $\design_i\rightarrow 0$. This can be achieved by replacing the diagonal of
    the design matrix with weights $w_{i, i}$ defined as
    \begin{equation}\label{eqn:design_weights_diag}
      w_{i, i}(\design) = \begin{cases}
        \frac{1}{\design_i} &; \,\, \design_i \in(0, 1] \,,  \\
        0 &; \,\, \design_i = 0  \,,
      \end{cases} \quad i=1, \ldots, \Nsens \,.
    \end{equation}

    This way, as $\design_i\rightarrow 0$, we weight the sensor such that
    accuracy of the weighted version decreases. In the case of of uncorrelated
    observational errors, that is, $\Cobsnoise$ is diagonal, one can show that
    $\pseudoinv{\designmat\Cobsnoise\designmat}
      =\Diag{\design} \Cobsnoise\inv \Diag{\design},
    $ which is the traditional weighting form discussed in~\Cref{subsec:OED_sensor_placement}.
   
    The form~\eqref{eqn:design_weights_diag}, however, is invalid for correlated observations.
    Specifically, consider the case when $\Cobsnoise$ is nondiagonal, that is, the
    observation errors are correlated. 
    Pre- and postmultiplication of $\Cobsnoise$ by $\designmat$ corresponds to
    weighting the $(i, j)$ entry of $\Cobsnoise$ by $w_{i,i} w_{j,j}$. More
    specifically, the design variable $\design_i$ contributes to the weight of
    the $i$th row/column of
    $\widetilde{\mat{\Gamma}}_{\rm noise}(\design):=\designmat(\design)\Cobsnoise\designmat(\design)$. 
    If we define the weights as described by~\eqref{eqn:design_weights_diag}, then while the
    variance of the $i$th sensor increases as $\design_i\rightarrow 0$, the
    correlations between the $i$th sensors and other sensors are also magnified.
    However, a less important sensor (to be discarded) should have higher
    uncertainty and be less correlated with other sensors, and thus the
    definition in~\eqref{eqn:design_weights_diag} becomes
    invalid. 

    Note that pre- and postmultiplication of $\Cobsnoise\inv$ with 
    $\designmat:=\Diag{\design}$ means that the precision matrix $\Cobsnoise\inv$ 
    is evaluated first assuming that all sensors are active and is then
    weighted, which results in a discontinuity in the objective at $\design=0$.
    This can be shown in the context of the following simple two-dimensional example.
    
    \paragraph{Illustration of the issue} Following the definition of the linear forward and inverse problem
    in~\Cref{sec:background}, we define the parameter-to-observable map $\F$ as
    a short, wide matrix that projects a model parameter/state onto the observation
    space.
    The prior and observation
    covariance matrices are defined, and the posterior covariance matrix
    $\Cparampostmat$ is then used to formulate the A-optimality criterion $\Psi^\GA$ as
    defined by~\eqref{eqn:A_and_D_optimality} with $\Cobsnoise$ replaced by
    a weighted version based on the relaxed design.
    The forward operator (randomly generated) and prior and observation noise 
    covariances are
    \begin{equation} \label{eqn:oneD_Toy}\small
      \F :=  \begin{bmatrix}
        -0.125 & -0.15 &  1.145 & -0.475 \\
         0.485 & -2.13 &  0.41  &  0.495
      \end{bmatrix}
      \,; \quad 
      \Cobsnoise:= 
        \begin{bmatrix}
          2.0 & 1.0 \\
          1.0 & 2.0
        \end{bmatrix}
      \; \quad
      \Cparamprior := \mat{I} 
        \,,
    \end{equation}
    where $\mat{I}\in \Rnum^{4\times4}$ is an identity matrix, 
    with the following form of the A-optimality criterion:
    \begin{equation}\label{eqn:oneD_Toy_objective}
      \obj(\design):= \Psi^\GA(\design)= \Trace{\Cparampost(\design)}
        = \Trace{
          \left( \F\tran \wdesignmat(\design) \F  
          + \Cparampriormat\inv \right)\inv
          }
        \,,
    \end{equation}
    where~\eqref{eqn:binary-mod-likelihood-pseudo} and~\eqref{eqn:design_weights_diag} are utilized to define $\wdesignmat(\design)$.
    The forward operator $\F$ in~\eqref{eqn:oneD_Toy} describes 
    a simulation at four model gridpoints, with $\Nsens=2$ candidate observational sensors.
    The first sensor measures the average value simulated at the first two 
    model gridpoints, and the second sensor measures the average
    value simulated at the last two gridpoints.
    \Cref{fig:TwoD_Toy_PrePostCov} (left) shows a surface plot of the OED A-optimality criterion 
    as a function of the relaxed design $\design\in[0, 1]^2$.
    The surface plot is evaluated only at points in the interior of the domain $(0, 1)^2$.
    The value of the optimality criterion evaluated at the binary design
    values $\binarydesign\in\{0, 1\}^2$, namely, at the corners of the domain, are
    shown as colored circles. Similar results are  obtained when the
    covariance matrix is pre- and postmultiplied by the design square root, as
    shown in~\Cref{fig:TwoD_Toy_PrePostCov} (right).
    \begin{figure}[!ht]
      \centering
      \includegraphics[width=0.47\linewidth, 
        trim={0 0 0 80pt},clip]{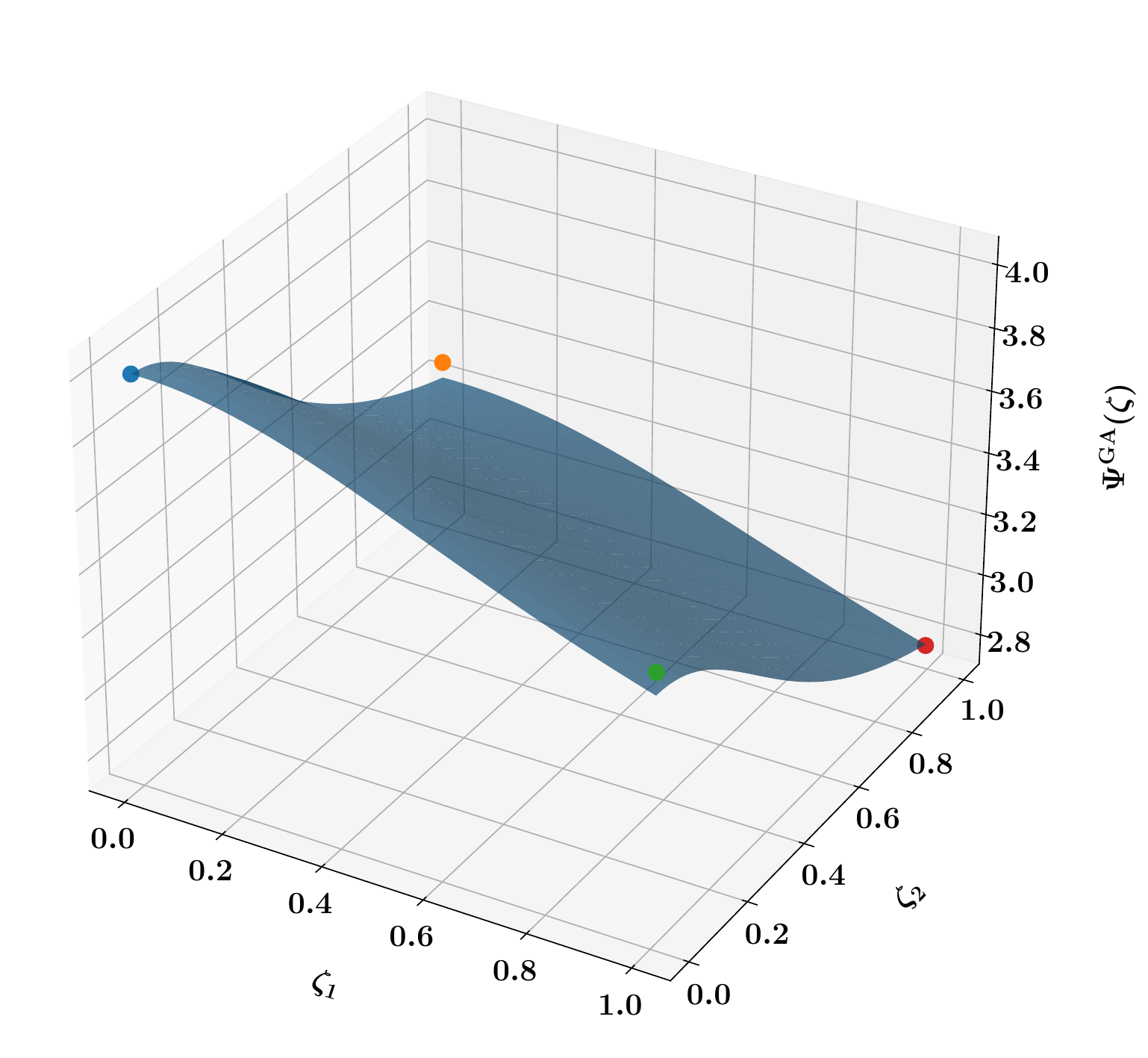}
      \includegraphics[width=0.47\linewidth, 
        trim={0 0 0 80pt},clip]{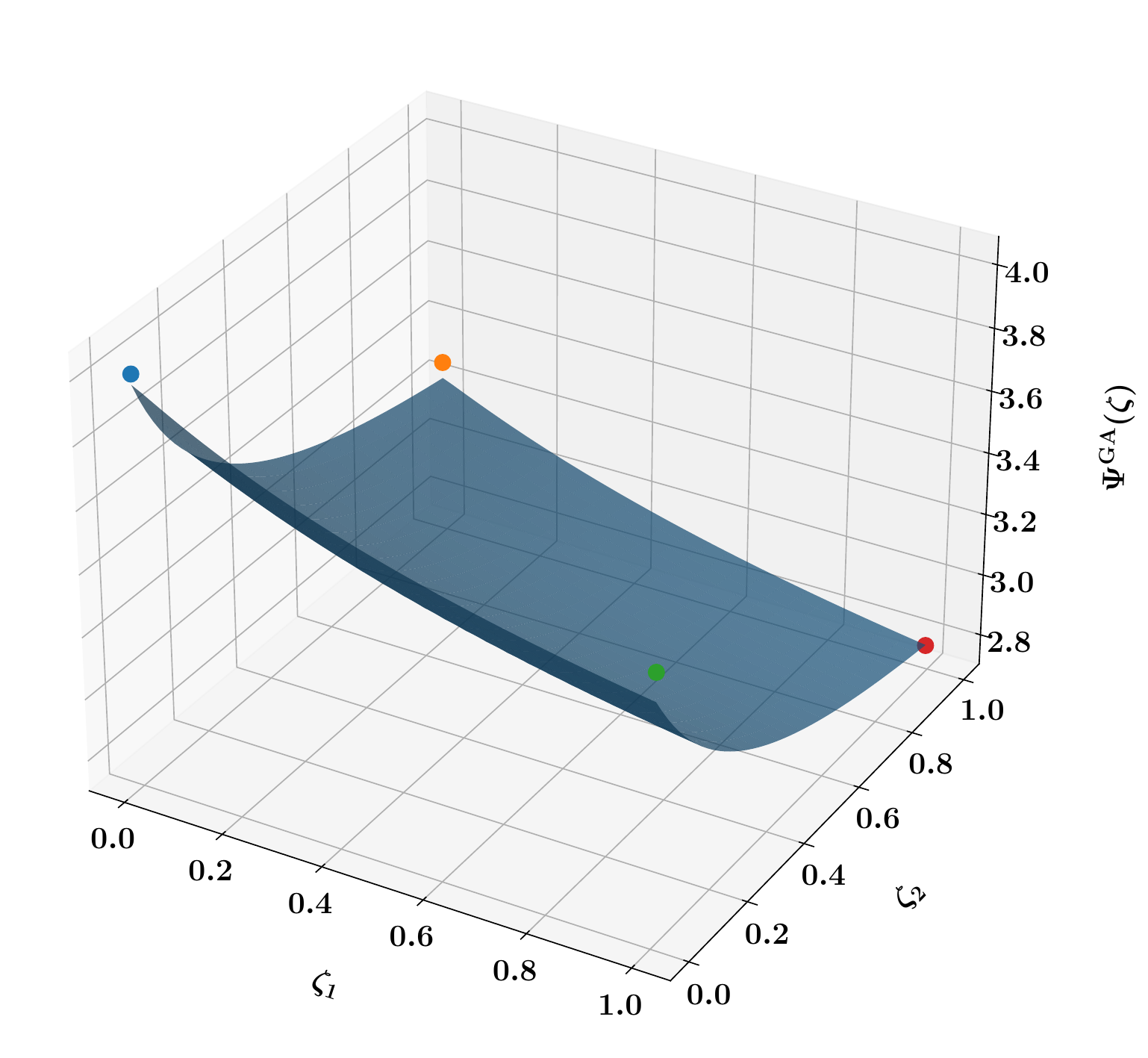}
      \caption{Values of the OED A-optimality criterion (objective) evaluated at points
        in the interior of the domain of the relaxed design $(0, 1)^2$, by using
        $\wdesignmat(\design):= \pseudoinv{\designmat(\design) 
             \Cobsnoise \designmat\tran(\design) }$ (left) and
        $\wdesignmat(\design):= \pseudoinv{\designmat^{\frac{1}{2}}(\design) 
             \Cobsnoise \designmat^\frac{\rm T}{2} (\design) }$ (right),
             respectively, with $\designmat(\design):=\Diag{\design}$.
        The value of the optimality criterion evaluated at the binary design
        values $\binarydesign\in\{0, 1\}^2$ is shown as distinct bullets.
        }
      \label{fig:TwoD_Toy_PrePostCov} 
    \end{figure}

    Results in~\Cref{fig:TwoD_Toy_PrePostCov} show that~\eqref{eqn:oneD_Toy_objective} is discontinuous at 
    $(0, 1)\tran,\,(1,0)\tran$. 
    While the mismatch in this example is not significant, such discontinuity is sufficient 
    to show that the standard relaxation approaches (precision pre- and postmultiplication
    or covariance pre- and postmultiplication with the ad hoc fix) are not
    valid in the case of correlated observations. Specifically, the limit of
    the relaxed objective, as the design attains a binary value, is not
    well defined. 
    This leads to discontinuity of the objective,
    invalidates the associated gradient, and shows that solving a relaxation 
    of the binary OED optimization problem is not guaranteed to produce a design
    that is either binary or matches the optimal solution of the original binary
    optimization problem.

  \subsection{Schur product formulation of the OED problem}
  \label{subsec:Schur_OED}
    To achieve the desired weighting behavior and resolve the issues discussed in~\Cref{subsec:OED_weighting}, 
    we reformulate the weighted likelihood by replacing the covariance matrix
    $\Cobsnoise$ with a weighted version
    $\widetilde{\mat{\Gamma}}_{\rm noise}(\design):=\designmat(\design) \odot \Cobsnoise$,
    where $\odot$ is the Hadamard (Schur) product, 
    $\wdesignmat(\design)=\widetilde{\mat{\Gamma}}^{\dagger}_{\rm noise}(\design)$,
    and $\designmat(\design)$ is
    a weighting matrix with entries defined as
    \begin{equation}\label{eqn:design_weights}
      \varweightfunc\left( t_m, t_n; \design_i, \design_j \right)
        = 
        \begin{cases}
          w_i(t_m, t_n)\, w_j(t_m, t_n) &;\,  i \neq j \\
          \begin{cases}
            0     &; \, w_i(t_m, t_n) =0 \\
            \frac{1}{w_i^2(t_m, t_n)}  &; w_i(t_m, t_n) \neq 0  
          \end{cases} &;\, i=j \\ 
        \end{cases} \,;\,\,
          \begin{matrix}
            & i,j = 1,2,\ldots,\Nsens \,,  \\[-1.0pt] 
            & m,n = 1,2,\ldots,\nobstimes  \,,
          \end{matrix}
    \end{equation}
    where $w_i(t_m, t_n)\in[0, 1]$ is a weight calculated based on the value of the design
    $\design_i$; this weight is applied to autocovariances of the $i$th candidate sensor at time instances $t_m,\,t_n$,
    respectively. 
    Autocovariance here refers to the covariance between observational noise at different time instances for the same observational sensor. 
    For uncorrelated temporal observational errors, the simplest form is to 
    let $w_i(t_m, t_n):=\design_i$.
    The weights given by~\eqref{eqn:design_weights} assure that
    the desired weighting scheme discussed in~\Cref{subsec:OED_weighting}
    is achieved by the weighting kernel $\designmat(\design)$. 
    The weighting function $\varweightfunc\left( t_m, t_n;\, \design_i,
    \design_j \right)$ is symmetric, and thus the weighting kernel $\designmat(\design)$ itself is 
    a symmetric and doubly nonnegative---real positive semidefinite square matrix---weighting 
    kernel.  
    The symmetry here is defined over any permutation of the time indexes or the
    design variables.
    In other words
    $
      \varweightfunc\left( t_m, t_n;\, \design_i, \design_j \right) = 
      \varweightfunc\left( t_n, t_m;\, \design_i, \design_j \right) = 
      \varweightfunc\left( t_m, t_n;\, \design_j, \design_i \right) = 
      \varweightfunc\left( t_n, t_m;\, \design_j, \design_i \right)
    $. 

    In the discussion below, we consider time-dependent settings, where the states 
    are checkpointed at the observation times for efficient evaluation and
    storage of the model state and the adjoints.
    As explained in~\Cref{subsec:numerical_setup},
    checkpointing is crucial for solving large-scale time-dependent Bayesian
    inverse problems.
    The case of time-independent models can be viewed as a special case of 
    the temporally uncorrelated setup with one observation time instance.
    We note that the entries of the observation error covariance matrix $\Cobsnoise$
    describe covariances between all candidate sensor locations, at all
    observation time instances, as discussed
    in~\Cref{subsec:OED_sensor_placement}.
    The general form of the weighting function given
    by~\eqref{eqn:design_weights} 
    enables scaling the entries of the covariance matrix $\Cobsnoise$, 
    where the respective weights are calculated based on both time and space.
    We explain this thought using a simplified example.
    Consider a time-dependent simulation with an observational grid 
    with $\Nsens=2$ candidate sensor locations.  
    Assume the observations are to be collected at two time instances
    $t_1,\, t_2$, respectively, with the observation error covariance matrix 
    \[
      \Cobsnoise=
      \left[
        \begin{array}{c:c}
        \mat{R}_{1,1}   & \mat{R}_{1,2}  
        \\ \hdashline
        \mat{R}_{2,1} & \mat{R}_{2,2}  
        \end{array} 
      \right] 
     \,,
    \]
    where each block $\mat{R}_{n,m} \in \Rnum^{2 \times 2}$ describes the 
    error covariances between observational gridpoints at time instances
    $t_n,\,t_m$, respectively, where $n,m=1,2$.
    The design $\design\in\Xi\subseteq\Rnum^2$ is associated with the 
    candidate sensor locations and thus is defined as 
    $\design=\left(\design_1, \design_2\right)\tran$.
    The weighting matrix is defined as follows:
    \[
      \designmat(\design) :=
      \left[
        \begin{array}{cc:cc}
          \varweightfunc(t_1, t_1;\, \design_1, \design_1)  &
          \varweightfunc(t_1, t_1;\, \design_1, \design_2)  & 
          \varweightfunc(t_1, t_2;\, \design_1, \design_1)  &
          \varweightfunc(t_1, t_2;\, \design_1, \design_2) 
          \\
          \varweightfunc(t_1,t_1;\, \design_2,\design_1)   &
          \varweightfunc(t_1,t_1;\, \design_2,\design_2)   &
          \varweightfunc(t_1,t_2;\, \design_2,\design_1)   &
          \varweightfunc(t_1,t_2;\, \design_2,\design_2)
          \\ \hdashline
          \varweightfunc(t_2,t_1;\, \design_1,\design_1)   &
          \varweightfunc(t_2,t_1;\, \design_1,\design_2)   &
          \varweightfunc(t_2,t_2;\, \design_1,\design_1)   &
          \varweightfunc(t_2,t_2;\, \design_1,\design_2)
          \\
          \varweightfunc(t_2,t_1;\, \design_2,\design_1)   &
          \varweightfunc(t_2,t_1;\, \design_2,\design_2)   &
          \varweightfunc(t_2,t_2;\, \design_2,\design_1)   &
          \varweightfunc(t_2,t_2;\, \design_2,\design_2)
        \end{array} 
      \right] \,.
    \]
    
    This shows that the dimension of the design is independent from the size of
    the temporal domain and is generally much less than the dimension of the 
    space-time domain.
    Additionally, this formulation gives more flexibility in forming the design matrix, 
    which acts as a weighting kernel in the observation space.
    For example, one can define the design matrix following the approach used to define 
    the covariance weighting kernels in~\cite{attia2018_oed_loc}.
    In fact, this formulation has some similarity with spatial covariance
    localization widely used in the  numerical weather prediction applications
    to remove spurious correlations in ensemble-based prior covariance matrices.
    Conversely, here we  define a design matrix $\designmat$ to weight/localize
    covariances between observational sensors based on their respective
    contribution to the uncertainty of the inversion parameter.
    If one believes that the effect of correlation on the design
    should decay in space or in the presence of spurious correlations in
    observation errors, one can encode this decay in the definition of the
    weighting function $\varweightfunc$, in other words, the entries of the design matrix
    $\designmat$.
    This formulation will enable us to fix the size of the design space to that of 
    the observation space, even when spatiotemporal correlations 
    are considered, thus reducing the computational cost of solving 
    the OED optimization problem.

    If the problem is time independent, the matrix $\Cobsnoise$ describes the
    covariances between errors of pairs of candidate sensors.
    In this case the weighting kernel $\designmat$ requires a weighting
    function that depends only on pairs of design variables, that is,
    $\varweightfunc:=\varweightfunc(\design_i,\,\design_j)$; and the weighting
    kernel, described by its respective entries, takes the form
    $\designmat=\left[
      \varweightfunc(\design_i,\,\design_j)\right]_{i,j=1,2,\ldots,\Nsens}$.
    
    Similarly, if the problem is time dependent, while only space correlations
    are considered, the weight 
    function $\varweightfunc(t_m,t_n;\,\design_i,\design_j)$ vanishes for
    any two different time instances $t_m \neq t_n$, 
    and the weighting kernel becomes a block diagonal matrix taking
    the form
    $
    \designmat
    \!=\!
      \directsum{m=1}{\nobs}{
        \left[\varweightfunc(\design_i,\,\design_j)\right]_{i,j=1,2,\ldots,\Nsens}
      } 
    $.
    Further discussion about the spatial weighting
    function $\weightfunc$ is given in~\Cref{subsec:weight_function}. 

  \subsection{On the choice of the weighting function}
  \label{subsec:weight_function}
  The value of the weighting function $\varweightfunc(\design_i,\design_j)$ 
  at any pair of observation gridpoints (and design variables) 
  is used to localize the effect of the correlation  between any pair of points. 
  Specifically, $\varweightfunc\left(\design_i,\design_j\right)
  :=\weightfunc(\design_i)\weightfunc(\design_j)$ scales the $(i,j)$th entry 
  of the spatial observation error covariance matrix $\Cobsnoise$.
  Moreover, when $i=j$, the weighting value scales the uncertainty
  level of the $i$th candidate sensor location.
  We regard the weighting value here 
  in the general sense as a value that indicates the relative importance of 
  each candidate sensor. 
  To formulate the gradient of the
  optimality criterion with respect to the design, with the range $[0,
    1]$, we need to require that the weighting function $\weightfunc$
  be differentiable. 
  Specifically, in order to control the sparsity of the design, in the optimization
  problem~\eqref{eqn:oed_relaxed_optimization_problem} the penalty function
  $\Phi(\cdot)$ is employed. 
  Assuming, for example, that an $\ell_p$ norm is used and that
  $\weightfunc_i\in[0, 1]$ is the weight associated with the $i$th
  candidate sensor, we define the penalty function as
  \begin{equation}\label{eqn:penalty_lp}
    \Phi(\design) := \Bigl\| 
      \Bigl( 
        \weightfunc_1,\, 
        \weightfunc_2,\,
        \ldots,\,
        \weightfunc_{\Nsens}
      \Bigr) \tran \Bigr\|_{p}  \,,
  \end{equation}
  where the penalty is asserted on
  the results of the weighting function, regardless of the  domain 
  of the design variables themselves. 
  The reason is that the role of the penalty function is to promote sparsity of the design and thus limit the number of activated sensors. 
  Following the discussion in~\Cref{subsec:OED_weighting}, a sensor can be deactivated when it is associated with 
  high uncertainty induced in the posterior and low correlation relative to the other sensors. Thus, sparsification can be achieved by reducing the weights $\weightfunc_i$, 
  which has the effect of reducing the correlation and increasing the uncertainty induced by specific sensors; see the formulation of $\varweightfunc$ given by~\eqref{eqn:design_weights}.
  Utilizing~\eqref{eqn:penalty_lp} as a penalty function drives the weights of the sensors to smaller values, which achieves the desired sparsification effect.

  As mentioned in~\Cref{subsec:Schur_OED}, the simplest approach is to let
  $\design_i\in[0, 1]$ and define the weight associated with the $i$th sensor as 
  \begin{equation}\label{eqn:cov-prod_kernel}
    \weightfunc(\design_i) \equiv \weightfunc_i 
      := \design_i; 
      \qquad
      \del{\weightfunc(\design_i)}{\design_k} = \delta_{i,k} \,,  
  \end{equation}
  where $\delta_{i,k}$  is the Kronecker delta function.
  The derivative will be used in developing the gradient of the OED optimization
  objective, as discussed in~\Cref{subsec:Schur_OED_Optimization_Problem}.
  Alternatively, one can let $\design_i\in(-\infty, 0]$ and define the weights as
  \begin{equation}\label{eqn:exp-prod_kernel}
    \weightfunc(\design_i) \equiv \weightfunc_i 
      := e^{\design_i}; 
      \qquad
      \del{\weightfunc(\design_i)}{\design_k} 
        = \weightfunc_i\, \delta_{i,k} \,.  
  \end{equation}
  Furthermore, one can let $\design_i\in\Rnum$ and define the weights using
  a sigmoid function as
  \begin{equation}\label{eqn:sigmoid-prod_kernel}
    \weightfunc(\design_i) \equiv \weightfunc_i 
      = \frac{1}{1+e^{- a \design_i} } ;
    \qquad
    \del{\weightfunc(\design_i)}{\design_k} 
      = a \, \weightfunc_i \left(1 - \weightfunc_i\right) \, \delta_{i, k} 
          \,,  
  \end{equation}
  where $a$ is a positive scaling factor (we generally set the scaling coefficient $a\!=\!1$, 
  unless otherwise stated explicitly). 
  The sigmoid function~\eqref{eqn:sigmoid-prod_kernel} frees the design
  variables to take any real value while keeping the weights in the interval $[0, 1]$ as desired.
  This will allow utilizing unconstrained optimization approaches to solve the
  relaxed OED optimization problem.

  In the remainder of~\Cref{sec:OED_Correlated_OBS} 
  we will keep the discussion independent from the specific choice of
  the design weighting function $\weightfunc$. 
  We start by formulating and discussing the relaxed OED optimization problem 
  and then discuss the solution approach in two cases. 
  In the first case the observation error covariance matrix is block diagonal 
  with only space correlations (\Cref{subsec:space_dependent_OED}), 
  and in the second case both space and time correlations are allowed 
  (\Cref{subsec:spacetime_dependent_OED}).

  \subsection{The relaxed OED optimization problem}
    \label{subsec:Schur_OED_Optimization_Problem}
  We utilize the weighted precision matrix $\wdesignmat(\design)$ to formulate the weighted likelihood for OED as follows:
  \begin{equation} \label{eqn:Schur_weighted_joint_likelihood}
    \Like{\obs}{\iparam; \design} 
      \propto \exp{\left( - \frac{1}{2} 
        \sqwnorm{\F \iparam - \obs}{\wdesignmat(\design)} \right) } \,; \quad
    \wdesignmat(\design) 
      := \pseudoinv{ \Cobsnoise  \odot \designmat(\design) } \,,
  \end{equation}
  where the weighting kernel $\designmat(\design)$ is constructed
  using~\eqref{eqn:design_weights}. Thus, 
  the A- and D-optimal design optimization
  problems~\eqref{eqn:A_and_D_optimality} take the following respective forms:
  \begin{subequations}\label{eqn:A_and_D_optimality_Schur}
  \begin{align}
    \design^{\rm A\!-\!opt} 
      &= \argmin_{\design\in\Xi} \Psi^\GA(\design) 
        := \Trace{\Predmat \left( \Fadj 
          \pseudoinv{ \Cobsnoise  \odot \designmat(\design) }  
        \F  
          + \Cparampriormat\inv \right)\inv \Predmat\adj
        }
        \,, \label{eqn:A_optimality_Schur}\\ 
    \design^{\rm D\!-\!opt} 
      &= \argmin_{\design\in\Xi} \Psi^\GD(\design) 
        := \logdet{\Predmat \left( \Fadj 
          \pseudoinv{ \Cobsnoise  \odot \designmat(\design) }  
        \F  
          + \Cparampriormat\inv \right)\inv \Predmat\adj
        }
        \,.\label{eqn:D_optimality_Schur}
  \end{align}
  \end{subequations}

  The remainder of~\Cref{subsec:Schur_OED_Optimization_Problem} is dedicated 
  to discussing the validity of the proposed formulation both theoretically and
  empirically.
  As discussed in~\Cref{subsec:OED_optimization}, in order to numerically solve the optimization 
  problems~\eqref{eqn:A_and_D_optimality_Schur}, the derivative of the weighted
  precision matrix $\wdesignmat(\design)$ is required.
  Thus, before solving the relaxed OED optimization problems~\eqref{eqn:A_and_D_optimality_Schur}, 
  it is important to show that $\wdesignmat(\design)$ is continuous in $\design$ and
  that it converges to the projected precision matrix obtained by using a binary design.
  To this end, for clarity we restrict the discussion to the case of spatial correlations,
  and we consider the case where
  $w_i:=\design_i\in[0,1],\,i=1,\ldots,\Nsens$ and 
  show that 
  \[
    \wdesignmat(\design) \rightarrow \proj\tran(\binarydesign)\Bigl({\proj(\binarydesign) \Cobsnoise
    \proj\tran(\binarydesign)}\Bigr)\inv \proj(\binarydesign)
    \,,
  \]
  as $\design\rightarrow\binarydesign$ for a binary design $\binarydesign$.
  The other forms of the weights, such as the exponential or sigmoidal
  weighting kernel, follow similarly since $w_{i,j}\in[0, 1]$.  
  \Cref{lemma:pseudo_inverse_pointwise} expands the form of the weighted
  precision matrix $\wdesignmat$ and shows similarity with~\Cref{lemma:pseudo_inverse_prepost}.
  \begin{lemma}\label{lemma:pseudo_inverse_pointwise}
    \begin{equation}
      \pseudoinv{ \Cobsnoise  \odot \designmat(\design) }
        = \proj\tran(\design) \left(\proj(\design) \Bigl( \Cobsnoise  \odot \designmat(\design) \Bigr)
        \proj\tran(\design) \right)\inv \proj(\design) \,.
    \end{equation}
  \end{lemma}
  \begin{proof}
    See~\Cref{app:Proofs}.
  \end{proof}

  The main result here is stated in~\Cref{theorem:continuity} and shows that the weighted
  precision $\wdesignmat(\design)$ is continuous and properly relaxes the
  binary design. The proof of~\Cref{theorem:continuity} makes use of \Cref{lemma:regularization_equivalence}, which shows that the weighting 
  scheme~\eqref{eqn:design_weights} is equivalent to regularization of pre-
  and postmultiplication, as suggested by~\eqref{eqn:noise_regularization}. 
  This approach is used to handle discontinuity issues raised for example in
  $c$-optimal design problems; see, for  example,~\cite[Section 5.4]{pronzato2013design}.
  \begin{lemma}\label{lemma:regularization_equivalence}
    Assume $\design \in(0, 1)^{\Nsens}$ and $\designmat(\design)$ is given by~\eqref{eqn:design_weights}. 
    Then
    $
    \left(\designmat(\design) \odot \Cobsnoise\right)\inv \rightarrow
    \mat{\Gamma}\inv_{\rm noise}
    $
    when $\design \rightarrow \vec{1}$,
    and 
    $
    \left(\designmat(\design) \odot \Cobsnoise\right)\inv \rightarrow
    \mat{0}
    $
    when $\design \rightarrow \vec{0}$ where the convergence is elementwise.
  \end{lemma}
  \begin{proof}
    See~\Cref{app:Proofs}.
  \end{proof}
  \begin{theorem}\label{theorem:continuity}
    The matrix-valued function
    $\wdesignmat(\design)=\pseudoinv{\Cobsnoise\odot\designmat(\design)}$,
    where $\design$ is a relaxed design $\design\in[0, 1]^{\Nsens}$ 
    and the entries of the design matrix $\designmat(\design)$ defined by
    by~\eqref{eqn:design_weights}, is continuous.
  \end{theorem}
  \begin{proof}
    See~\Cref{app:Proofs}.
  \end{proof}

  \subsection{Space correlations}\label{subsec:space_dependent_OED}
    Let $\design\in\Xi\subseteq\Rnum^{\Nsens}$ be the design, and assume that the 
    observation errors are temporally uncorrelated.
    In this case, 
    as discussed in~\Cref{subsec:Schur_OED},
    the observation error covariance matrix and the design weighting matrix
    take the form
    \begin{equation}\label{eqn:space_covariance_designmat}
      \Cobsnoise = \directsum{m=1}{\nobstimes}{\mat{R}_m}
      \,; \quad
      \designmat 
      = 
      \directsum{m=1}{\nobstimes}{
        \sum_{i,j=1}^{\Nsens}{ \varweightfunc(\design_i,\,\design_j) 
          \vec{e}_i \vec{e}_j\tran }  
      }  \,,
    \end{equation}
    where $\vec{e}_i$ is the $i$th versor of $\Rnum^{\Nsens}$, that is, the  $i$th vector in the natural basis. 
    In this case the derivative of $\designmat$ with respect to a given entry 
    of the design vector, $\design_j$,
    is a sparse symmetric matrix with only nonzero entries in the $j$th row 
    and the $j$th column of each block.
    Specifically, if we define the vector of elementwise partial derivatives
    of weights 
    \begin{equation}\label{eqn:weights_derivative_space}
      \vec{\eta}_j 
        := \left( 
        \frac{ \vec{\eta}^{(1)}_j}{1\!+\!\delta_{1,j}}, 
        \ldots, 
        \frac{\vec{\eta}^{(\Nsens)}_j}{1\!+\!\delta_{\Nsens,j}} 
        \right)\tran 
        ; \,\,\,
      \vec{\eta}^{(i)}_j
        := 
        \begin{cases}
          \del{w_i}{\design_i} w_j  &;  i \neq j \\
          \begin{cases}
            0     &;  w_i =0 \\
            \frac{-2}{w_i^3}  \del{w_i}{\design_i}   &; w_i \neq 0  
          \end{cases} &; i=j \\ 
        \end{cases} ;
          \begin{matrix}
            & i,j = 1,\ldots,\Nsens ,  \\[-1.0pt] 
          \end{matrix}
    \end{equation}
    which is obtained by piecewise differentiation
    of~\eqref{eqn:design_weights},
    then it immediately follows that 
    \begin{equation}\label{eqn:weighted_obs_derivative_space}
      \del{\wdesignmat(\design)}{\design_j} 
        = - \wdesignmat(\design) 
          \directsum{m=1}{\nobstimes}{ 
          \vec{e}_j \left( \left(\mat{R}_m \vec{e}_j\right) 
            \odot \vec{\eta}_j\right)\tran 
          + \left( \left(\mat{R}_m \vec{e}_j\right) 
            \odot \vec{\eta}_j\right) \vec{e}_j\tran 
        }   \wdesignmat(\design) \,,
    \end{equation}
    where 
    we utilized the distributive property of the Hadamard product 
    and the fact that $\mat{R}_m$ is symmetric and
    $
      \mat{R}_k\odot\left(\vec{\eta}_j\vec{e}_j\tran\right) 
        = \left( \left( \mat{R}_k\vec{e}_j \right)
          \odot\vec{\eta}_j \right) \vec{e}_j\tran
    $.
    From~\eqref{eqn:Schur_weighted_joint_likelihood}, it follows that
    \begin{equation}\label{eqn:spacial_weighted_noise}
      \wdesignmat(\design) 
        = \directsum{m=1}{\nobstimes}{ \mat{V}_m^{\dagger}(\design)  } \,,
        \quad
        \mat{V}_m(\design) 
        := 
            \mat{R}_m \odot 
            \left( \sum_{i,j=1}^{\Nsens}{ \varweightfunc(\design_i,\,\design_j) 
              \vec{e}_i \vec{e}_j\tran }   \right)
          \,,
    \end{equation}
    where the pseudo inverse in~\eqref{eqn:spacial_weighted_noise} 
    can be efficiently evaluated given the fact that
    $
    \mat{V}_m^{\dagger}(\design) 
    = \proj_m\tran \left(\proj_m \mat{V}_m(\design) \proj_m\tran\right)\inv\proj_m 
    $,
    where $\proj_m$ extracts the rows/columns from $\mat{R}_m$ corresponding to
    active sensors and all are equal; that is, $\proj_k=\proj_\ell\,,~\forall
    k,\ell=1,\ldots,\nobstimes$.
    Note that~\eqref{eqn:design_weights} and~\eqref{eqn:weights_derivative_space} contain
    terms inversely proportional to the second and third powers of the weights $w_i$, respectively. 
    These terms tend to infinity as ${w_i}$ approaches zero, which may yield numerical problems unless handled properly in the software implementation.
    Thus, a robust implementation of this formulation requires handling such difficulty. 
    The simplest approach is to apply proper rounding to very small weights which can, for example, be encoded in $\proj$.
    Given the general form of the derivative~\eqref{eqn:general_gradient_entries}, 
    the gradients of the A- and D-optimality
    criteria~\eqref{eqn:A_and_D_optimality}, 
    in the case of temporally uncorrelated observation errors, take the
    following respective forms (detailed derivation is given
    in~\Cref{app:A_Optimality_Space},~\Cref{app:D_Optimality_Space}):
    \begin{subequations}\label{eqn:A_and_D-optimality_gradient_space_final}
      \begin{align}
        \nabla_{\design}{ \Psi^\GA(\design) }
          &= 2 \sum_{m=1}^{\nobstimes}
          \diag{  
              \mat{V}_m^{\dagger}(\design) \Find{0}{m}
              \Hessmat\inv(\design) \Predmat\adj \Predmat \Hessmat\inv(\design)
              \Fadjind{0}{m} \mat{V}_m^{\dagger}(\design)  
              \left( \mat{R}_m \odot \designmat^{\prime} \right)
          } \,,\label{eqn:A-optimality_gradient_space_final}  \\
        \nabla_{\design}{ \Psi^\GD(\design) }
          &= 2 \sum_{m=1}^{\nobstimes} \diag{\!
            \mat{V}_m^{\dagger}(\design) \Find{0}{m} 
            \Hessmat\inv\!(\design) \Predmat\adj \Cpredpost\inv\!(\design)
            \Predmat \Hessmat\inv\!(\design)  \Fadjind{0}{m}
            \mat{V}_m^{\dagger}(\design) 
              \left( \mat{R}_m \!\odot\! \designmat^{\prime} \right) } 
                \,,\label{eqn:D-optimality_gradient_space_final}
      \end{align}
    \end{subequations}
    where $\diag{\mat{A}}$ is the diagonal of a square matrix $\mat{A}$, 
    $\Fadjind{0}{m}$ is the adjoint of $\Find{0}{m}$, 
    and $\designmat^{\prime}$ is a matrix with columns set by 
    using~\eqref{eqn:weights_derivative_space}; that is,
    \begin{equation}\label{eqn:designmat_derivative}
      \designmat^{\prime} = \left[ \vec{\eta}_1,\, \vec{\eta}_2, \ldots,\,
      \vec{\eta}_{\Nsens} \right]  \,. 
    \end{equation} 

    If we assume a diagonal observational error covariance matrix
    $\Cobsnoise$ and utilize the kernel~\eqref{eqn:cov-prod_kernel}, then the
    matrix of the derivative reduces to $\designmat^{\prime}=\frac{1}{2}\mat{I}$.
    In this case the gradient of the A-optimality
    criterion~\eqref{eqn:A-optimality_gradient_space_final} and the gradient of
    the D-optimality criterion~\eqref{eqn:A-optimality_gradient_space_final},
    respectively, reduce to
    \begin{subequations}
      \begin{align} \label{eqn:spatial_correlation_gradient} 
        \nabla_{\design}{ \Psi^\GA(\design) }
          &= \sum_{m=1}^{\nobstimes} \sum_{i=1}^{\Npred}
            \vec{s}_{i,m} \odot \vec{s}_{i,m}\,; \qquad
        \nabla_{\design}{ \Psi^\GD(\design) }
          = \sum_{m=1}^{\nobstimes} \sum_{i=1}^{\Npred}
            \widetilde{\vec{s}}_{i,m} \odot \widetilde{\vec{s}}_{i,m}\,, \\
        \vec{s}_{i,m} 
          &:=  \pseudoinv{\Diag{\design}\mat{R}_m^{\frac{1}{2}} }
              \Find{0}{m} \Hessmat\inv(\design) \Predmat\adj \vec{e}_i 
          \,,  
          \\
        \widetilde{\vec{s}}_{i,m} 
          &:= \pseudoinv{\Diag{\design}\mat{R}_m^{\frac{1}{2}} }
            \Find{0}{m} \Hessmat\inv(\design) \Predmat\adj 
              \Cpredpost^{-\frac{1}{2}}\!(\design)\vec{e}_i
         \,, 
      \end{align} 
    \end{subequations}
    where we utilized the following~\Cref{lemma:mat_id_1}, which can be proven
    with elementary linear algebra. 
    %
    \begin{lemma}\label{lemma:mat_id_1}
      Given a symmetric matrix $\mat{A}$ and a nonnegative diagonal matrix $\mat{B}$ such
      that $\mat{A},\mat{B}\in \Rnum^{n\times n}$, then
      $ 
        \diag{\mat{A}\mat{A}\tran \mat{B}} 
        = \diag{\mat{B}^{\frac{1}{2}} \mat{A}\mat{A} \mat{B}^{\frac{1}{2}}} 
        = \sum_{i=1}^{n} \vec{s}_i \odot
        \vec{s}_i\,; \text{ where }
        \vec{s}_i = \mat{B}^{\frac{1}{2}} \mat{A} \vec{e}_i \,.
      $ 
    \end{lemma}

    Note that in this case, that is, when $\mat{R}_m$ are diagonal, 
    by utilizing the weighting kernel~\eqref{eqn:cov-prod_kernel}
    $\varweightfunc(\design_i, \design_j) := \design_i \design_j$ 
    we retrieve the standard OED formulation:
    \begin{equation}
      \wdesignmat 
        = \directsum{m=1}{\nobstimes}{ 
          \pseudoinv{ \pseudoinv{\Diag{\design}} \mat{R}_m  \pseudoinv{\Diag{\design}}  } 
        }
        = \directsum{m=1}{\nobstimes}{ 
          \Diag{\design} \mat{R}\inv_m \Diag{\design}  } 
        \,.
    \end{equation}

    This shows that the traditional OED formulation (for uncorrelated
    observational errors) is equivalent to applying a
    Hadamard product weighting kernel to the observation covariance matrix.

    Generally speaking, the proposed formulation reduces to the standard OED 
    approach only if the observation errors are uncorrelated 
    and the symmetric kernel function utilized is separable, 
    that is, $\weightfunc(\design_i, \design_j) = g(\design_i) g(\design_j)$ 
    for some some function
    $g_*(\design_*):\Omega \subseteq \Rnum \rightarrow[0,1]$.
    In this case the weighted precision matrix can be written as 
    \[
      \wdesignmat= \Bigl(
        \Diag{g(\design_1), g(\design_2),\ldots, g(\design_{\Nsens}) } \,
        \Cobsnoise \,
        \Diag{g(\design_1), g(\design_2),\ldots, g(\design_{\Nsens}) } 
        \Bigr)\inv \,. 
    \]

  Note that the gradients~\eqref{eqn:A_and_D-optimality_gradient_space_final}
  are well defined only if~\eqref{eqn:weighted_obs_derivative_space} is
  continuous. This is stated in~\Cref{lemma:gradient_continuity}. 
  \begin{lemma}\label{lemma:gradient_continuity}
    The matrix-valued entrywise
    derivative~\eqref{eqn:weighted_obs_derivative_space} is
    continuous over the relaxation domain $\design\in[0, 1]^{\Nsens}$.
  \end{lemma}
  \begin{proof}
    See~\Cref{app:Proofs}.
  \end{proof}

  We conclude this subsection with an empirical validation.
  Specifically, we validate the formulation of the relaxed
  objective~\eqref{eqn:A_and_D_optimality_Schur} and
  derivative~\eqref{eqn:A_and_D-optimality_gradient_space_final} empirically
  using the two-dimensional problem~\Cref{eqn:oneD_Toy}.
  \Cref{fig:TwoD_Toy_PointwiseCov} shows a surface plot (left) of the relaxed
  objective $\Psi^\GA(\design)$ described by~\eqref{eqn:A_optimality_Schur}, 
  and a vector-field plot (right) of the gradient~\eqref{eqn:A-optimality_gradient_space_final} evaluated at the
  discretization point of the relaxed design space. 
  The results show that the relaxation produces a surface that properly connects
  the corner points that represent the possible values of the binary design.
  Moreover, the gradient shows continuity in the whole domain including near the
  boundaries when any of the design variables attain a binary value.
  \begin{figure}[!ht]
    \centering
    \includegraphics[width=0.48\linewidth, 
      trim={0 0 0 80pt},clip]{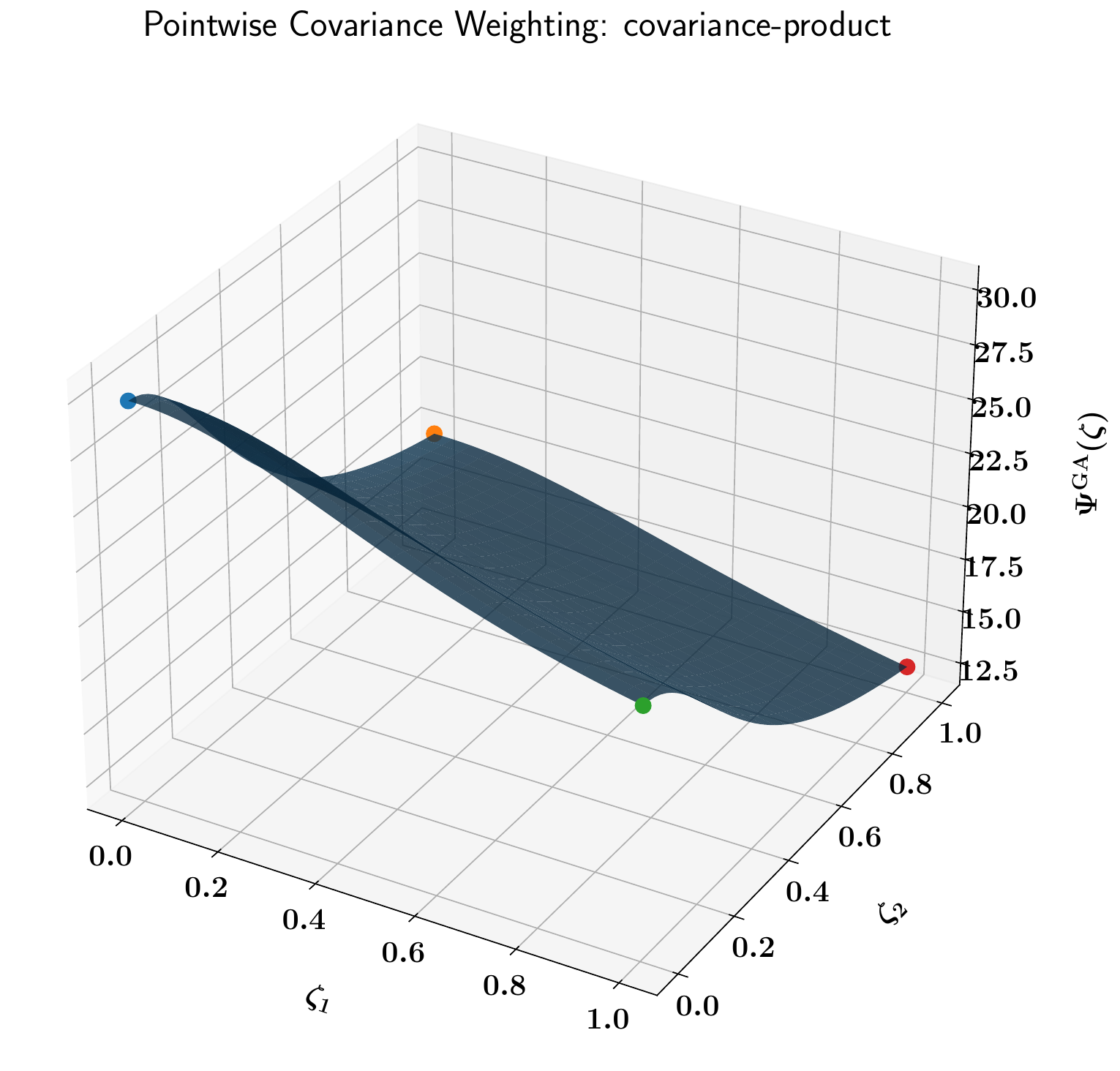}
    \hfill 
    \includegraphics[width=0.45\linewidth, 
      trim={0 0 0 5pt},clip]{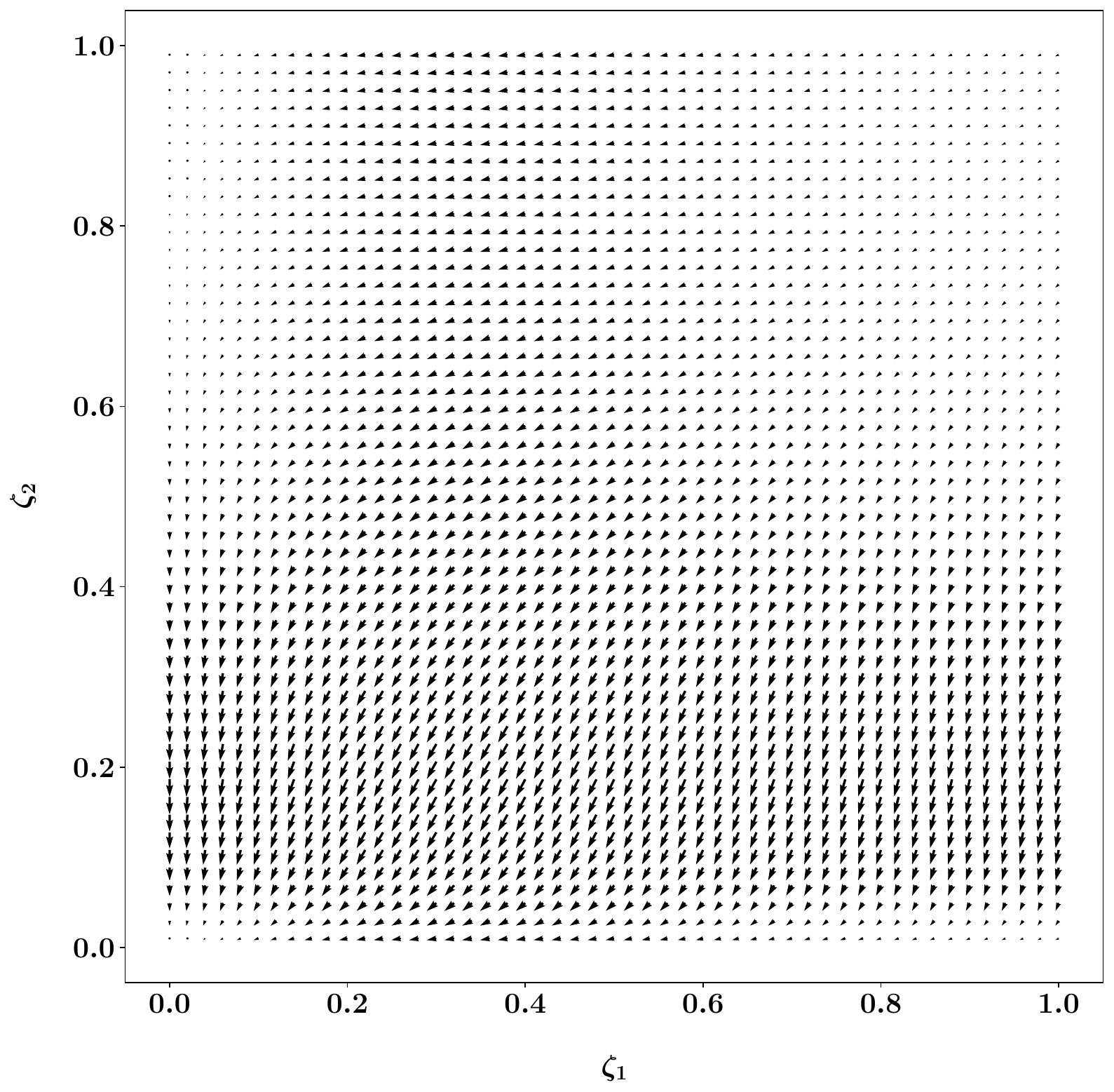}
    \caption{
      Left: values of the OED A-optimality objective~\eqref{eqn:A_optimality_Schur} evaluated at points
        in the interior and at the boundaries of the domain of the relaxed design $(0, 1)^2$.
        The values of the objective evaluated at the binary design
        values $\binarydesign\in\{0, 1\}^2$ are shown as colored circles. 
      Right: a quiver plot showing a vector-field representation of the
      gradient~\eqref{eqn:A-optimality_gradient_space_final} evaluated at the
      discretization point of the relaxed design space (including boundary
      points).
    }
    \label{fig:TwoD_Toy_PointwiseCov} 
  \end{figure}
  %

  \subsection{Spatiotemporal correlations}\label{subsec:spacetime_dependent_OED}
    In time-dependent problems, when the observation errors are temporally correlated, 
    the observation noise covariance $\Cobsnoise$ is a full symmetric matrix, 
    with blocks $\mat{R}_{mn}$ representing cross-covariances between 
    observation errors at time instances $t_n$ and $t_m$, respectively.
    In space-time settings, one can associate a design variable
    $\design_{i,m}$ for the $i$th candidate sensor, from $\Nsens$ candidate 
    locations, at time instance $t_m$. 
    In this case, however, the design space
    $\Rnum^{\Nsens\times\nobstimes}$ grows with the number of observation time
    instances $\nobstimes$, and the decision-making process (e.g., thresholding) 
    will be harder to carry out or even interpret.
    An alternative approach is to associate a design variable $\design_i$ 
    with the $i$th candidate sensor location, while controlling its effect 
    in space and over time. 
    For example, we can fix the design variable $\design_i$ over time, 
    that is, $\design_i= \design_{i,m}\,,~\forall m=1,2,\ldots,\nobstimes$, 
    while controlling its effect  
    by using  the form of $\weightfunc$ and by introducing 
    a temporal decorrelation function.
    
    In the majority of time-dependent applications, the temporal correlation strength 
    generally decays in time.
    In some applications, the correlation between distant temporal
    points is dominated by spurious correlations, and their effect on the design
    should be damped out. In this case, in order to choose an optimal design, 
    the temporal effect of a design weight for a candidate sensor location
    should also decay in time to damp down the effect of spurious correlations
    on the design selection. 
    For example, one could use 
    $\varweightfunc(\design_i,\design_j;\, \cdot,\cdot) := \rho(t_m,\,t_n)\,
    \varweightfunc(\design_i,\design_j;\cdot,\cdot)$ 
    to weight the entries of $\mat{R}_{mn}$, where, as before,
    $\varweightfunc(\cdot)$ is a space weighting function and 
    $\rho(t_m,\,t_n)$ is a symmetric function such that $\rho(t_m,\,t_m)=1$ 
    and is conversely related with the distance between $t_m$ and $t_n$.
    In this case the entries of $\Cobsnoise$ are weighted by
    \begin{equation}
      \varweightfunc(\design_i,\design_j;\, t_m,t_n) 
        :=  
          \begin{cases}
            \rho(t_m,\,t_n)\, \weightfunc(\design_i)\, \weightfunc(\design_j) &;\,  i \neq j \\
            \begin{cases}
              0     &; \, \weightfunc(\design_i) =0 \\
              \frac{\rho(t_m,\,t_n)\, }{\weightfunc(\design_i)^2}  &; \weightfunc(\design_i) \neq 0  
            \end{cases} &;\, i=j \\ 
          \end{cases} \,;\,\,
            \begin{matrix}
              & i,j = 1,2,\ldots,\Nsens \,,  \\[-1.0pt] 
              & k,l = 1,2,\ldots,\nobstimes  \,,
            \end{matrix}
    \end{equation}
    where, as before, $\weightfunc_i\equiv\weightfunc(\design)$ is the weight
    associated with the $i$th candidate sensor.
    One could use a Gaussian-like decorrelation function 
    \begin{equation}\label{eqn:Gauss_decorrelation}
      \rho(t_m,\,t_n):=\brexp{{-d (t_m,\, t_n)}/{2 \ell^2}} \,,
    \end{equation}
    where $d (t_m,\,t_n)$ is a distance function between time 
    instances $t_m$ and $t_n$ and where $\ell$ is a predefined temporal-correlation 
    length scale. 
    An alternative choice of the weighting coefficients is to employ 
    the fifth-order piecewise-rational function~\eqref{eqn:Gaspari_Cohn} 
    of Gaspari and Cohn~\cite{Gaspari_1999_correlation}:
    \begin{equation}\label{eqn:Gaspari_Cohn}
      \rho(t_m,t_n) =
      \begin{cases}
        - \frac{\upsilon^5}{4}  
        + \frac{\upsilon^4}{2} 
        + \frac{5 \upsilon^3}{8}  
        - \frac{5 \upsilon^2}{3} 
        + 1\,, \quad  & 0 \leq \upsilon \leq 1 \\
        \frac{\upsilon^5}{12}  
        - \frac{\upsilon^4}{2} 
        + \frac{5 \upsilon^3}{8}
        + \frac{5 \upsilon^2}{3}
        - 5 \upsilon
        + 4 
        - \frac{2}{3 \upsilon} \,,  
        \quad & 1 \leq \upsilon \leq 2 \\
        0 \,, \quad & 2 \leq \upsilon \,,
      \end{cases}
    \end{equation}
    where $\upsilon:=\frac{d(t_m,t_n)}{\ell}$ and, as before, 
    $\ell$ is a predefined correlation length scale and  
    $d(t_n,t_m)$ measures the temporal distance between time 
    instances $t_n,\,t_m$. 
    The simplest choice of such a temporal distance function is the Euclidean
    distance, defined as $d (t_m,\, t_n):=\left|t_m - t_n\right|$.
    Note that the parameter $\ell$ here controls the speed at which 
    the decorrelation decays over time, and this should 
    be application specific.
    The function~\eqref{eqn:Gaspari_Cohn} is designed to have compact support 
    such that it is nonzero only for a small local region and zero everywhere 
    else and is used for covariance localization 
    in the data assimilation context; 
    see, for
    example,~\cite{Hamill_2001,Whitaker_2002a,attia2018_oed_loc,moosavi2019tuning}.  

    Note that this function is independent from the design variable $\design$
    and is introduced for flexibility, for example, to damp out spurious correlations.
    One can simply choose $ \rho(t_m,t_n)=1$ if the observation error correlations are
    correctly specified at all scales. The discussion below is independent from
    the choice of the temporal weighting function $\rho$.

    In space-time settings, the weighted version of the space-time observation error 
    precision matrix takes the form 
    $\wdesignmat(\design)= \pseudoinv{ \Cobsnoise \! \odot\! \designmat(\design) }$,
    with the space-time weighting matrix $\designmat(\design)$ defined, elementwise, as 
    \begin{equation}\label{eqn:spacetime_designmatrix}
      \designmat(\design) 
        :=  
        \left[
          \varweightfunc 
          \left(
            \design_{ {(k-1)} \bmod {\Nsens} +1} , 
            \design_{ {(h-1)} \bmod {\Nsens} +1} ;\,  
            t_{ \Floor{\frac{k-1}{\nobstimes}} +1} , 
            t_{ \Floor{\frac{h-1}{\nobstimes}} +1}
            \right)
        \right]_{k,h=1,2,\ldots,\Nobs} 
      \,,
    \end{equation}
    where $\floor{\cdot}$ is the floor operation 
    and $\bmod$ represents the modulo 
    operation.
    Note that the entrywise representation of the weighting
    kernel~\eqref{eqn:spacetime_designmatrix} shows that the design variables
    are time independent.

    The derivative of the optimality criterion, required for the gradient-based
    solution of the OED problem, follows directly once the matrix derivative 
    of the weight matrix $\designmat(\design)$ in~\eqref{eqn:spacetime_designmatrix} is
    formulated.
    To this end, let us define the vector of partial derivatives 
    $\vec{\vartheta}_{k,\,n}\in \Rnum^{\Nobs}$, $\Nobs=\Nsens\!\times\!\nobstimes$,
    with the $k$th entry denoted as $\vec{\vartheta}_{i,m}[k]$ and given by 
    \begin{equation}\label{eqn:weights_derivative_spacetime}
      \vec{\vartheta}_{i,m}[k]:
        = \frac{1}{1 \! +\! \delta_{i,\, {(k\!-\!1)} \bmod {\Nsens} +1 }} 
        \del{}{\design_i}
          \varweightfunc
          \left( \!
            \design_{i}, 
            \design_{{(k\!-\!1)} \bmod {\Nsens} +1} ;\, 
            t_{m}, 
            t_{\Floor{\!\frac{k-1}{\nobstimes}\!} +1 \!}  \!
          \right)
          \,; \, 
          \begin{matrix}
            & i = 1,\ldots, \Nsens\,,  \\[-2.0pt]
            & m = 1,\ldots, \nobstimes \,, \\[-2.0pt]
            & k = 1,\ldots, \Nobs \,, 
          \end{matrix}
    \end{equation}
    where the piecewise partial derivatives
    in~\eqref{eqn:weights_derivative_spacetime} are evaluated using~\eqref{eqn:weights_derivative_space}.
    The partial derivative of the weight matrix $\designmat(\design)$ with respect to the
    design variable $\design_i$ is a symmetric matrix given by
    \begin{equation}\label{eqn:designmatrix_derivative_spacetime}
      \del{\designmat(\design)}{\design_i} 
        = \sum_{m=1}^{\nobstimes} \vec{e}_q \vec{\vartheta}_{i, m}\tran 
          + \sum_{m=1}^{\nobstimes} \vec{\vartheta}_{i, m} \vec{e}_q \tran  
        \,,\quad 
        q={i+(m\!-\!1)\Nsens} 
        \,;\quad 
        i=1,2,\ldots,\Nsens 
        \,,
    \end{equation}
    where $\vec{e}_q$ here is the $q$th natural basis vector in
    $\Rnum^{\Nobs}$. 
    By applying the distributive property of the Hadamard product, the partial derivatives of the weighted precision matrix take the form
    \begin{equation}\label{eqn:spacetime_precision_derivative}
      \del{\wdesignmat(\design)}{\design_i} 
          = - \wdesignmat(\design) \left( 
            \sum_{m=1}^{\nobstimes} 
            \vec{e}_q \left(
                \left( \Cobsnoise\inv \vec{e}_q \right) 
                \odot \vec{\vartheta}_{i, m}
              \right) \tran
            +
            \sum_{m=1}^{\nobstimes}
            \left(
              \left( \Cobsnoise\inv \vec{e}_q \right) 
              \odot \vec{\vartheta}_{i, m}
              \right) \vec{e}_q \tran \right) \wdesignmat(\design) \,.
    \end{equation}

    Given (\ref{eqn:spacetime_precision_derivative},~\ref{eqn:weights_derivative_spacetime}) 
    and, as shown in (\Cref{app:A_Optimality_SpaceTime},~\Cref{app:A_Optimality_SpaceTime}),
    in the presence of spatiotemporal observation correlations, the gradient of 
    the A- and D-optimality criteria  take the following respective forms: 
    \begin{subequations}
      \begin{align}
        \nabla_{\design}{ \Psi^\GA (\design) }
          &= 2 \sum_{i=1}^{\Nsens} \sum_{m=1}^{\nobstimes}
            \vec{e}_{i} \vec{e}_q \tran
            \wdesignmat(\design) \F \Hessmat\inv(\design) \Predmat\adj \Predmat \Hessmat\inv(\design) \Fadj \wdesignmat(\design)  
              \left(
                \left( \Cobsnoise\inv \vec{e}_q \right)
                \odot \vec{\vartheta}_{i, m}
              \right)  \,, \label{eqn:spatiotemporal_correlations_gradient}  \\
          \nabla_{\design}{ \Psi^\GD\!(\design) }
            &= 2 \sum_{i=1}^{\Nsens}\! \sum_{m=1}^{\nobstimes}\!
              \vec{e}_i \!\left( 
                  \Cobsnoise\inv \vec{e}_q \! \odot\! \vec{\vartheta}_{i, m} 
                \right) \tran \! 
              \wdesignmat(\design) \F  \Hessmat\inv\!(\design) \Predmat\adj  \Cpredpost\inv\!(\design) \Predmat 
                \Hessmat\inv\!(\design)  \Fadj \wdesignmat(\design)  \vec{e}_q \,.
      \end{align}
    \end{subequations}
    %

    \subsection{Computational considerations}
    \label{subsec:computational_cost}
      The OED objective function and the associated gradient together form the 
      main bottleneck in the process of an OED optimization problem. 
      Evaluating the objective function $\T$ of the OED optimization
      problem~\eqref{eqn:oed_relaxed_optimization_problem} and the associated gradient 
      requires evaluating the optimality criterion $\obj$ and the
      associated gradient. They also require specifying the penalty function
      $\Phi$. The cost of evaluating the penalty function and the associated
      gradient, with respect to the design parameter,  is negligible, however,
      compared with the cost of evaluating $\obj$ and the associated gradient. 
      Thus, in what follows we focus on the cost of the A-optimality
      criterion and the associated gradient in terms of the
      number of forward model evaluations. 
      The analysis extends easily to other optimality criteria including D-optimality.

      For simplicity, we assume that the prediction operator $\Predmat$ and the
      corresponding adjoint $\Predmat\adj$ each require one forward model
      evaluation.
      The A-optimality criterion $\design^{\rm A\!-\!opt}\! $ described by~\eqref{eqn:A_and_D_optimality}
      requires one Hessian solve, a forward, and an adjoint integration of the
      model $\F$ for each entry of the posterior covariance matrix diagonal. 
      The Hessian, being the inverse of the posterior covariance of the model
      parameter $\iparam$, is a function of the relaxed design; see, for example,~\cite{attia2018goal}
      for details.

      A preconditioned conjugate gradient (CG) method is used for Hessian solves,
      which requires a forward and an adjoint model evaluation for each
      application of the Hessian.
      We use the prior covariance as a preconditioner. In this case, and by assuming 
      the numerical rank of the prior preconditioned data misfit 
      Hessian~\cite{bui2013computational} is $r \ll
      \Nstate$, then one Hessian solve costs $\mathcal{O}(r)$ CG iterations, that
      is, $\mathcal{O}(2 r)$ forward model evaluations.
      Thus, one evaluation of the objective function $\T$
      in~\eqref{eqn:oed_relaxed_optimization_problem}, assuming A-optimality,
      costs $\Oh{2\, r\, \Npred + 2\, \Npred}=\Oh{2\, r\, \Npred }$ forward model solves.
      
      In the case of space correlations, evaluating the gradient of the
      A-optimality criterion, as described
      by~\eqref{eqn:spatial_correlation_gradient}, requires one Hessian solve and
      two applications of the forward model (including one evaluation of the
      prediction operator) for each vector in the natural basis $\Rnum^{\Npred}$
      at each time instance.
      Thus, the cost of evaluating the gradient 
      $ \nabla_{\design}{ \Psi^\GA(\design) } $ is 
      $\Oh{ 2\, r\, \nobstimes \Npred + 2\, \nobstimes \Npred } 
      = \Oh{ r\, \nobstimes \Npred }$.
      Similarly, in the case of spatiotemporal correlations, evaluating the
      gradient~\eqref{eqn:spatiotemporal_correlations_gradient} 
      costs 
      $\Oh{\Nsens\,\nobstimes ( 4\, r\ +  4 ) } = \Oh{  4\, r\, \nobstimes \,  \Nsens }$ 
        evaluations of the forward
      model $\F$.

  \subsection{Efficient computation of OED objective and gradient}
  \label{subsec:randomization_OED}
    Solving the OED optimization problem requires repeated evaluation of the
    trace of the posterior covariance matrix. Moreover, as discussed
    in~\Cref{subsec:computational_cost},
    constructing the gradient of the optimality criterion requires many forward and backward evaluations
    of the numerical model, with the computational cost dominated by the cost of 
    evaluating the Hessian matrix $\Hessmat$.
    To reduce  the computational cost, we approximate the
    Hessian 
    by a randomized approximation of the
    eigenvalues of a Hermitian matrix as described
    in~\cite{saibaba2016randomizedHermit}; see~\Cref{supp:subsec:reduced_Hessian} 
    for details.
    Note that the Hessian matrix $\Hessmat$ is never constructed in 
    practice; alternatively, only the effect of its inverse on a vector 
    $\Hessmat\inv\x$ is required.
   
    Here we discuss using randomized approaches to approximate the optimality criterion 
    for efficient calculation of both the objective and its gradient.
    Specifically, we utilize a randomized  approximation of the matrix trace 
    for A-optimal designs, and we defer the discussion 
    of randomized approximation of the D-optimality criterion to future work.

    Given a covariance matrix $\mat{C}$, 
    one can estimate its trace following a 
    Monte Carlo approach, using the relation
    $ 
      \Trace{ \mat{C} } = \Expect{}{\vec{z}\tran\mat{C}\vec{z}} 
        \approx \sum_{r=1}^{ n_r } { \vec{z}_r\tran \mat{C} \vec{z}_r }  \,,
    $ 
    where $\{ \vec{z}_r \}_{i=1,2,\ldots,{n_r}}$ are sampled from 
    the distribution of $\vec{z}$, which is generally an i.i.d. random variable 
    that follows a specific probability distribution.
    The most commonly used are Gaussian and 
    Rademacher distributions~\cite{avron2011randomized}.
    The development of the criterion below is independent from the choice of 
    the probability distribution. In the numerical experiments, however, 
    we resort to the Hutchinson trace estimator where the samples $\vec{z}_i$ 
    are drawn from the Rademacher distribution.

    An approximate A-optimality criterion 
    $\widetilde{\Psi}^\GA (\design) \approx \Psi^\GA(\design)$
    takes the form
    \begin{equation}\label{eqn:randomized_A_optimality}
      \widetilde{\Psi}^\GA (\design) 
        \! = \! \frac{1}{n_r} \! \sum_{r=1}^{ n_r } 
          {\vec{z}_r\tran \Cpredpostmat(\design) \vec{z}_r }
        \! = \! \frac{1}{n_r} \! \sum_{r=1}^{ n_r } 
          \vec{z}_r \tran \Predmat \!
            \left( \F \, \wdesignmat(\design)   \, \Fadj 
              \!+\! \Cparampriormat\inv 
            \right)\inv \!\Predmat\adj \vec{z}_r \,,
    \end{equation}
    with $\vec{z}_r \in \Rnum^{\Npred} $, and $\wdesignmat(\design) $ is given
    by~\eqref{eqn:Schur_weighted_joint_likelihood}. 
    The derivative of this randomized A-optimality criterion with respect to the
    $i$th design variable, where $i=1, 2,\ldots, \Nsens $, is
    \begin{equation}
      \del{ \widetilde{\Psi}^\GA (\design) }{\design_i}
        = \frac{1}{n_r} \sum_{r=1}^{ n_r } 
          \vec{z}_r\tran \Predmat \Hessmat\inv(\design) 
          \Fadj \wdesignmat(\design) \left( \Cobsnoise \odot \del{ \designmat(\design) }{\design_i} \right)  \wdesignmat(\design)\F
          \Hessmat\inv(\design) \Predmat\adj \vec{z}_r \,.
    \end{equation}

    If we assume that the observations are temporally uncorrelated, the gradient takes 
    the following form (see~\Cref{app:A_Optimality_Trace} for details):
    \begin{equation}
      \nabla_{\design}{\widetilde{\Psi}^\GA} (\design)
        = \frac{2}{n_r}  
        \sum_{r=1}^{ n_r } \sum_{m=1}^{\nobstimes}
            \vec{\psi}\adj_{r,m} \odot \left( 
              \left( \mat{R}_m \odot \designmat^{\prime} \right)  
            \vec{\psi}_{r,m}\right)  \,; \quad 
      \begin{aligned}
        \vec{\psi}_{r,m} &:= \mat{V}_m^{\dagger}(\design) \F_{0, m} \Hessmat\inv(\design) \Predmat\adj
        \vec{z}_r\,, \\
        \vec{\psi}\adj_{r,m} &:= \left( \vec{z}_r\tran \Predmat \Hessmat\inv(\design) 
             \Fadjind{0}{m} \mat{V}_m^{\dagger}(\design) \right)\tran \,.
      \end{aligned}
    \end{equation}

    In the presence of spatiotemporal correlations, the gradient is
    (see~\Cref{app:A_Optimality_Trace})
    \begin{equation}
      \nabla_{\design}{\widetilde{\Psi}^\GA(\design)}
        = 2 \sum_{r=1}^{ n_r }\!
        \sum_{i=1}^{\Nsens} \! 
          \sum_{m=1}^{\nobstimes}\!
            \vec{e}_i \vec{\psi}_r\adj 
              \vec{e}_q \left( 
                \left( \Cobsnoise \vec{e}_q\right)
                  \!\odot\! \vec{\vartheta}_{i, m} 
              \right) \tran  
          \vec{\psi}_r  \,; \quad
      \begin{aligned}
        \vec{\psi}_r &:= \wdesignmat(\design) \F \Hessmat\inv(\design) \Predmat\adj \vec{z}_r \,, \\
        \vec{\psi}_r\adj &:= \vec{z}_r\tran \Predmat \Hessmat\inv(\design) \Fadj
        \wdesignmat(\design) \,.
      \end{aligned}
    \end{equation}

\section{Numerical Experiments}
\label{sec:numerical_experiments}
We use an advection-diffusion model to simulate the transport of a contaminant
field $\xcont(\vec{x}, t)$ in a closed domain $\domain$. 
A set of candidate sensor locations are predefined, at which the contaminant is to be 
measured at specific locations in the domain at predefined time instances.
The end goal is to 
predict the concentration of the contaminant at a future
time instance, beyond the simulation time. 
To achieve this goal, first we need to find the \textit{optimal} subset of sensors, 
from the candidate locations, to deploy a small number of sensors.
Here, we define optimality in the sense of A-optimal designs.
%

\subsection{Problem setup}
\label{subsec:numerical_setup}
In this section we describe in detail the setup of the 
numerical experiments carried out in this work.

\subsubsection{Model: advection-diffusion}
The governing equation of the contaminant $\xcont(\vec{x}, t)$
 is 
  \begin{equation}\label{eqn:advection_diffusion}
    \begin{aligned}
      \xcont_t - \kappa \Delta \xcont + \vec{v} \cdot \nabla \xcont &= 0     
        \quad \text{in } \domain \times [0,T],   \\
      \xcont(x,\,0) &= \theta \quad \text{in } \domain,  \\
      \kappa \nabla \xcont \cdot \vec{n} &= 0  
        \quad \text{on } \partial \domain \times [0,T],
    \end{aligned}
\end{equation}
where $\kappa>0$ is the diffusivity, $T$ is the simulation final time, and
$\vec{v}$ is the velocity field. 
The domain $\domain$, sketched in~\Cref{fig:AD_Domain_Velocity}(left), 
is the region $(0, 1) \times (0, 1)$, where  the rectangular regions 
model two buildings $\mathbf{B}_1,\, \mathbf{B}_2$, respectively. 
The interior of these rectangular regions is excluded from 
the domain $\domain$, and the contaminant is not allowed to enter. 
The boundary $\partial \domain$ includes both the external boundary 
and the building walls.
The velocity field $\vec{v}$ is obtained by solving a steady Navier--Stokes
equation, with the side walls driving the flow, as detailed in~\cite{PetraStadler11},
and is shown in~\Cref{fig:AD_Domain_Velocity}(middle).
We consider $\kappa>0$ and $\vec{v}$ to be known exactly.
\begin{figure}[!ht]\vspace{-5pt}
  \centering
  \includegraphics[width=0.8\linewidth]{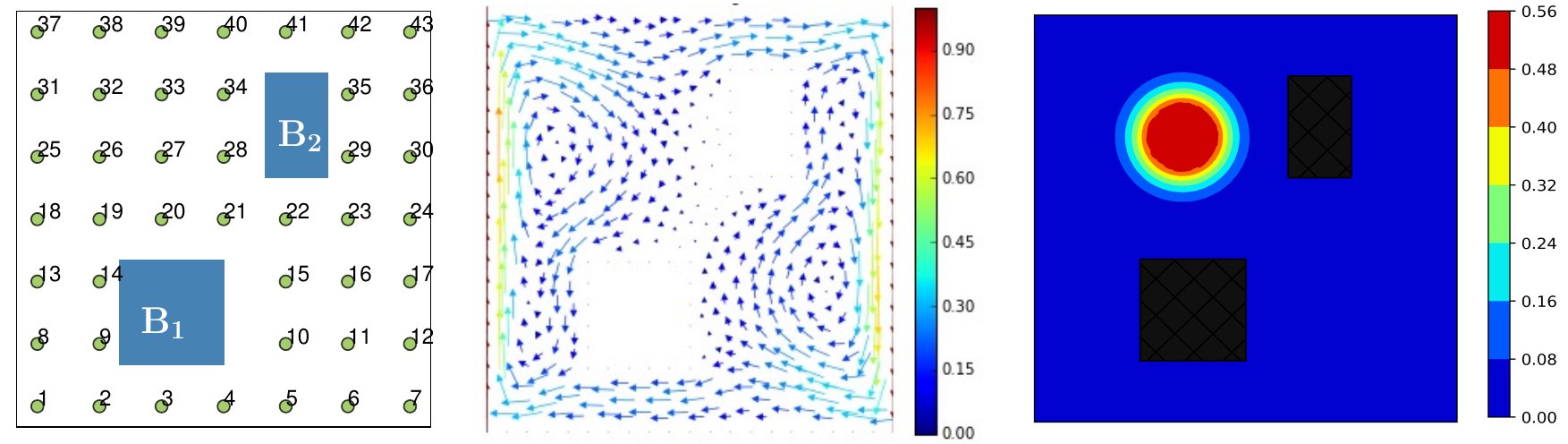}
  \caption{Advection-diffusion model domain, candidate sensor locations,
    velocity field, and ground truth of the
    model parameter. 
    Left: the physical domain $\domain$, including outer boundary, 
    and the two buildings $\mathbf{B}_1,\, \mathbf{B}_2$.
    The small circles indicate candidate sensor locations; the sensors are
    annotated based on their ordering in an observation vector.
    Middle: the velocity field, which is assumed to be fixed over time.
    Right: ground truth of the model parameter, in  other words, the true value of the model initial condition $\iparam$.
    }
  \label{fig:AD_Domain_Velocity} 
\end{figure}

  In this work we develop an implementation that utilizes 
  the package HippyLib~\cite{VillaPetraGhattas2016}, which provides
  the forward operator $\F$, the adjoint $\Fadj$, and a reduced-order 
  Hessian approximation~\cite{AlexanderianPetraStadlerEtAl14};
  see~\Cref{supp:subsec:reduced_Hessian} for additional details on the
  reduced-order Hessian approximation.
  We follow the approach in~\cite{Bui-ThanhGhattasMartinEtAl13}, 
  to discretize and simulate the flow PDE~\eqref{eqn:advection_diffusion}. 
  Specifically, we use finite elements, with Lagrange triangular elements 
  of order $2$, with $\Nstate\!=\!7863$ spatial degrees of freedom in space, and 
  using implicit Euler for simulation in time.
  The true initial model parameter (shown in~\Cref{fig:AD_Domain_Velocity}(right)) is used 
  to create a reference trajectory and synthetic observations.

\subsubsection{Observations and observation noise}
\label{subsubsec:obs_noise}
  The contaminant is observed at a set of (candidate) observation gridpoints 
  uniformly distributed in the space domain 
  $\{ \vec{x}_1,\, \ldots,\, \vec{x}_{\Nsens} \}\!\subset\!\domain$,
  at fixed time instances 
  $\{ t_1,\,  t_2,\, \ldots,\, t_{\nobstimes} \} \subset [0, T]$. 
  The candidate sensor locations, shown in~\Cref{fig:AD_Domain_Velocity}(left),
  constitute the observational grid, with $\Nsens = 43$ candidate sensor locations.
  The observation operator $\ObsOperind{k}$ here is a restriction operator applied 
  to the solution $u(\vec{x}, t)$ to extract solution values (i.e., concentration 
  of the contaminant) at the predefined observation gridpoints at any time 
  instance $t_k$.
  Here the observation times are $t_1\!+\! s \Delta t$, with initial observation
  time $t_1\!=\!1$;  $\Delta t\!=\!0.2$ is the model simulation time step; 
  and $s=0, 1, \ldots, 5$, resulting in $\nobs=6$ observation time instances.
  The matrix representation of the observation operator
  $\ObsOperind{k}$ at any time instance $t_k$ is plotted 
  in~\Cref{fig:Obs_Space_Covariance}(left).
  \begin{figure}[!ht]\vspace{-10pt}
    \centering
    \includegraphics[width=0.45\linewidth]{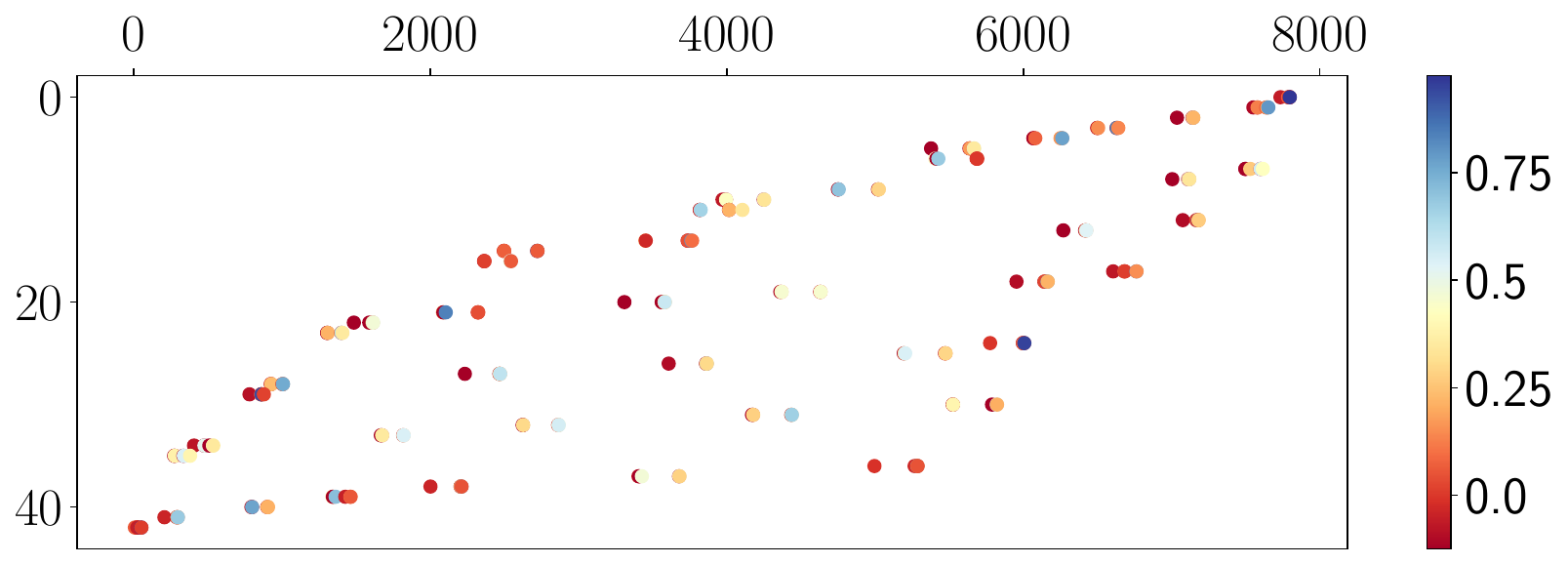}
    \quad
    \includegraphics[width=0.19\textwidth]{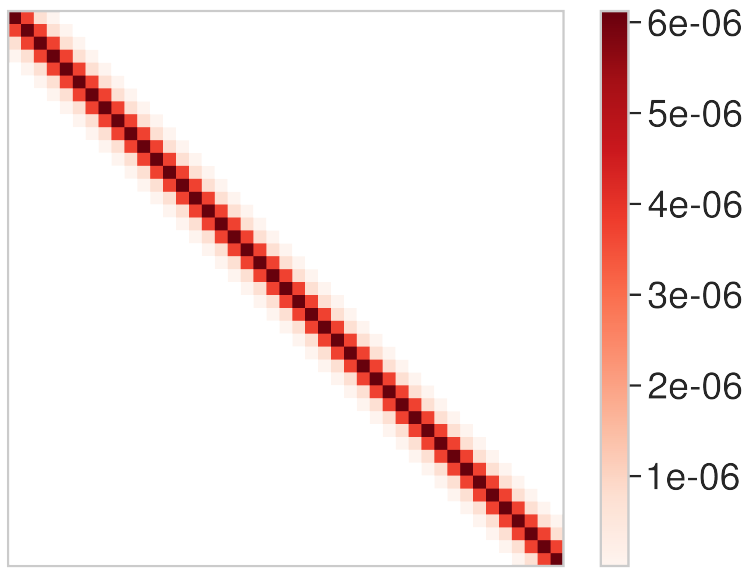}
    \quad
    \includegraphics[width=0.19\textwidth]{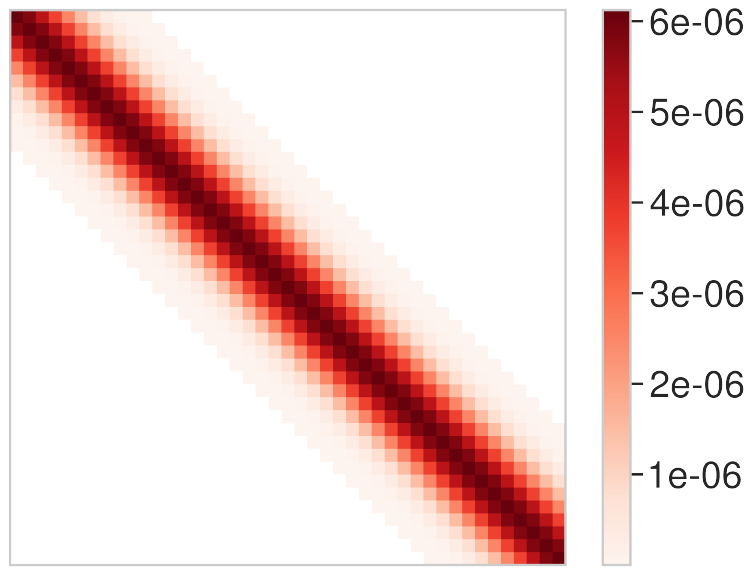}
    \caption{
      Observation operator $\ObsOperind{k}$ 
      and observation error covariance matrix $\mat{R}_k$ at 
      any observation time instance $t_k$.
      To construct synthetic observation error covariances $\mat{R}_k$, 
      we use a Gaspari--Cohn function~\eqref{eqn:Gaspari_Cohn} with a predefined
      length scale $\ell$.
      Left: space observation operator $\ObsOperind{K}$.
      Middle: observation error covariance matrix $\mat{R}_k$ 
      with length scale of $\ell\!=\!1$.
      Right: observation error covariance matrix $\mat{R}_k$
      with length scale of $\ell\!=\!3$.
      }
    \label{fig:Obs_Space_Covariance}
  \end{figure}

  In the numerical experiments we use synthetic observations and 
  consider two cases of observation correlations.
  First, we assume that observations are uncorrelated, 
  with fixed standard deviation $\sigma_{\rm obs}\!=\! 2.475\!\times\! 10^{-2}$. 
  This value is obtained by considering a noise level of $5\%$ of the maximum 
  concentration of the contaminant at all observation points, found 
  from a reference simulation, that is, the solution 
  of~\eqref{eqn:advection_diffusion}, over the simulation timespan $[0,T]$,
  with the initial condition (\Cref{fig:AD_Domain_Velocity}(right)).
  Second, we consider the case of space correlations where we run two sets of 
  experiments with block diagonal covariance matrices, 
  whose structure is shown in~\Cref{fig:Obs_Space_Covariance}.
  These are the covariances $\mat{R}_k$ between observation gridpoints at any 
  observation time $t_k$ and are constructed by using the fifth-order 
  piecewise-rational function~\eqref{eqn:Gaspari_Cohn}, with covariance
  length scales $\ell\!=\!1,\, \ell\!=\!3$, respectively. 
  We use~\eqref{eqn:Gaspari_Cohn} to create correlations
  $\decorrcoeff_{i,j}$ between $i$ and $j$ entries of the correlation matrix 
  and scale it by observation noise variance $\sigma_{\rm obs}^2$ to create 
  a spatial covariance model, as shown in~\Cref{fig:Obs_Space_Covariance}.
  The row/column ordering of the covariance matrix corresponds to the entries of
  an observation vector; see~\Cref{fig:AD_Domain_Velocity}(left).
  The spatial covariance is stationary and is fixed over time; that is, 
  $\mat{R}_k\!=\!\mat{R}\,,~ \forall k$.

\subsubsection{Forward and adjoint operators}\label{subsubsec:forward_adjoint_F}
  The forward operator $\F$ is evaluated by solving~\eqref{eqn:advection_diffusion} 
  over the simulation timespan, and the observation operator is applied 
  to the solution to extract solution values at the sensor locations 
  at the predefined observation time instances.
  We use a space-time formulation where the model states at all simulation 
  time points are stacked in one vector.
  A forward operator $\Find{0}{k}\equiv\Ftind{0}{k}$ maps the discretized model 
  parameter $\iparam$ to the measurements at observational gridpoints at time
  instance $t_k$.
  If we define $\Solind{0}{k}\equiv\Soltind{0}{k}$ to be the solution operator 
  over the interval $[t_0, t_k]$, and since $\ObsOperind{k}$ is the observation 
  operator at time instance $t_k$, then 
  $\Find{0}{k} \iparam = \ObsOperind{k}\Solind{0}{k} \iparam$. 
  In the experiments here, we use the same observational grid at all time 
  instances; that is, $\ObsOperind{k}=\ObsOperind{},\,\forall k$.
  Note that $\ObsOperind{}$ here is an interpolation operator from the model 
  gridpoints to observational gridpoints.

  A space-time observation vector is defined as 
  $\obs=\begin{bmatrix}\obs_1\tran, \obs_2\tran, \ldots, \obs_{\nobstimes}\tran\end{bmatrix}\tran$, where
  $\obs_k\equiv\obs[t_k]$ represents the sensor measurements at time instance
  $t_k$, with entries ordered as shown in~\Cref{fig:AD_Domain_Velocity}(left).
  The observation error covariance matrix $\Cobsnoise$ is defined accordingly 
  (see~\Cref{subsec:OED_sensor_placement} for more details). 
  Given this space-time formulation, we can rewrite the forward problem~\eqref{eqn:forward_problem} 
  as 
  \begin{equation}\label{eqn:forward_problem2}
    \begin{bmatrix}\obs_1\tran, \obs_2\tran, \ldots, \obs_{\nobstimes}\tran \end{bmatrix}\tran 
      = \begin{bmatrix}\left(\Find{0}{1} \iparam \right) \tran, \left(\Find{0}{2} \iparam\right)\tran, \ldots, \left(\Find{0}{\nobstimes} \iparam\right)\tran \end{bmatrix}\tran 
      + \vec{\delta}\,,\quad \vec{\delta} \sim \GM{\vec{0}}{\Cobsnoise} \,.
  \end{equation}

  When the observation noise is temporally uncorrelated, 
  the observation error covariance matrix $\Cobsnoise$ is 
  block diagonal, and the forward model can be written in the form
  \begin{equation}\label{eqn:forward_problem3}
    \obs_k = \Find{0}{k} \iparam  + \vec{\delta}_k\,,\quad 
      \vec{\delta}_k \sim \GM{\vec{0}}{\Cobsnoisemat_k} \,, \quad
      k=1,2,\ldots,\nobs \,.
  \end{equation}
  This formulation~\eqref{eqn:forward_problem3} of the forward problem 
  is advantageous for large-scale
  time-dependent problems because \textit{checkpointing} is essential for 
  scalability of the solution of both the Bayesian inverse problem and the OED.
  Forward and adjoint simulations scale efficiently by checkpointing the model 
  solution at observation time instances $t_k$ and by utilizing a model solution 
  operator $\Solind{k-1}{k}\equiv \Soltind{k-1}{k}$ to propagate the 
  model state over the simulation window $[t_{k-1}, t_k]$.
  %

  For a linear operator $\F$, the adjoint $\Fadj$ satisfies the property
  $\ip{\F\vec{x}}{\vec{y}}=\ip{\vec{x}}{\Fadj\vec{y}}$, where $\vec{x},\,\vec{y}$ 
  are elements of the space on which the product is defined.
  If $\F$ is defined on the Euclidean space equipped with inner product 
  $\ip{\vec{x}}{\vec{y}}:=\vec{x}\tran\vec{y}$, then 
  the adjoint of the forward operator is equal to the matrix transpose of 
  the discretized forward operator; that is, $\Fadj=\F\tran$.
  Here,  however, we use a finite-element discretization of the linear Bayesian 
  inverse problem where the underlying infinite-dimensional problem is 
  formulated on the space equipped with $L^2(\D)$ inner product.
  As explained in~\cite{Bui-ThanhGhattasMartinEtAl13}, the model adjoint is defined 
  by using the Euclidean inner product weighted by the finite-element mass matrix 
  $\mat{M}$. 
  Specifically, for $\vec{u} \in \Rnum^{\Nparam},\, \vec{v} \in \Rnum^{\Nobs}$, 
  it follows that 
  $\ip{\mat{F}\vec{u}}{\vec{v}} = (\mat{F}\vec{u})\tran \vec{v}
        = \vec{u}\tran \mat{F}\tran \vec{v} =
          \vec{u}\tran \mat{M} \mat{M}\inv \mat{F}\tran \vec{v}
        = \wip{\vec{u}}{\mat{M}\inv\mat{F}\tran\vec{v}}{\mat{M}}$, 
  resulting in the model adjoint $\F^* = \mat{M}\inv\mat{F}\tran$.
  In the time-dependent settings utilized here, the adjoint takes the form
  $\Fadjind{0}{k}=\mat{M}\inv \Solind{0}{k}\tran \ObsOperind{k}\tran$.

\subsubsection{Goal operator and its adjoint}\label{subsubsec:forward_adjoint_P}
The QoI  $\pred$ is the value of contaminant concentration predicted 
around the second building $\mathbf{B}_2$  (see~\Cref{fig:AD_Domain_Velocity}) 
at a future time instance $t_p$, beyond the simulation time. 
Specifically, we aim to predict the contaminant concentration withing a specific 
distance $\epsilon$ from the walls (i.e., boundary) of that building.
For that, we set $\epsilon\!=\!0.02$ and $t_p\!=\!2.2$, which results in 
a prediction vector of size $138$.
If we define $\mat{C}_p$ to be the matrix representation of a restriction operator 
that projects the solution at time $t_p$ onto the prediction gridpoints 
(shown in~\Cref{fig:PredGrid_True_Pred_and_Prior}(left)), 
then the goal operator is defined as $\Predmat=\mat{C}_p\, \Solind{0}{p}$.
Since we are using finite-element discretization, the adjoint of the goal operator 
is defined as 
$ \Predmat\adj = \mat{M}\inv \Solind{0}{p}\tran \, \mat{C}_p\tran\,$.

\subsubsection{The prior}
\label{subsubsec:prior}
  Following the setup in~\cite{attia2018goal,PetraStadler11}, 
  the prior distribution of the parameter $\iparam$ is $\GM{\iparb}{\Cparampriormat}$, 
  with $\Cparampriormat$ being a discretization of $\mathcal{A}^{-2}$, 
  where $\mathcal{A}$ is a Laplacian operator. 
  In particular, we use 
  $ \mathcal{A} = \delta I + \gamma  \Theta \nabla$, 
  where $\Theta$ is a symmetric positive definite tensor for anisotropic
  diffusion of the PDE~\eqref{eqn:advection_diffusion}, 
  $\delta\gamma$ governs the variance of the prior samples, and
  the ratio $\frac{\gamma}{\delta}$ governs the correlation length
  scale~\cite{AlexanderianPetraStadlerEtAl14,VillaPetraGhattas2016}.
  In our experiments we set $\gamma = 1$ and $\delta = 16$.
  The ground truth and the prior QoI (shown
  in~\Cref{fig:PredGrid_True_Pred_and_Prior}(middle)) 
  are obtained by applying the prediction 
  operator $\Predmat$ to the true and the prior initial condition, respectively.
  That is, $\pred_{\rm true}=\Predmat \iparam_{\rm true}$, and 
  the goal QoI $\pred$ assumes a Gaussian prior 
  $\GM{\Predmat\iparb}{\Predmat \Cparampriormat \Predmat\adj}$. The prior
  covariance matrix of the goal QoI, that is, $\Predmat \Cparampriormat
  \Predmat\adj$, is displayed in~\Cref{fig:PredGrid_True_Pred_and_Prior}(right).
  The effect of the prior on the resulting design is an important issue; 
  however, it is out of the scope of this work. Here we fix the prior across
  all experiments and focus on the performance due to observation covariance
  weighting.

  \begin{figure}[!ht]\vspace{-5pt}
    \centering
    \includegraphics[width=0.21\linewidth ,valign=t]{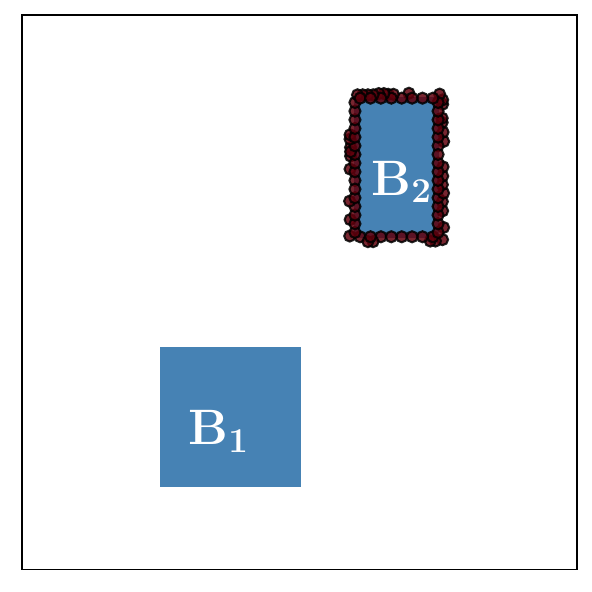}
    ~
    \includegraphics[width=0.47\textwidth,valign=t]{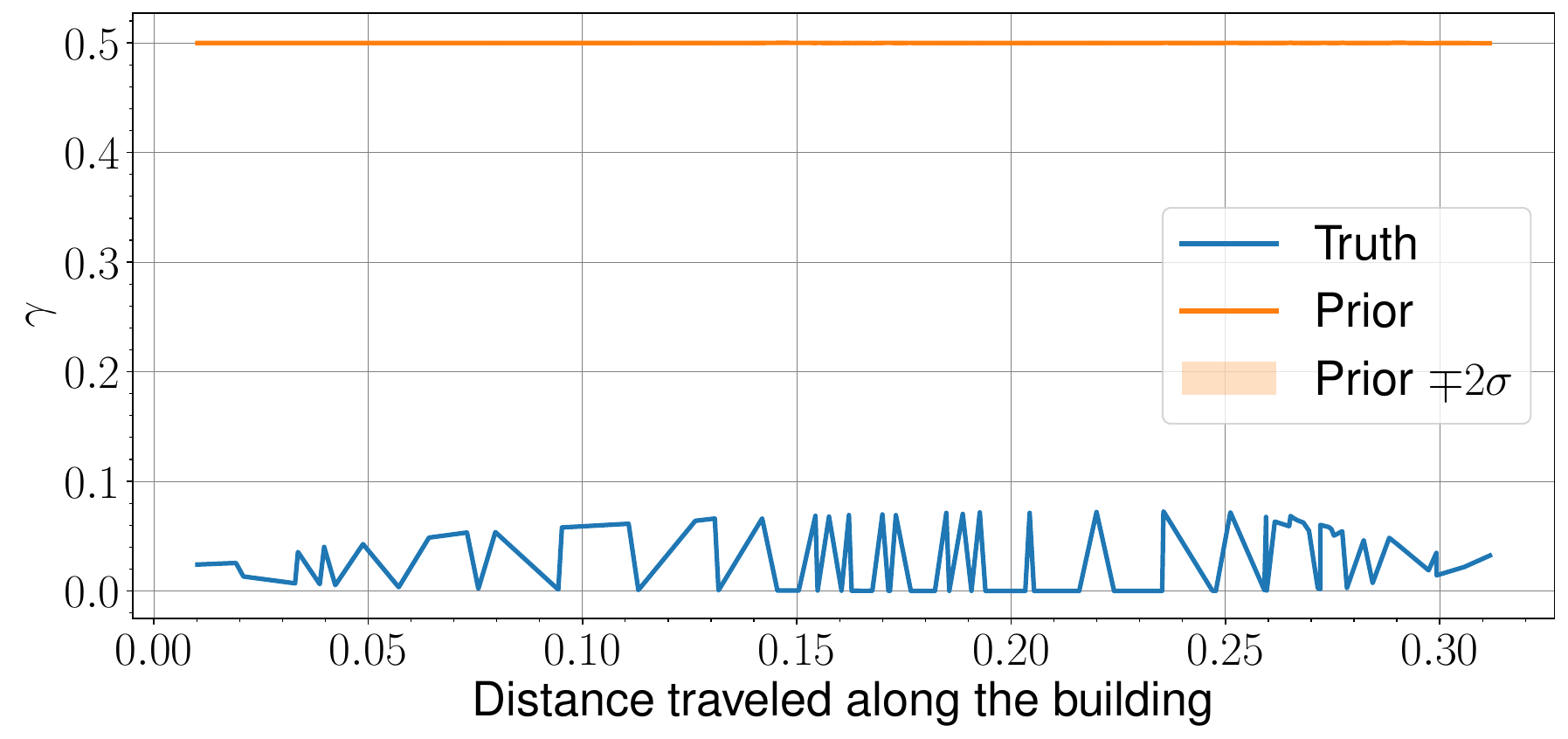}
    ~
    \includegraphics[width=0.24\textwidth ,valign=t]{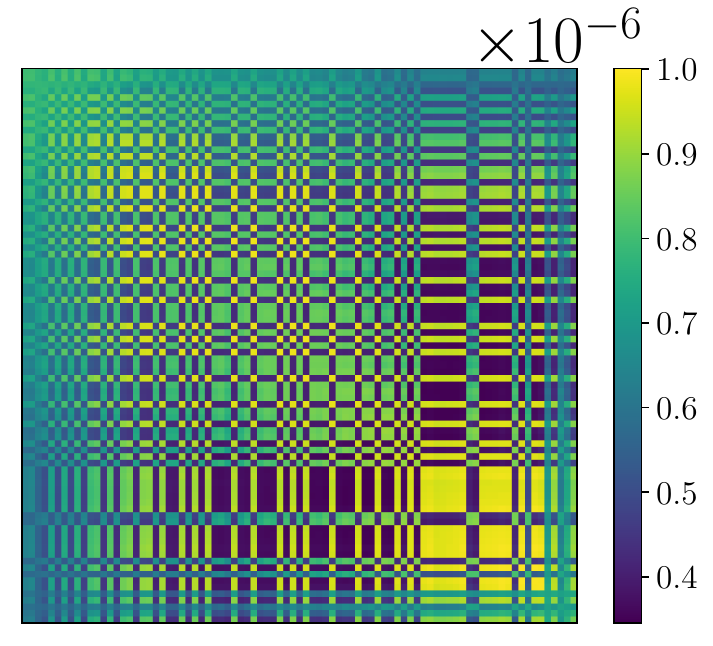}
    \caption{
      Left: gridpoints around the second building $\mathbf{B}_2$, at which 
        the goal QoI $\pred$ is predicted.
      Middle: ground truth and the prior of the goal QoI $\pred$.
      Here, the x-axis of the middle panel corresponds to the prediction gridpoints
      shown in the left panel reordered based on the distance traveled along the
      second building starting from the lower-left corner.
      Right: entries of the prior covariance matrix $\Predmat \Cparampriormat
      \Predmat\adj$, with rows/columns reordered to match the ordering on the
      x-axis of~\Cref{fig:PredGrid_True_Pred_and_Prior}(middle). 
      }
    \label{fig:PredGrid_True_Pred_and_Prior} 
  \end{figure}

  In the numerical results discussed in~\Cref{subsec:numerical_results}, 
  we  use $\pred_{\rm True}$ to calculate an accuracy measure of the solution 
  of an inverse problem.
  Specifically, we use the 
  root-mean-squared error (RMSE) of the QoI retrieved 
  by Bayesian inversion $\pred$ defined as 
  \begin{equation}\label{eqn:RAE_RMSE}
      {\rm RMSE} := \frac{\wnorm{ \pred -\pred_{\rm true}}{2} }{\sqrt{\Npred}} \,. 
  \end{equation}
  %

  \subsubsection{OED regularization}
    We employ an $\ell_1$ regularization term to induce a sparse design.
    Specifically, given the specific choice of the design weighting function $\weightfunc$, 
    the regularization term and the associated derivative take the
    form 
    \begin{equation}
      \Phi(\design) 
        = \norm{\weightfunc(\design, \design)}_1 
        = \sum_{i=1}^{\Nsens}{\left|\weightfunc(\design_i, \design_i)\right|}
        =  \sum_{i=1}^{\Nsens}{\weightfunc(\design_i, \design_i)} \,;
        \quad
        \nabla_{\design} \Phi(\design) =
        \sum_{i=1}^{\nobs}{\del{\weightfunc(\design_i,\design_i)}{\design_i}}\vec{e}_i
        \,.
    \end{equation}
    %

\subsubsection{Optimization algorithms}
To solve the OED problem~\eqref{eqn:oed_relaxed_optimization_problem}, we use Python's optimization routines provided in Scipy. 
Specifically, we use the implementation of L-BFGS-B~\cite{byrd1995limited,scipyopt_lbfgsb_web}, 
provided by \textsf{fmin\_l\_bfgs\_b()} with a stopping criterion based on the maximum entry of the projected gradient.
We set the tolerance to $\mathsf{pgtol}=1\times10^{-5}$.
This allows us to set bound constraints when needed, for example, when the
kernel~\eqref{eqn:cov-prod_kernel} is invoked.
For all choices of the weighting function, the initial guess passed to the optimization algorithm is set to yield
weighting values $\weightfunc(\design_i,\design_i)=0.5;\, i=1,2,\dots,\Nsens$.
This means that all sensors are as likely to be activated as to be turned off.

\subsection{Numerical results}
\label{subsec:numerical_results}
We investigate several scenarios. 
The first case is where no observation correlations are considered; 
then we incorporate spatial correlations.
The correlations are synthetically generated by using a Gaspari--Cohn function with a predefined length scale.
Note that this choice is made for convenience to create synthetic covariances and is independent from the
weighting function $\rho$ used to control the temporal relative importance of
observation covariances; see~\Cref{subsec:spacetime_dependent_OED}. 
We use A-optimality as the main criterion for sensor placement in our numerical experiments. 
As discussed in~\Cref{sec:OED_Correlated_OBS}, we can use a randomized trace 
estimator to formulate the optimality criterion and the associated gradient. 
In our experiments we used the Hutchinson randomized trace estimator 
to approximate the A-optimality criterion, 
and we set the sample size to $n_r=25$. 
This enabled us to carry  out several comparative experiments both accurately and efficiently. 
Numerical results validating this assertion are given in~\Cref{supp:subsec:randomized_trace_estimation}.
For fair comparison, we used the same realizations of random vectors to calculate the
objective, namely, the optimality criterion, and the gradients in all experiments.

Note that while the A-optimality criterion is calculated based on the sum of
the posterior variances discarding posterior covariances, it does in fact
account for observation correlations, since posterior variances are influenced
by both prior and observation error covariances (see~\eqref{eqn:Posterior_Params}).
Since the case with temporal observation correlations is not applicable here and 
since D-optimality might be more suitable for handling spatiotemporal observation correlations, 
we will leave the investigation of this case to future studies.
%

  We are  interested mainly in the design of an observational grid that will optimally serve the Bayesian retrieval 
  of the goal QoI. Specifically, we are looking for the optimal subset of candidate sensor locations to deploy 
  for data collection.
  We start by showing the results of solving the Bayesian inverse problem, with a fully deployed observational grid.
  \Cref{fig:IP_Prediction_AllSensors} shows the ground truth, the prior, and the posterior prediction QoI $\pred$ 
  obtained by solving the Bayesian inverse problem with all sensors activated. 
  The results are shown for the case where no observation correlations are assumed, 
  that is, when the correlation length scale is  $\ell=0$, and two experiments with space correlations with $\ell$ 
  set to $1$ and $3$, respectively.
  This plot will serve as a visual benchmark as needed.
  \begin{figure}[!ht]\vspace{-5pt}
    \centering
    \includegraphics[width=0.60\linewidth]{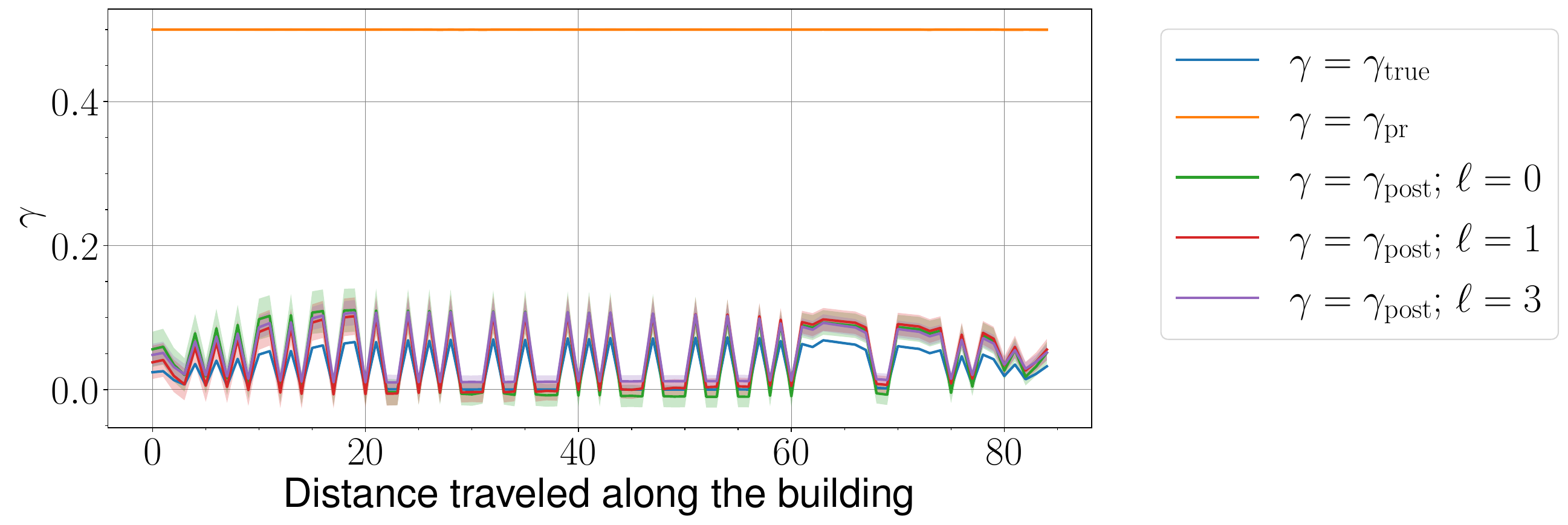}
    \caption{The truth, the  prior, and the posterior ($\pm2\sigma_{\rm post}$) prediction QoI $\pred$.
      The truth is based on the ground truth of the model initial condition.
      The prior QoI is based on the mean of the prior distribution for all scenarios.
      The posterior prediction QoI is obtained by using the solution of the inverse problem, 
      using all candidate sensors.
      The correlation length scale used in the Gaspari--Cohn function to construct correlations is indicated in the legends.
      The case when $\ell=0$ corresponds to the diagonal covariance matrix, that is,   no correlations, is assumed.
      }
    \label{fig:IP_Prediction_AllSensors}
  \end{figure}
  %

  The goal of solving an OED problem is to find a small subset of sensors (e.g.,  given a limited budget)
  that yields a posterior QoI as close as possible to the ground truth shown 
  in~\Cref{fig:PredGrid_True_Pred_and_Prior}(right), with minimum uncertainty.
  In what follows, we assume a maximum budget of $\budget$ observation sensors to be ideally placed in the domain. 
  The solution of the OED problem is expected to be sparse but might not be binary.
  A simple rounding method to obtain a binary design is used when needed, 
  by activating the sensors corresponding to the largest $\budget$ weights
  in the design resulting from solving an OED optimization problem. 
  In our experiments we show results for two choices of the budget $\budget$: 
  first $\budget=8$ and second 
  $\budget=\Nsens/2$, where we allow up to $50\%$ of the sensors to be activated.
  We discuss results obtained by solving the OED
  problem~\eqref{eqn:oed_relaxed_optimization_problem} with the A-optimality
  criterion $\Psi(\design)\!=\!\Psi^\GA$,
  using the product kernel~\eqref{eqn:cov-prod_kernel}, 
  the EXP kernel~\eqref{eqn:exp-prod_kernel}, and the
  logistic sigmoid~\eqref{eqn:sigmoid-prod_kernel} kernels.
  Experiments are carried out for multiple values of the regularization parameter $\alpha$, to study the 
  effect of the regularization term on both  sparsity and the order of relative importance of candidate sensors.
  We start with numerical experiments with uncorrelated observation errors, and then
  we show experiments with spatial observation correlations with multiple 
  length scales, as described in~\Cref{sec:numerical_experiments}.

  \subsubsection{No-correlations results}
  \label{subsec:no_corr_results}
    \Cref{fig:OED_No_Correlation_A_Optimal_Designs} shows 
    the optimal weights resulting from solving the standard OED 
    relaxation~\eqref{eqn:oed_relaxed_optimization_problem} with ad hoc
    fix~\eqref{eqn:standard_prepost_fixed} and the OED problem using the
    Schur-product formulation~\eqref{eqn:A_optimality_Schur} with the
    weighting matrix $\designmat(\design)$ defined using the 
    kernel~\eqref{eqn:cov-prod_kernel}, the EXP kernel~\eqref{eqn:exp-prod_kernel}, and the
    logistic sigmoid kernels~\eqref{eqn:sigmoid-prod_kernel}, respectively, 
    for multiple values of the penalty parameter $\alpha$.
    All formulations behave similarly and yield similar estimates of the
    optimal design. 
    Note that in this setup the sigmoid weighting kernel~\eqref{eqn:sigmoid-prod_kernel} leads to
    a sparser design  even with a
    penalty value $\alpha=0$, that is, without enforcing a sparsity-promoting
    penalty term.
    \begin{figure}[!ht]\vspace{-10pt}
    \centering
      \includegraphics[width=0.24\linewidth, trim=80 80 0 0, clip]{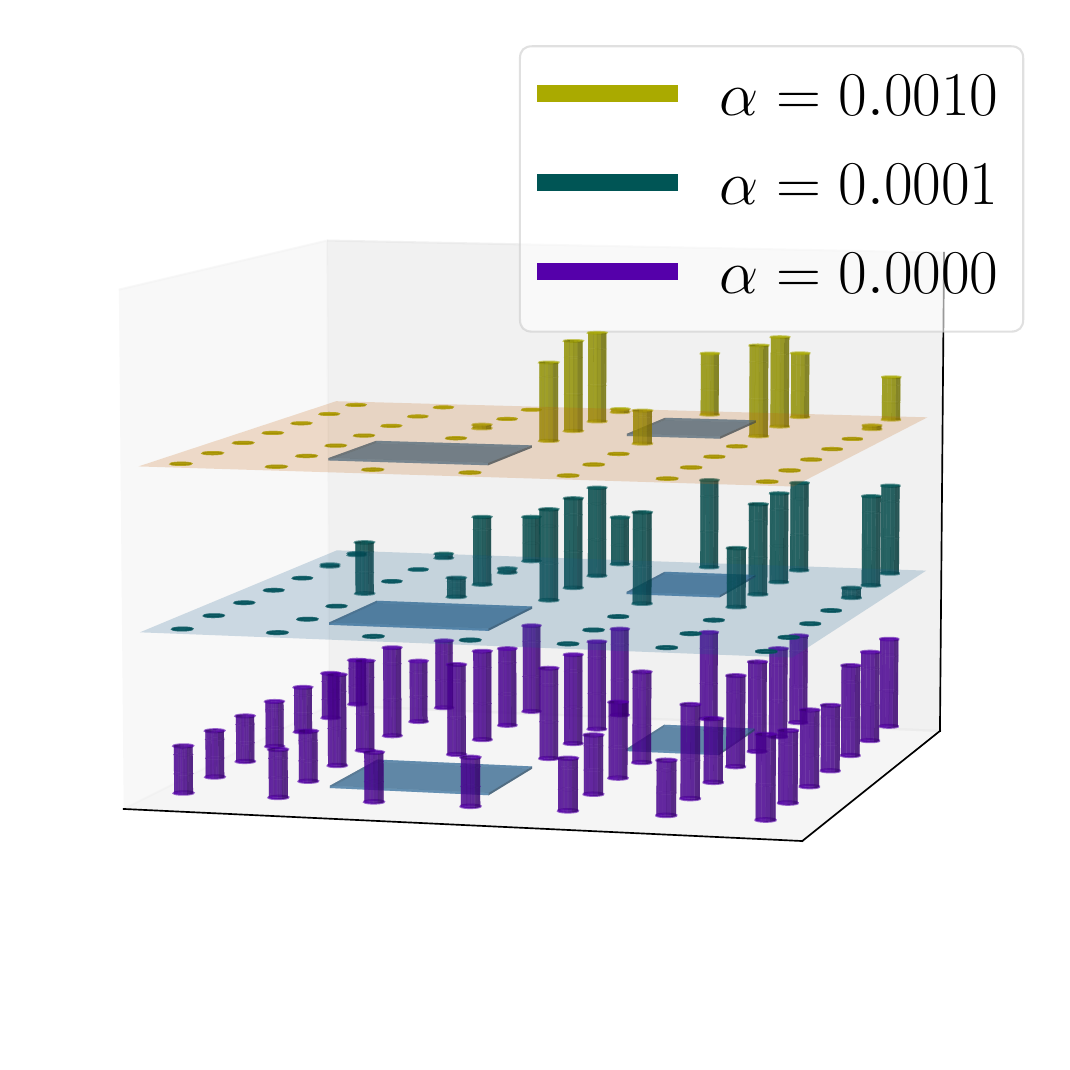} 
      \includegraphics[width=0.24\linewidth, trim=80 80 0 0, clip]{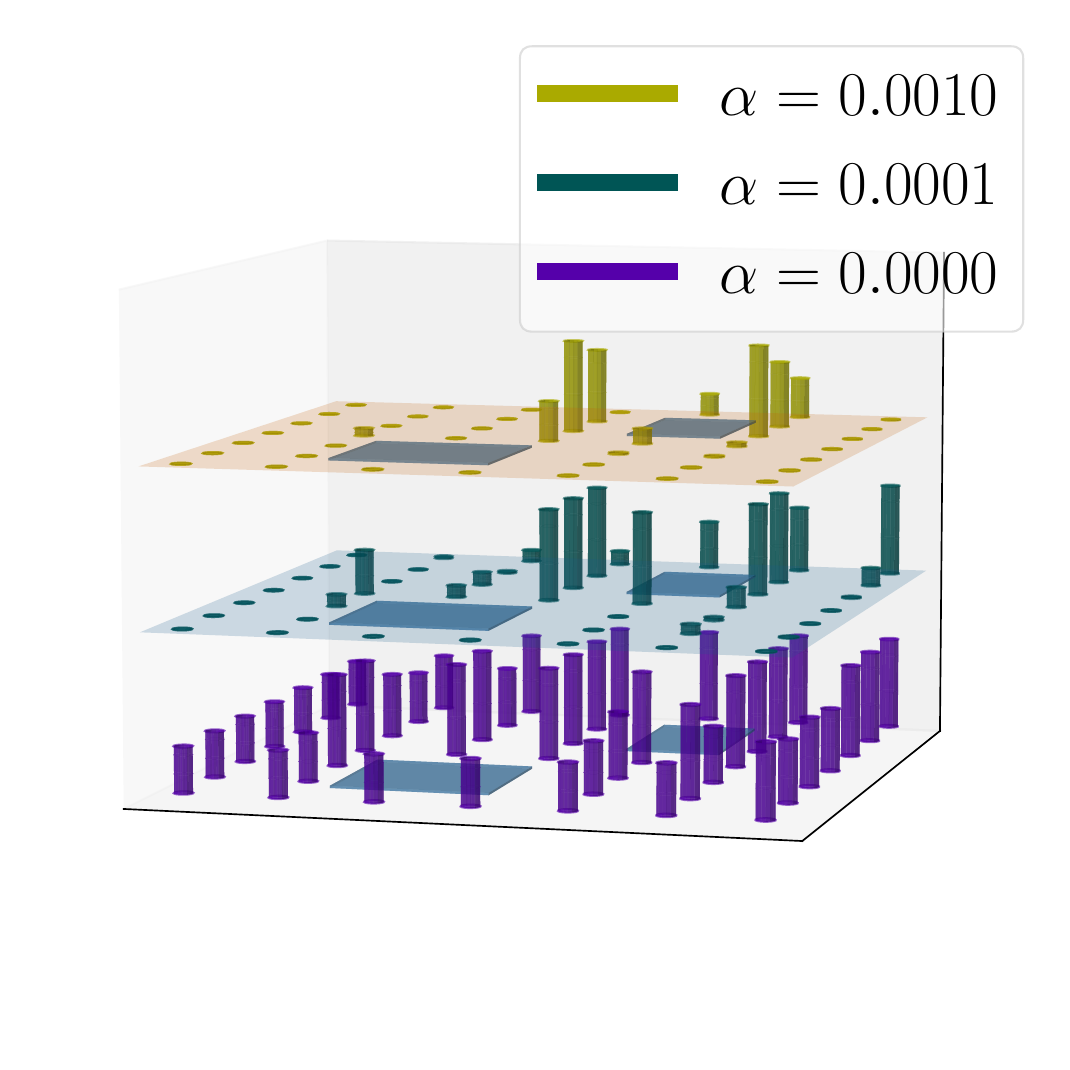}
      \includegraphics[width=0.24\linewidth, trim=80 80 0 0, clip]{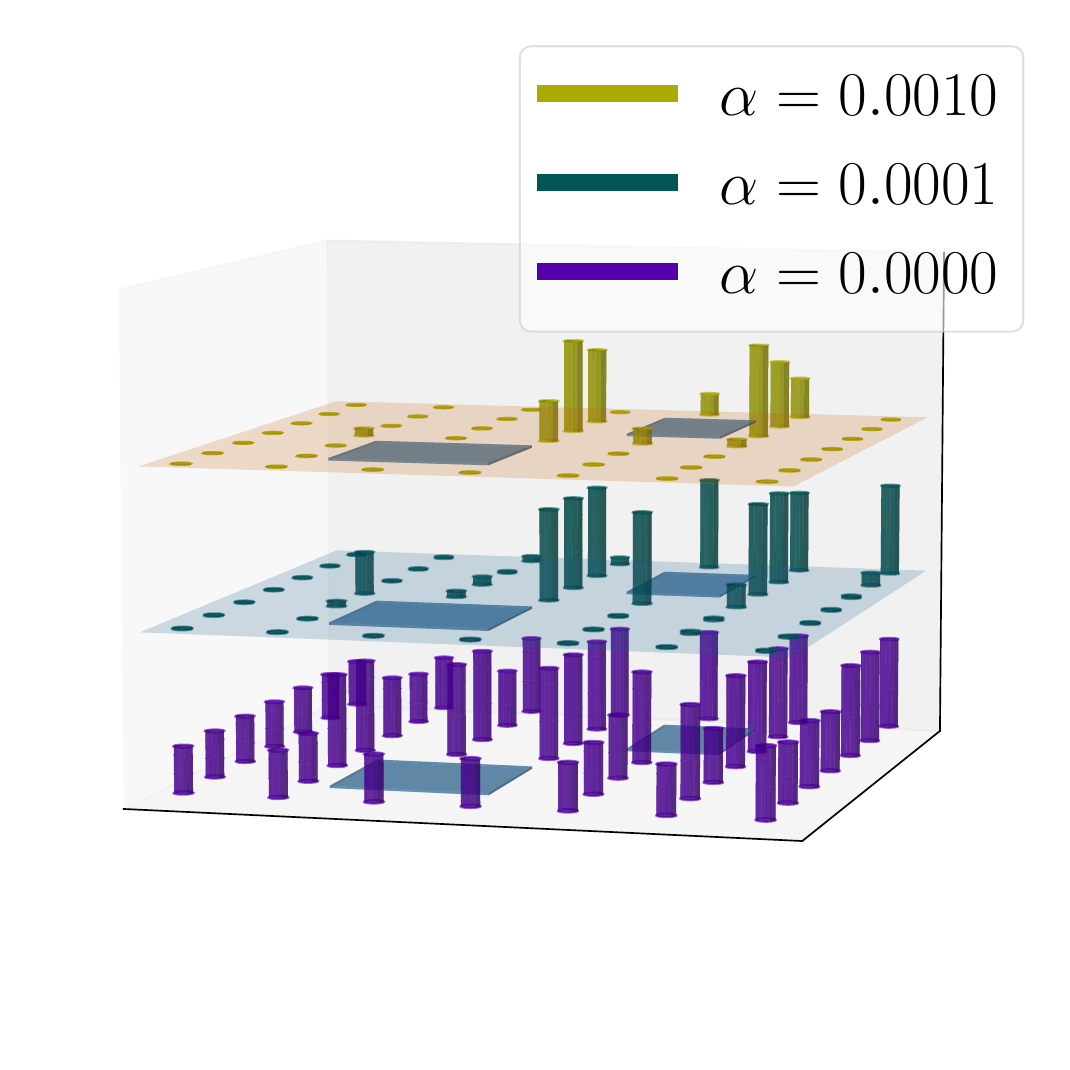}
      \includegraphics[width=0.24\linewidth, trim=80 80 0 0, clip]{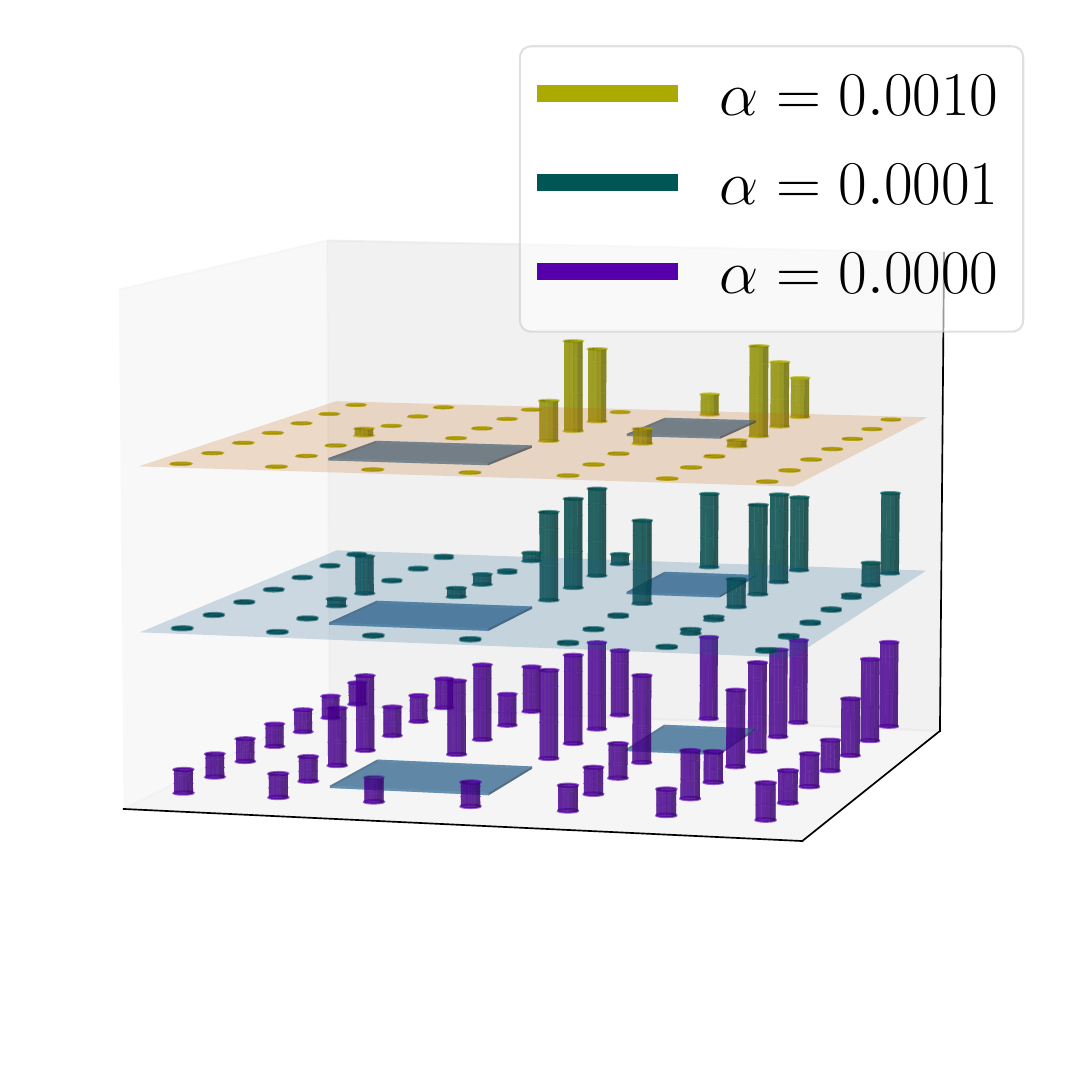}
      \caption{Optimal weights resulting from solving the OED problem~\eqref{eqn:oed_relaxed_optimization_problem}.
        Results are obtained (from left to right) by using the standard OED
        formalism~\eqref{eqn:oed_relaxed_optimization_problem} with ad hoc
        fix~\eqref{eqn:standard_prepost_fixed}, Schur-OED with weighting
        kernel~\eqref{eqn:cov-prod_kernel}, the EXP kernel~\eqref{eqn:exp-prod_kernel}, 
        and the logistic sigmoid kernels~\eqref{eqn:sigmoid-prod_kernel}, respectively, 
        for values of the parameter $\alpha$ set to $0,\, 1e\!-4,\, 1e\!-3$.
        The solution is plotted in the form of bars located on the grid at the corresponding 
        candidate sensor locations. 
        The height of each bar is set to the value of the weight obtained by
        solving~\eqref{eqn:oed_relaxed_optimization_problem}.
        The $z-$axis has a scale (limits) set to $[0, 1]$ for all subplots to
        fairly compare the values of weights resulting by solving the relaxed
        OED problem.
        }
      \label{fig:OED_No_Correlation_A_Optimal_Designs}
    \end{figure}

    To better inspect the difference between the resulting optimal designs shown 
    in~\Cref{fig:OED_No_Correlation_A_Optimal_Designs}, 
    we apply rounding to obtain a binary design by activating the maximum
    $\lambda=8$ sensors; the resulting binary designs are shown in~\Cref{fig:OED_No_Correlation_A_Optimal_Designs_Max8}.
    \begin{figure}[!ht]\vspace{-10pt}
    \centering
      \includegraphics[width=0.24\linewidth, trim=80 80 0 0, clip]{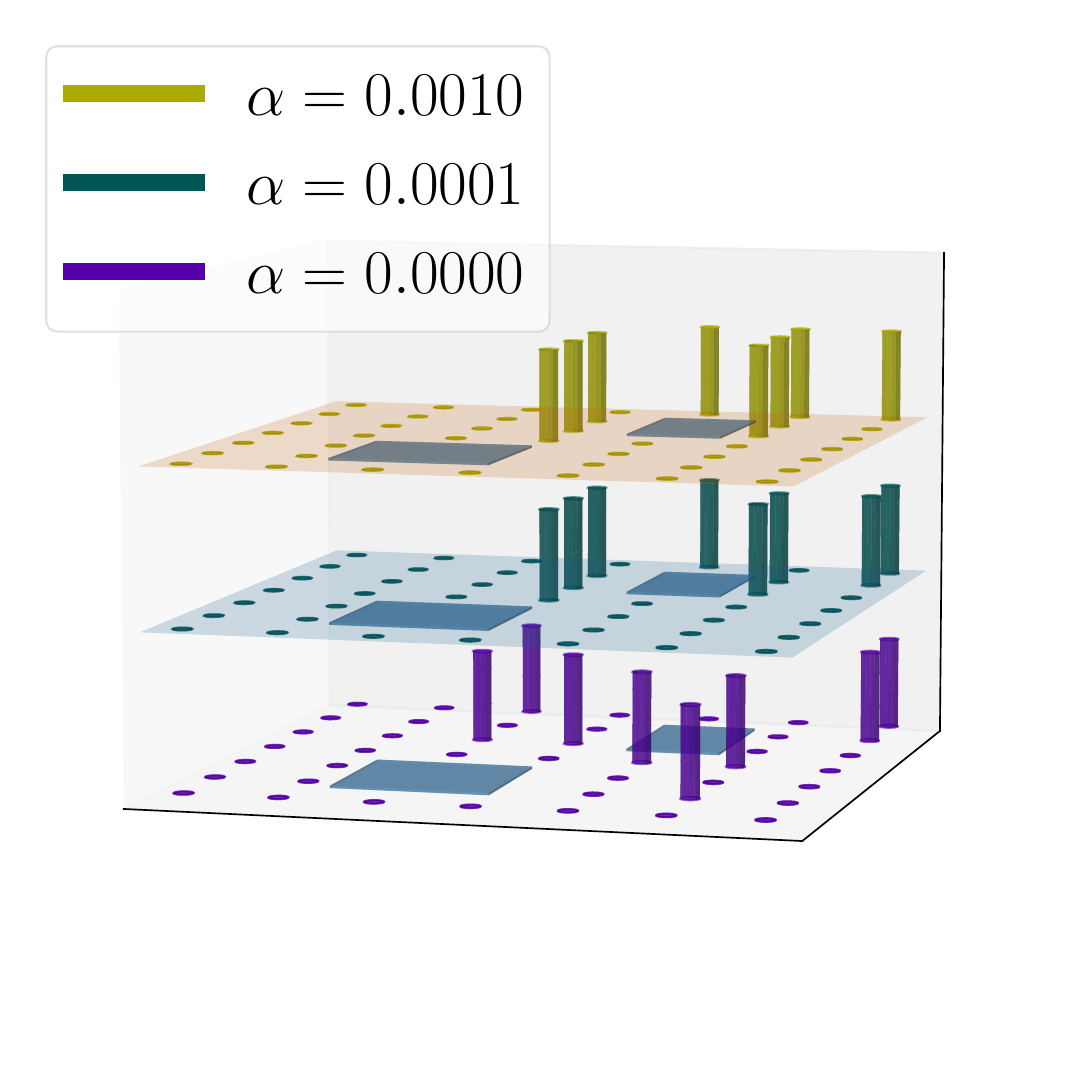} 
      \includegraphics[width=0.24\linewidth, trim=80 80 0 0, clip]{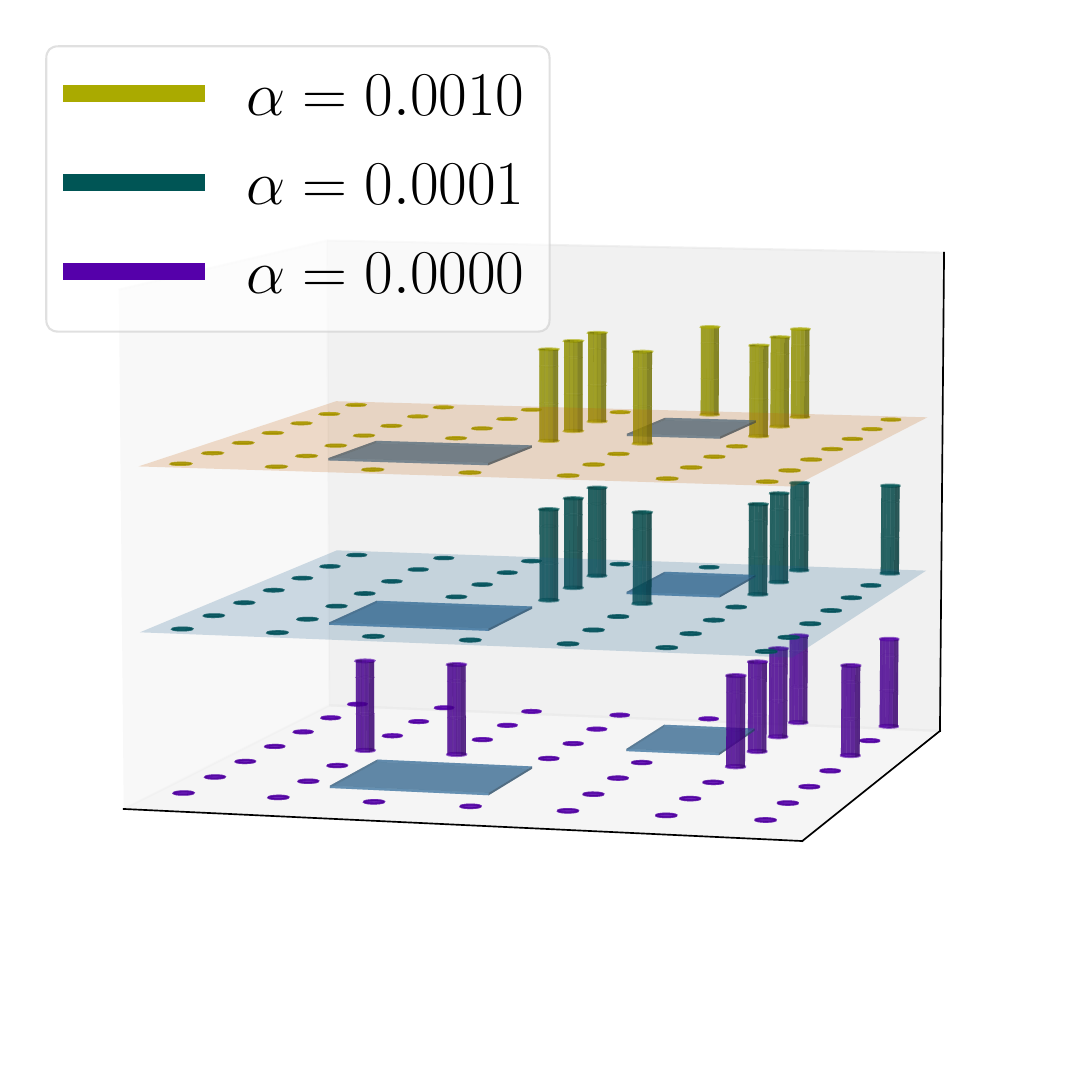}
      \includegraphics[width=0.24\linewidth, trim=80 80 0 0, clip]{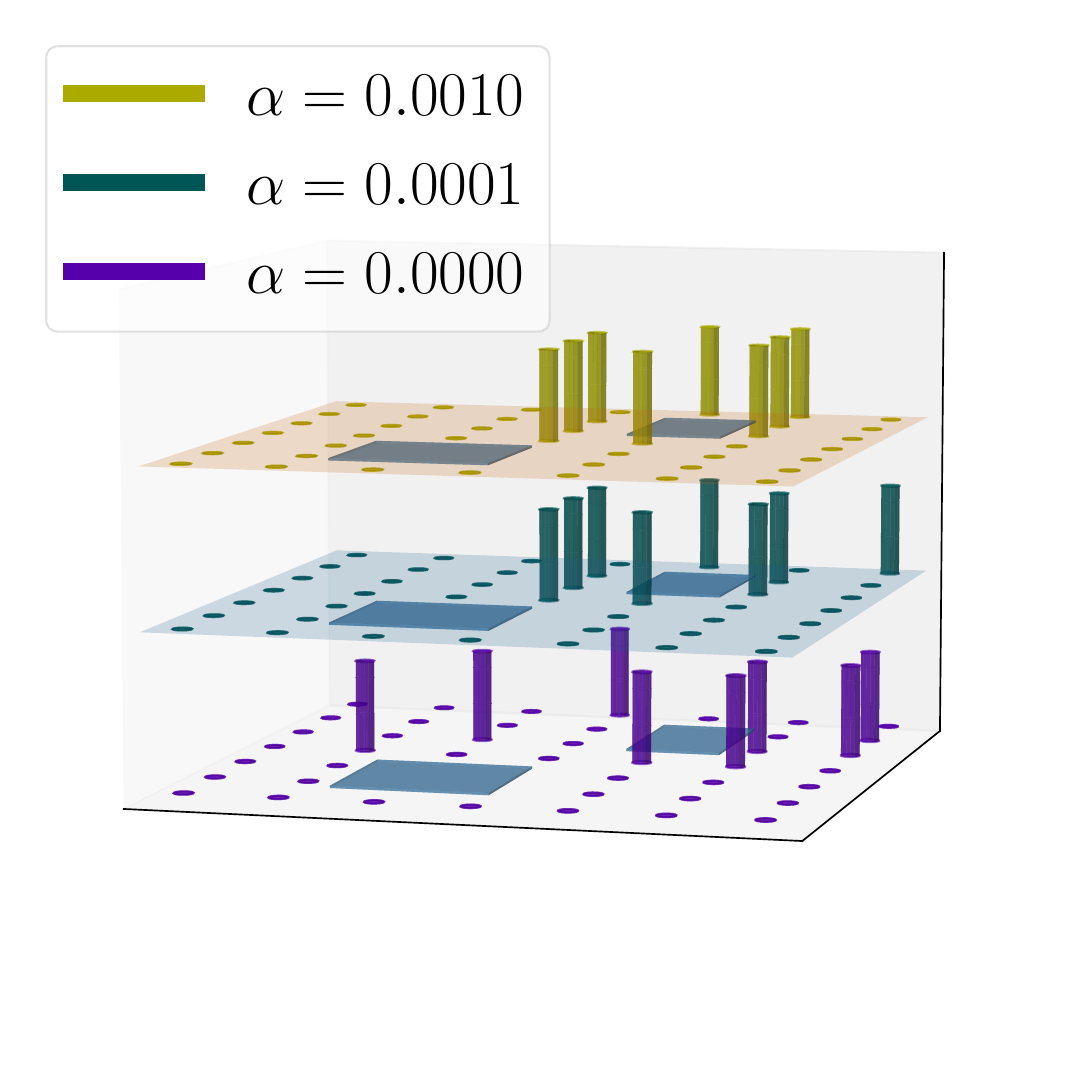}
      \includegraphics[width=0.24\linewidth, trim=80 80 0 0, clip]{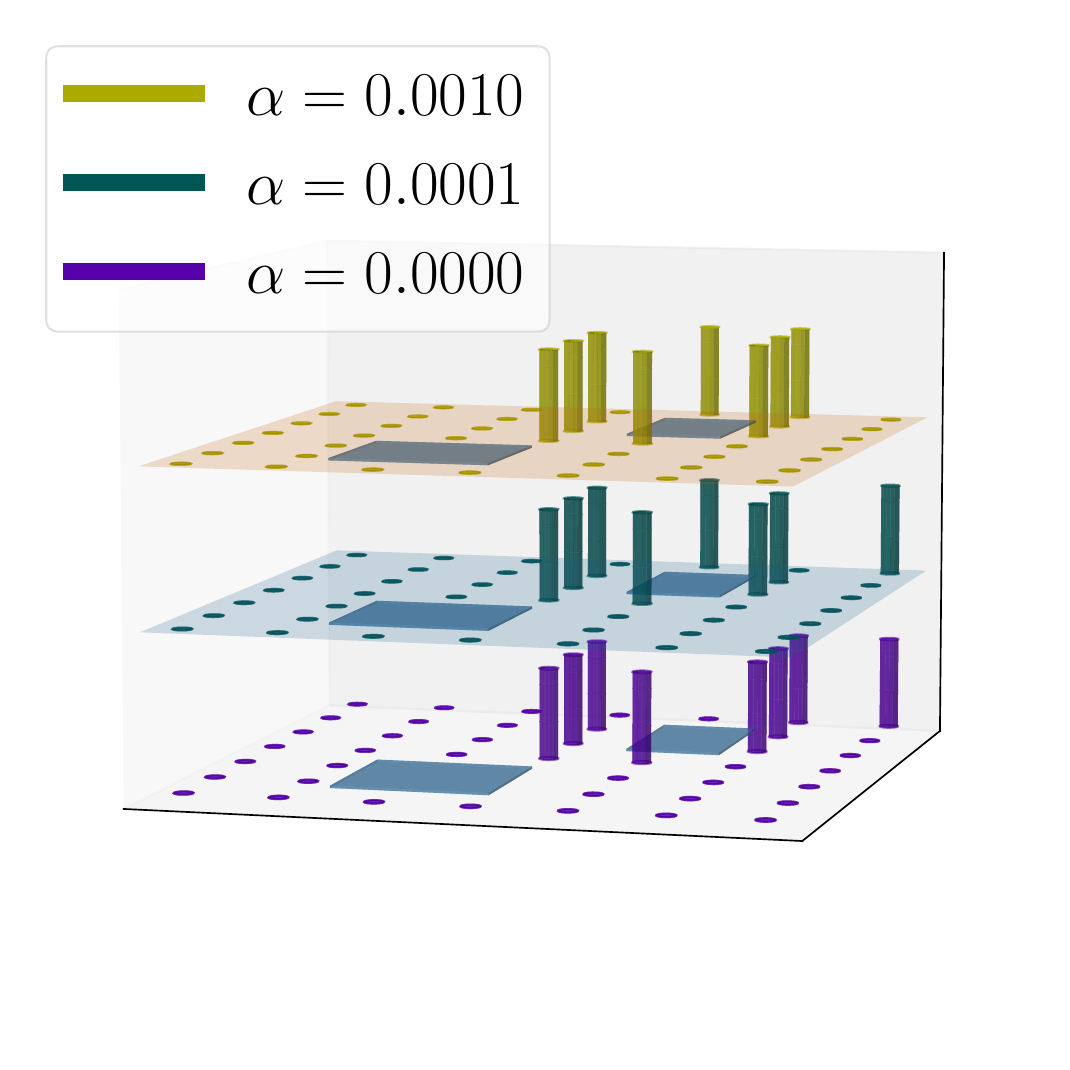}
      \caption{ Similar to \Cref{fig:OED_No_Correlation_A_Optimal_Designs}.
        Here we show  the binary design corresponding to the highest $\budget=8$
        weights.
        }
      \label{fig:OED_No_Correlation_A_Optimal_Designs_Max8}
    \end{figure}
    The results in~\Cref{fig:OED_No_Correlation_A_Optimal_Designs}
    and~\Cref{fig:OED_No_Correlation_A_Optimal_Designs_Max8} indicate that, in general, 
    increasing the value of the penalty parameter enforces sparsification
    on the solution of the OED problem. 
    However, the optimal weights, the ranking, and the sparsity levels 
    obtained by the four kernels are not identical. 
    It would be hard, however, to judge which solution is better without inspecting two elements: 
    the trace of the posterior covariance (i.e., the OED optimality criterion) and 
    the quality of the solution of an inverse problem obtained by reconfiguring the observational grid based on the optimal weights obtained.

    To be able to fairly judge the quality of the A-optimal designs, we inspect
    the
    value of the optimality criterion corresponding to the optimal solution.
    \Cref{tbl:Posterior_Trace_No_Correlation}
    shows the value of the A-optimality criterion, namely, the trace of the posterior covariance matrix,
    obtained by the optimal design before and after thresholding for multiple values of the penalty parameter $\alpha$.
    Results are obtained (from left to right) by using the standard OED
    formalism~\eqref{eqn:oed_relaxed_optimization_problem} with ad hoc
    fix~\eqref{eqn:standard_prepost_fixed},
    and the Schur-OED formulation~\eqref{eqn:A_optimality_Schur} with the
    weighting matrix $\designmat(\design)$ defined using the 
    kernel~\eqref{eqn:cov-prod_kernel}, the EXP kernel~\eqref{eqn:exp-prod_kernel}, and the
    logistic sigmoid kernels~\eqref{eqn:sigmoid-prod_kernel}, respectively. 
    For each choice of the weighting kernel, the first column shows the value of the posterior covariance trace obtained by setting the design to 
    the solution of the relaxed OED problem~\eqref{eqn:oed_relaxed_optimization_problem}, 
    that is, $\Psi^\GA(\design)=\Trace{\Cpredpost(\design)}$ with
    $\design=\design\opt$.
    The results in the second column correspond to the value of the optimality
    criterion $\Psi^\GA(\design^{\rm b})$, where $\design^{\rm b}$ is a binary design obtained by
    setting the entries of $\design\opt$ corresponding to the highest $\budget=8$ weights to $1$ and everything else to zero.
    The third column is similar to the second column, but the binary design is
    obtained by activating the sensors corresponding to the highest $\Nsens/2$
    weights.
    The results of both the standard and the Schur-product OED formulations are
    almost identical, with similar estimates of the optimal design.
    %
    \begin{table}[!ht]
      \caption{Value of the A-optimality criterion, namely, the
        trace of the posterior covariance matrix
        $\Psi^\GA(\design)=\Trace{\Cpredpost(\design)}$, multiplied by $10^3$,
        obtained by the optimal design, before and after thresholding, 
        for multiple values of the penalty parameter.
        Here, observations are assumed to be uncorrelated. 
        Two thresholding/rounding approaches are used: 
          (1) ``Max 8,'' in which the $\budget=8$ sensors with highest weights are
          activated, and 
          (2) ``Median,'' in which all sensors with weights above the median value of the optimal weights are activated.
        }\label{tbl:Posterior_Trace_No_Correlation}
      \centering\pgfkeys{/pgf/number format/fixed,
            /pgf/number format/precision=2}
    \resizebox{0.99\linewidth}{!}{%
      \begin{tabular}{|c||E|E|E||E|E|E||E|E|E||E|E|E||}
      \hline
        \multirow{3}{*}{\parbox{0.9in}{\centering Regularization Penalty}} 
          & \multicolumn{3}{c||}{Standard
            OED~(\ref{eqn:oed_relaxed_optimization_problem},~\ref{eqn:standard_prepost_fixed})}
          & \multicolumn{3}{c||}{Product Kernel~\eqref{eqn:cov-prod_kernel}}
          & \multicolumn{3}{c||}{EXP Kernel~\eqref{eqn:exp-prod_kernel}}
          & \multicolumn{3}{c||}{Sigmoid Kernel~\eqref{eqn:sigmoid-prod_kernel}}
          \\ \cline{2-13} 
          & \multicolumn{1}{c|}{A-optimal}
          & \multicolumn{2}{c||}{Thresholded Design} 
          & \multicolumn{1}{c|}{A-optimal}
          & \multicolumn{2}{c||}{Thresholded Design} 
          & \multicolumn{1}{c|}{A-optimal}
          & \multicolumn{2}{c||}{Thresholded Design} 
          & \multicolumn{1}{c|}{A-optimal}
          & \multicolumn{2}{c||}{Thresholded Design} 
          \\ \cline{3-4} \cline{6-7}  \cline{8-10} \cline{11-13} 
            & \multicolumn{1}{c|}{Design}
            & \multicolumn{1}{c|}{Max 8}
            & \multicolumn{1}{c||}{Median}   
            & \multicolumn{1}{c|}{Design}
            & \multicolumn{1}{c|}{Max 8}
            & \multicolumn{1}{c||}{Median}   
            & \multicolumn{1}{c|}{Design}
            & \multicolumn{1}{c|}{Max 8}
            & \multicolumn{1}{c||}{Median}   
            & \multicolumn{1}{c|}{Design}
            & \multicolumn{1}{c|}{Max 8}
            & \multicolumn{1}{c||}{Median}     \\ \hline
          $\alpha=0$     & 0.00896 & 0.01439 & 0.00939 & 0.00896 & 0.04589 & 0.00940 & 0.00896 & 0.04589 & 0.00940 & 0.00906 & 0.01052 & 0.00940  \\ \hline
          $\alpha=1e-4$  & 0.00924 & 0.01386 & 0.00946 & 0.00936 & 0.01052 & 0.00942 & 0.00935 & 0.01046 & 0.00940 & 0.00939 & 0.01046 & 0.00940  \\ \hline
          $\alpha=3e-4$  & 0.00964 & 0.01113 & 0.00948 & 0.00988 & 0.01039 & 0.00948 & 0.00992 & 0.01048 & 0.00940 & 0.00991 & 0.01048 & 0.00948  \\ \hline
          $\alpha=5e-4$  & 0.01001 & 0.01113 & 0.00985 & 0.01029 & 0.01048 & 0.00948 & 0.01030 & 0.01048 & 0.00940 & 0.01035 & 0.01048 & 0.00940  \\ \hline
          $\alpha=1e-3$  & 0.01069 & 0.01113 & 0.00989 & 0.01125 & 0.01048 & 0.00965 & 0.01125 & 0.01048 & 0.00953 & 0.01126 & 0.01048 & 0.00948  \\ \hline
      \end{tabular}%
      } 
    \end{table}
    %
   
    Now we turn our attention to the quality of the solution of an inverse
    problem (MAP point) obtained by reconfiguring the observational grid based on the
    optimal design.
    First, the observational grid is reconfigured as follows. We solve the A-OED problem to obtain the optimal
    design vector $\design\opt$; and then we calculate the vector of optimal weights
    $\varweightfunc(\design\opt, \design\opt)$,  which is then rounded to yield a binary
    design $\design^{\rm b}$. 
    The resulting binary design $\design^{\rm b}$ is used to reconfigure the
    observation grid, the observation operator---hence the forward operator
    $\F$---and the observation error covariance matrix. The inverse problem is
    then solved by using the reconfigured settings.
    \Cref{tbl:RMSE_No_Correlation} shows the RMSE results obtained by 
    solving the inverse problem using the binary designs corresponding
    to optimal designs with objective values shown in~\Cref{tbl:Posterior_Trace_No_Correlation}.
    We also show the RMSE results corresponding to the prior QoI and the
    solution of the inverse problem with all sensors activated, setting a
    benchmark of the RMSE values.
    The four settings yield optimal (rounded) designs that result
    in high-quality solutions of the inverse problems as suggested by the small
    RMSE values. 
    These results are also supported by the results 
    in~\Cref{fig:Restricted_IP_No_Correlation_Max8_Penalty_0.0001}, which
    displays the solution of the inverse problem along with the associated posterior
    uncertainties.
    Thus, in the case of uncorrelated observational errors, both the standard
    (pre- and postmultiplication of the precision matrix with the relaxed
    design) and the Schur-OED formalism behave almost identically except that 
    the Schur-OED optimization problem with logistic kernel is solved without
    enforcing bound constraints on the design.
    Note that by allowing more sensors to be deployed, additional data is
    collected, and hence 
    more information is gained. This is supported by the decrease in
    both objective values (\Cref{tbl:Posterior_Trace_No_Correlation}) and 
    the RMSE results (\Cref{tbl:RMSE_No_Correlation}) when the number
    of sensors is increased  from ``Max 8'' to ``Median'' rounding.
    %
    \begin{table}[!ht]
      \caption{Similar to~\Cref{tbl:Posterior_Trace_No_Correlation}. 
      Here we show the RMSE results obtained by solving the inverse problem 
      given the optimal design.
        }\label{tbl:RMSE_No_Correlation}
      \centering\pgfkeys{/pgf/number format/fixed,
            /pgf/number format/fixed zerofill,
            /pgf/number format/precision=2}
    \resizebox{0.80\linewidth}{!}{%
      \begin{tabular}{|c||c|c||c|c||c|c||c|c||c|c||}
      \hline
        \multirow{3}{*}{\parbox{0.9in}{\centering Regularization Penalty}} 
        & \multicolumn{10}{c||}{RMSE}
        \\ \cline{2-11}
          & \multicolumn{2}{c||}{All Sensors}
          & \multicolumn{2}{c||}{Standard
            OED~(\ref{eqn:oed_relaxed_optimization_problem},~\ref{eqn:standard_prepost_fixed})}
          & \multicolumn{2}{c||}{Product Kernel~\eqref{eqn:cov-prod_kernel}}
          & \multicolumn{2}{c||}{EXP Kernel~\eqref{eqn:exp-prod_kernel}}
          & \multicolumn{2}{c||}{Sigmoid Kernel~\eqref{eqn:sigmoid-prod_kernel}}
          \\ \cline{2-11}
            & \multicolumn{1}{c|}{Prior}
            & \multicolumn{1}{c||}{Posterior}   
            & \multicolumn{1}{c|}{Max 8}
            & \multicolumn{1}{c||}{Median}   
            & \multicolumn{1}{c|}{Max 8}
            & \multicolumn{1}{c||}{Median}   
            & \multicolumn{1}{c|}{Max 8}
            & \multicolumn{1}{c||}{Median}   
            & \multicolumn{1}{c|}{Max 8}
            & \multicolumn{1}{c||}{Median}     \\ \hline
            $\alpha=0$     & \multicolumn{1}{c|}{\multirow{5}{*}{0.4704}} & \multicolumn{1}{c||}{\multirow{5}{*}{0.0258}}  
                               & 0.02751 & 0.02610 & 0.05295 & 0.02798 & 0.05295 & 0.02798 & 0.04048 & 0.02799  \\ \cline{1-1} \cline{4-11}
            $\alpha=1e-4$  & & & 0.04154 & 0.02717 & 0.04048 & 0.02724 & 0.03984 & 0.02714 & 0.03984 & 0.02714  \\ \cline{1-1} \cline{4-11}
            $\alpha=3e-4$  & & & 0.03966 & 0.02805 & 0.03190 & 0.02971 & 0.04040 & 0.02797 & 0.04040 & 0.02971  \\ \cline{1-1} \cline{4-11}
            $\alpha=5e-4$  & & & 0.03966 & 0.03352 & 0.04040 & 0.02967 & 0.04040 & 0.02797 & 0.04040 & 0.02797  \\ \cline{1-1} \cline{4-11}
            $\alpha=1e-3$  & & & 0.03966 & 0.03364 & 0.04040 & 0.03136 & 0.04040 & 0.03028 & 0.04040 & 0.02971  \\ \hline
      \end{tabular}%
      } 
    \end{table}
    %
    
    \begin{figure}[!ht]\vspace{-5pt}
      \centering
        \includegraphics[width=0.75\linewidth]{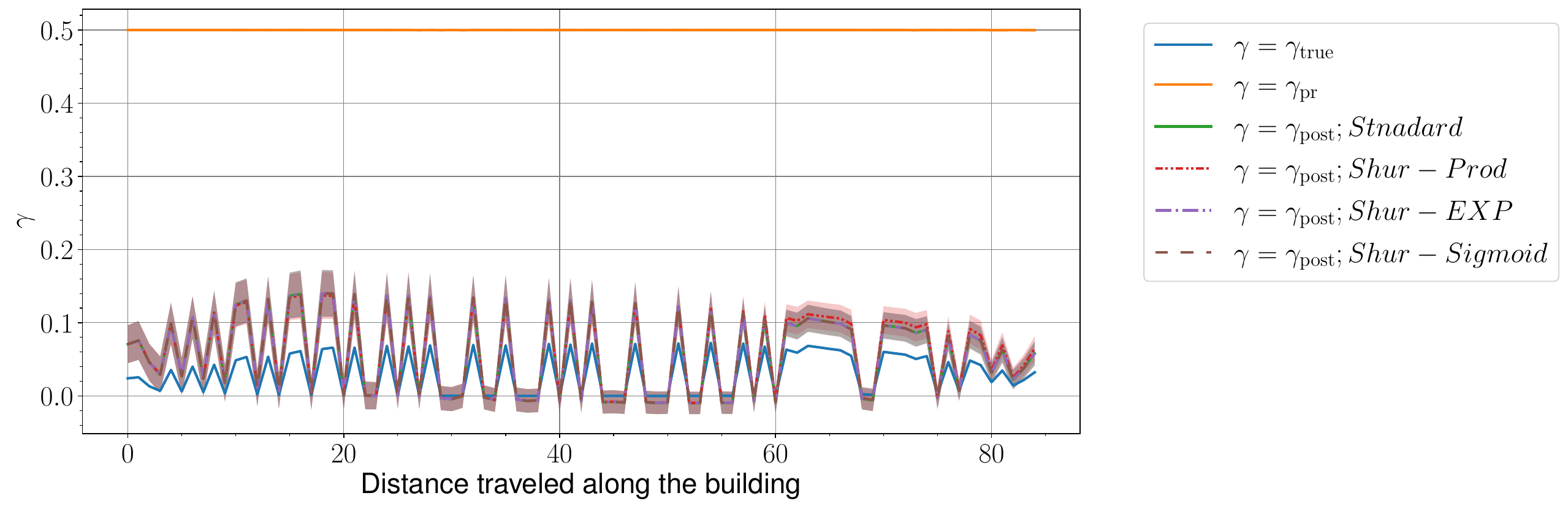}
        \caption{Solution of the inverse problem, along with associated
          uncertainty ($\pm2\sigma_{\rm post}$), obtained based on the rounded
          A-optimal design resulting from solving
          the OED problem using settings described in~\Cref{fig:OED_No_Correlation_A_Optimal_Designs}.
          Here, observations are assumed to be uncorrelated, and we show results obtained with $\alpha=1e-4$
          and with $\budget=8$ sensors  activated. 
          Inversion results here correspond to the second row of results in~\Cref{tbl:RMSE_No_Correlation}.
          }
        \label{fig:Restricted_IP_No_Correlation_Max8_Penalty_0.0001}
      \end{figure}
      %

  \subsubsection{Space-correlations results}
  \label{subsec:space_corr_results}
  Here we show the results of experiments carried out with spatially correlated observations.
  \Cref{fig:OED_Correlation_LengthScale_1_Weights_Max8} shows the rounded
  binary designs obtained by setting the correlation length scale to $\ell=1$,
  and~\Cref{fig:OED_Correlation_LengthScale_3_Weights_Max8} shows binary designs
  for $\ell=3$. Similar to the discussion in~\Cref{subsec:no_corr_results}, results are 
  shown here for multiple values of the penalty parameter $\alpha$.
  These results (in comparison
  with~\Cref{fig:OED_No_Correlation_A_Optimal_Designs_Max8}) indicate that
  discarding the  observation correlations in an OED problem results in different
  optimal designs and thus can greatly degrade the quality of the solution of
  an inverse problem.
  Moreover, in the presence of spatial correlations between neighboring
  candidate sensor locations, the optimal design tends to be more spread in the
  domain.
  \begin{figure}[!ht]\vspace{-10pt}
  \centering
      \includegraphics[width=0.24\linewidth, trim=80 80 0 0, clip]{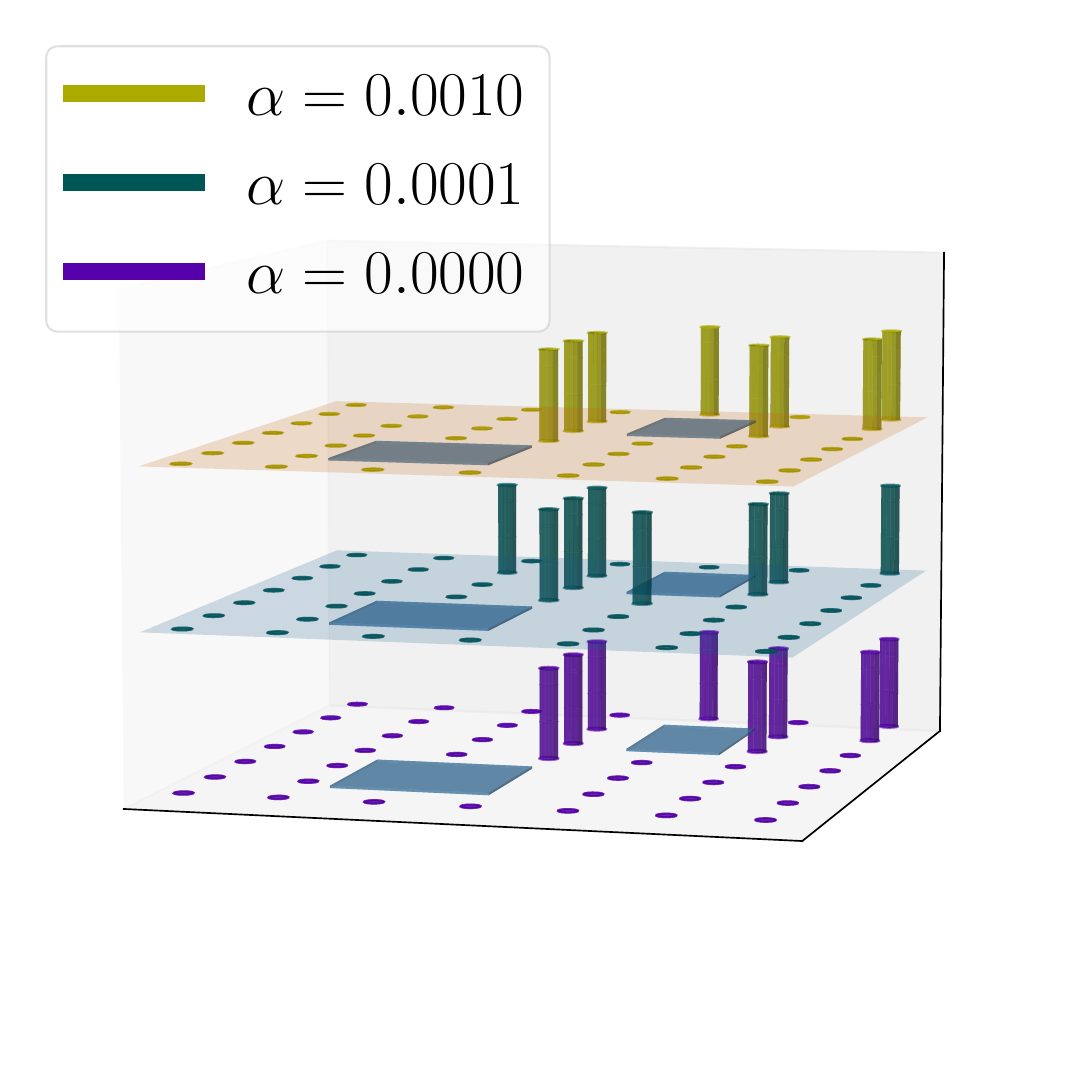} 
      \includegraphics[width=0.24\linewidth, trim=80 80 0 0, clip]{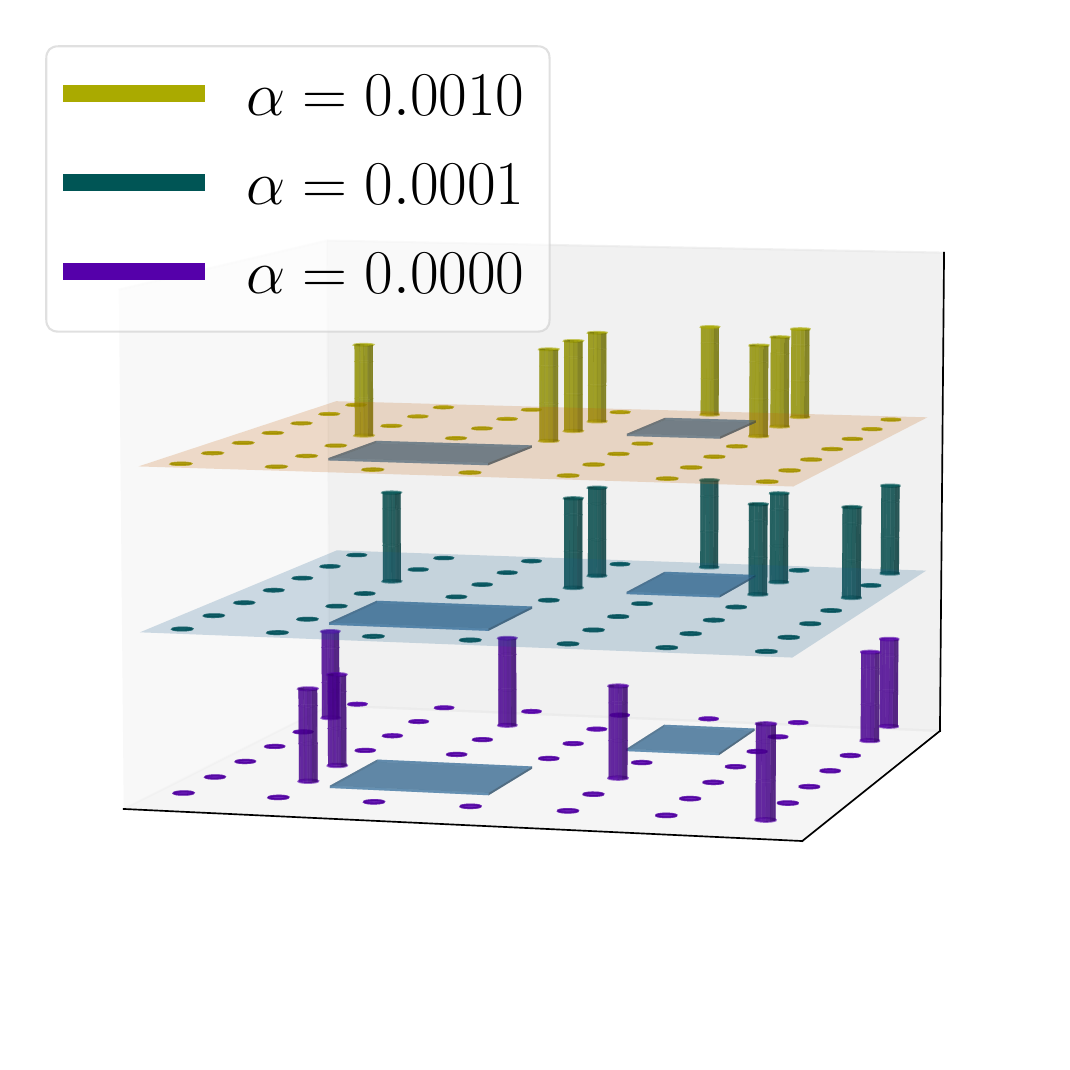}
      \includegraphics[width=0.24\linewidth, trim=80 80 0 0, clip]{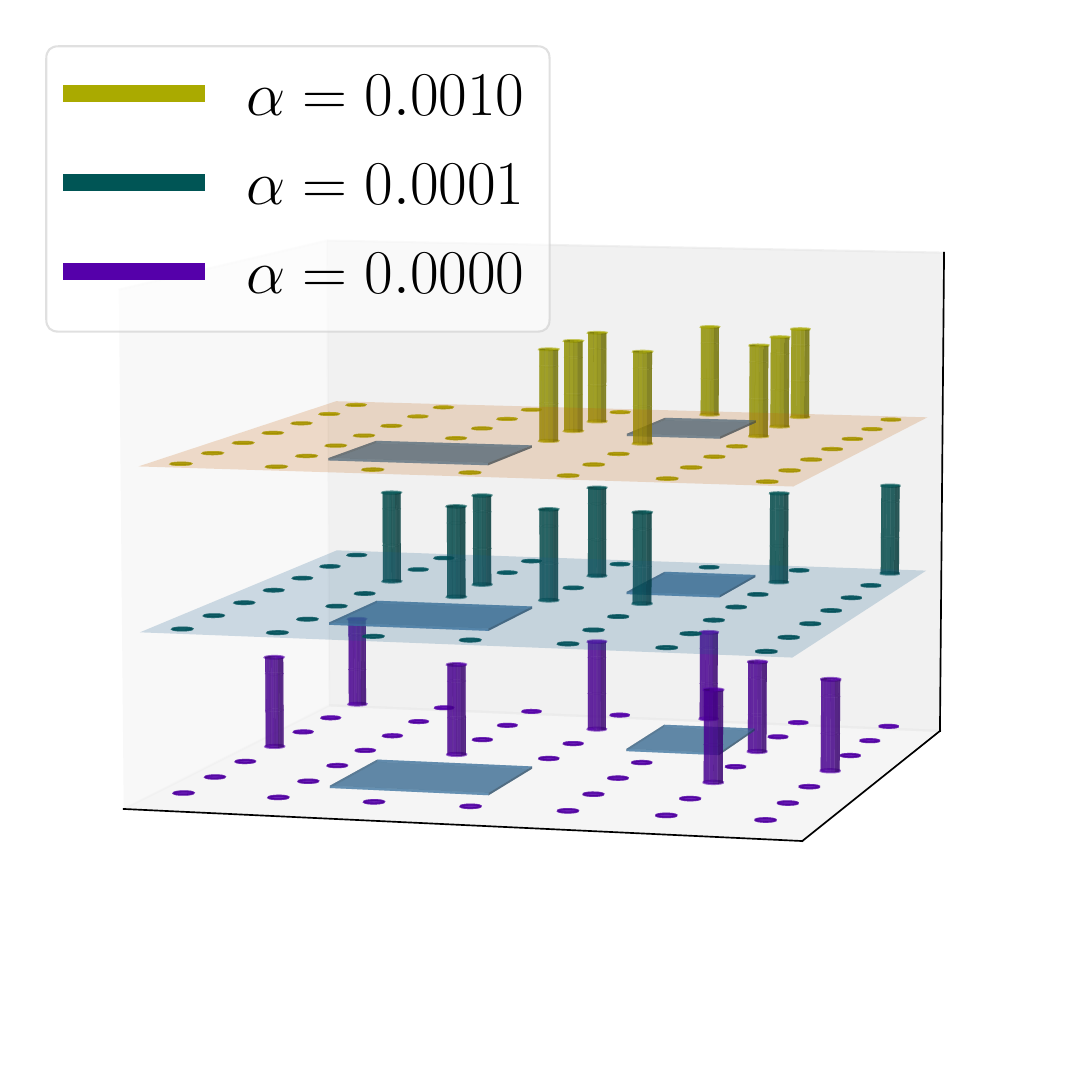}
      \includegraphics[width=0.24\linewidth, trim=80 80 0 0, clip]{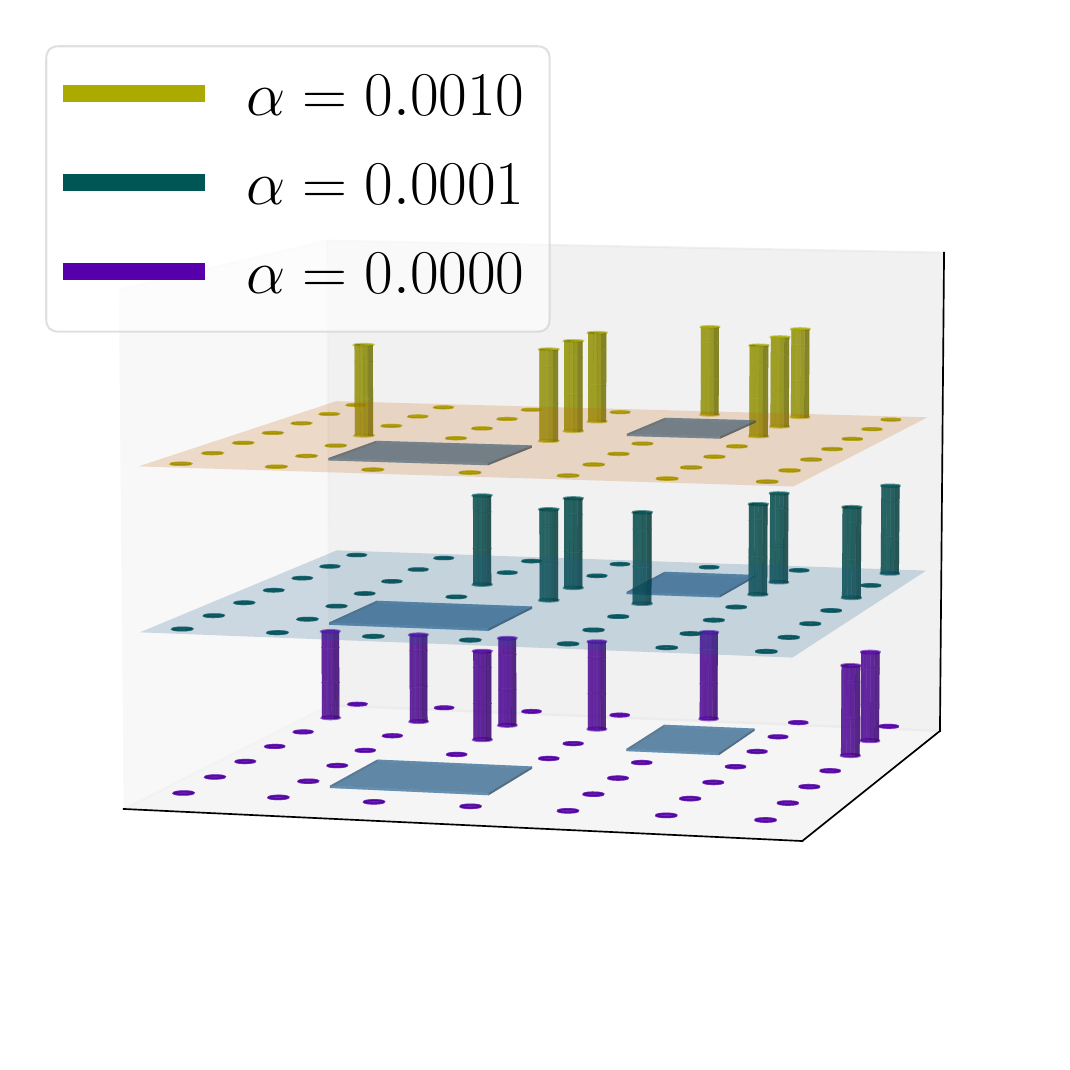}
    \caption{Similar to~\Cref{fig:OED_No_Correlation_A_Optimal_Designs_Max8}.
      Here we allow observation correlations with length scale $\ell=1$.
      }
    \label{fig:OED_Correlation_LengthScale_1_Weights_Max8}
  \end{figure}
  \begin{figure}[!ht]\vspace{-10pt}
  \centering
      \includegraphics[width=0.24\linewidth, trim=80 80 0 0, clip]{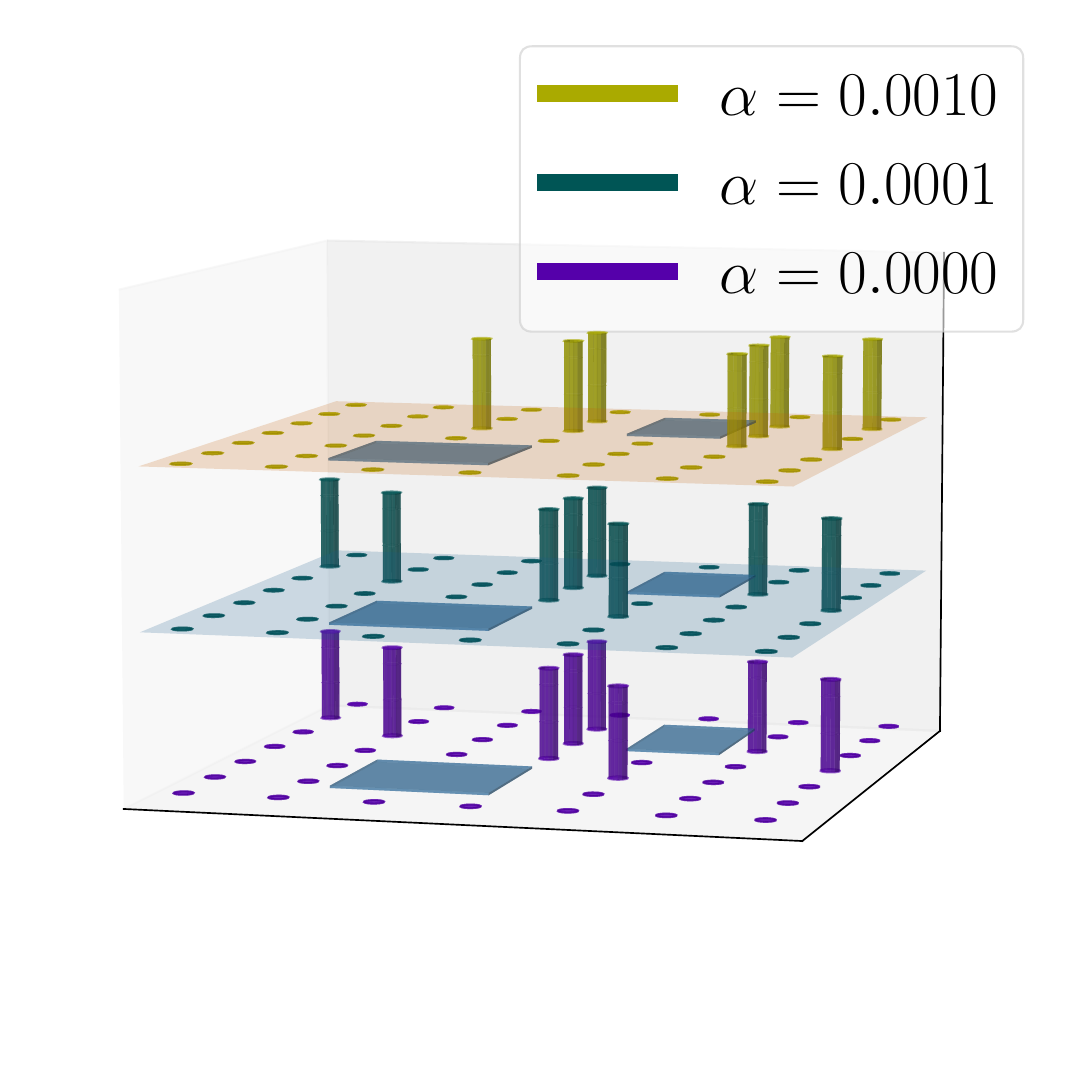} 
      \includegraphics[width=0.24\linewidth, trim=80 80 0 0, clip]{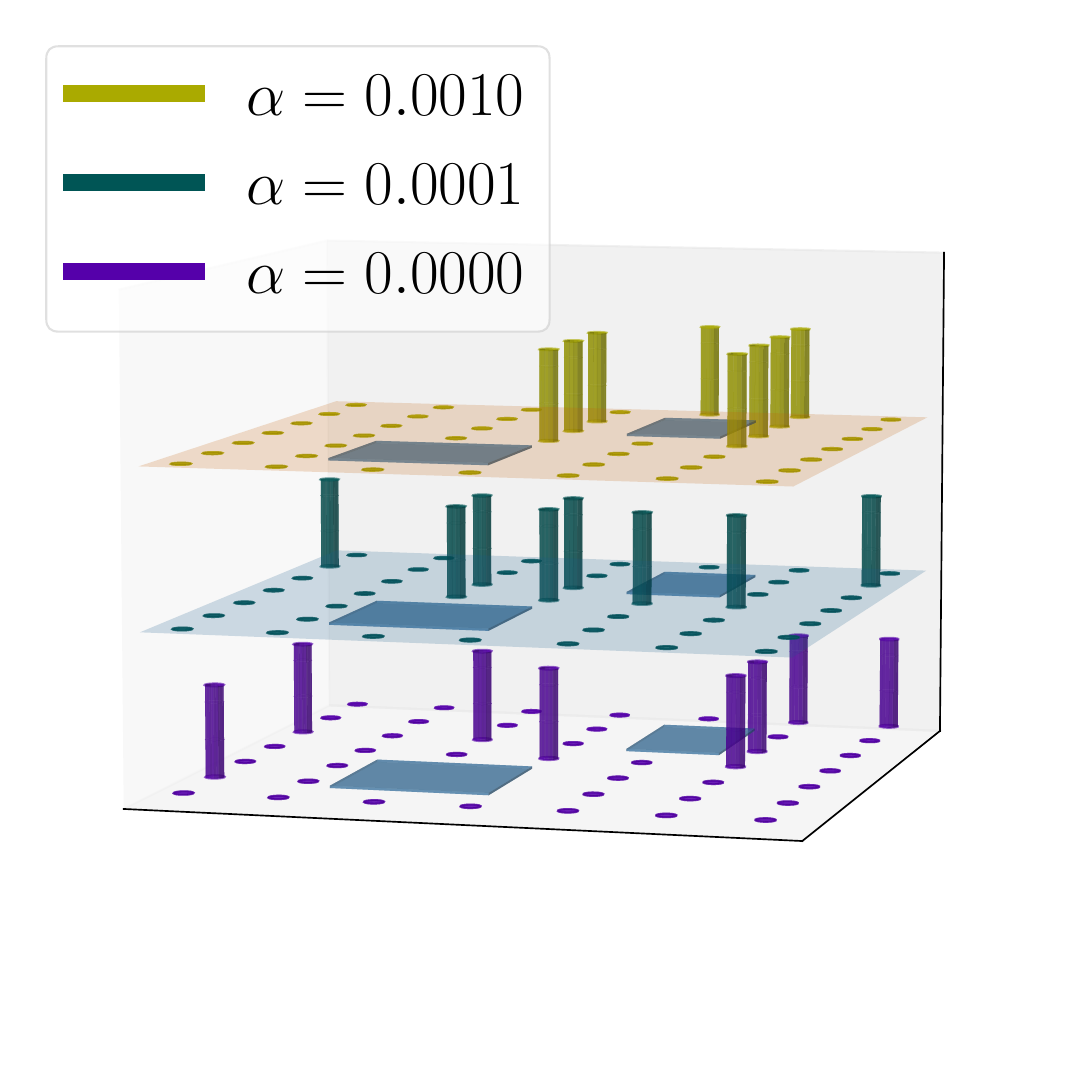}
      \includegraphics[width=0.24\linewidth, trim=80 80 0 0, clip]{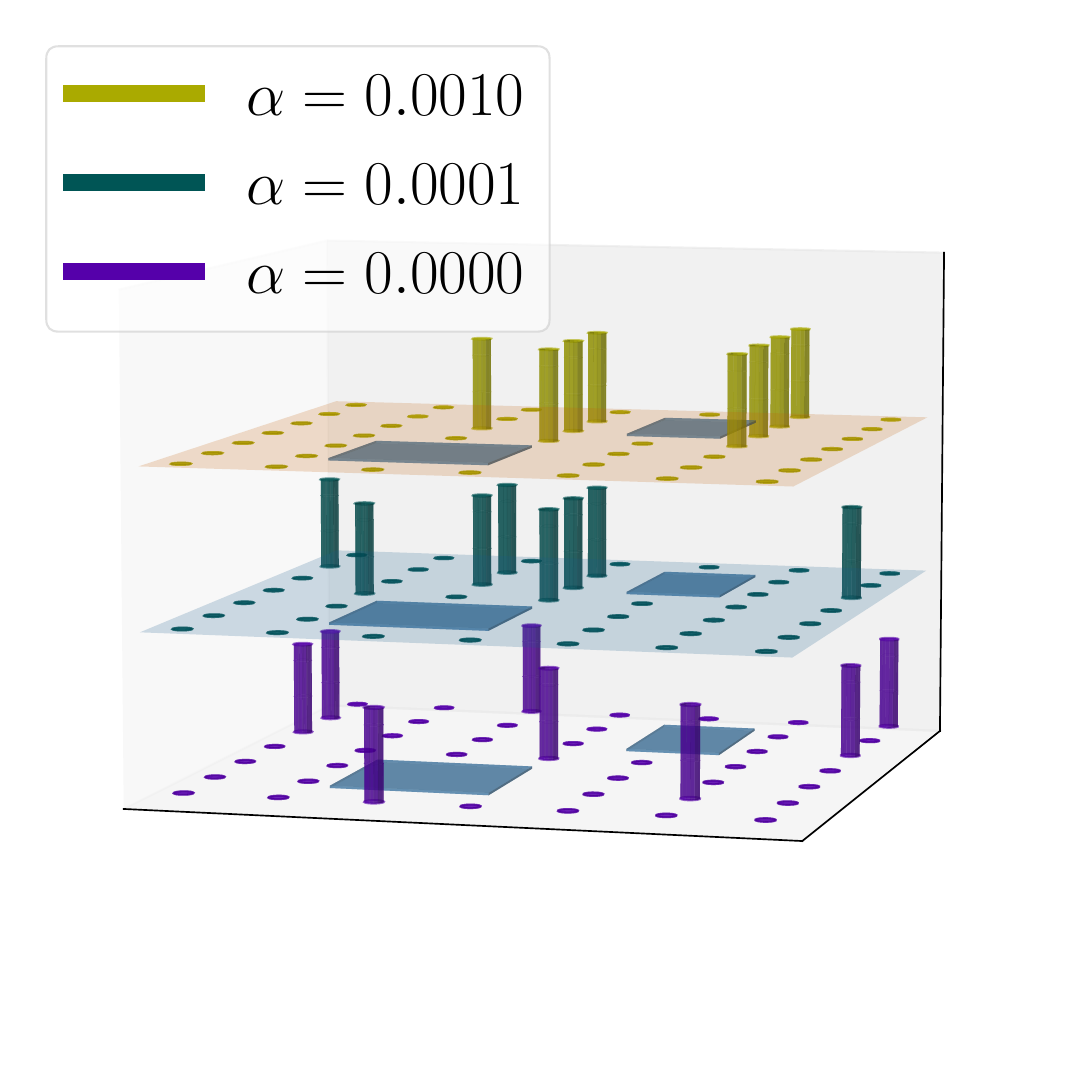}
      \includegraphics[width=0.24\linewidth, trim=80 80 0 0, clip]{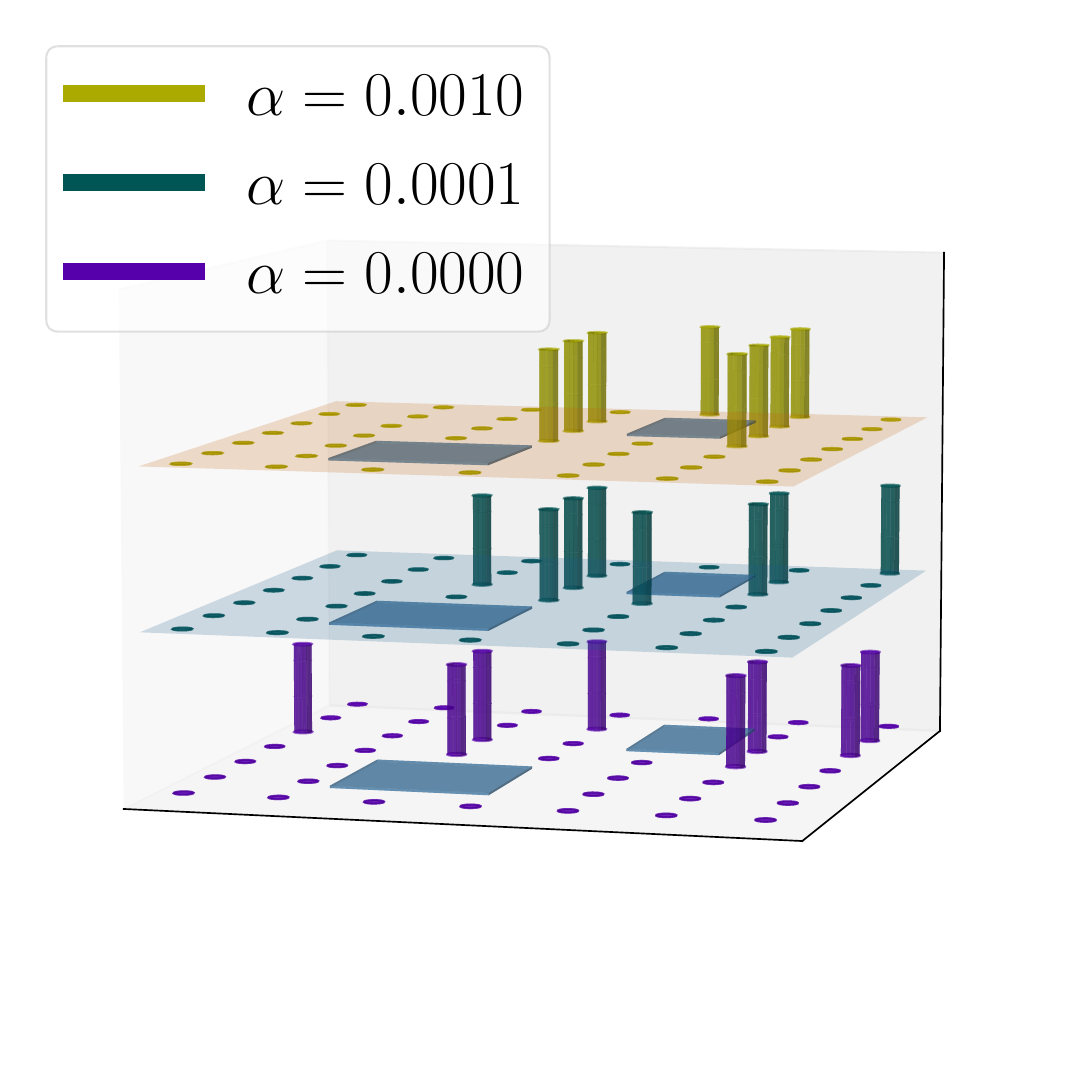}
    \caption{Similar to~\Cref{fig:OED_No_Correlation_A_Optimal_Designs_Max8}.
      Here we allow observation correlations with length scale $\ell=3$.
      }
    \label{fig:OED_Correlation_LengthScale_3_Weights_Max8}
  \end{figure}

  To understand the empirical performance of the experiments carried out with observation
  spatial correlations, we analyze the value of the optimality criterion
  (\Cref{tbl:Posterior_Trace_Correlation}) as
  well as the RMSE results (\Cref{tbl:RMSE_Correlation}). 
  \Cref{tbl:Posterior_Trace_Correlation} shows the value of the A-optimality criterion, 
  namely, the trace of the posterior covariance matrix,
  obtained by the optimal design, before and after thresholding, for multiple values of the 
  penalty parameter $\alpha$.
  The results in~\Cref{tbl:Posterior_Trace_Correlation} show that the standard
  OED formulation is not consistent with the Schur-OED formulation.
  It is hard to compare the two sets of results because they show
  numerical approximation obtained by using two different objective functions.
  Nevertheless, this comparison can in part be made based on the accuracy of the resulting
  solution, for example, by inspecting RMSE results.
  \Cref{tbl:RMSE_Correlation} shows the RMSE of the solution of the inverse
  problem corresponding to optimal designs with objective values displayed
  in~\Cref{tbl:Posterior_Trace_Correlation}.
  The overall performance explained by results in~\Cref{tbl:RMSE_Correlation} 
  shows that the Schur-OED formulation outperforms the standard OED formulation
  especially for stronger presence of observational error correlations, for example, for larger correlation length scale.
  We note that with a wider correlation length scale, that is, by allowing more
  sensors to be highly correlated, fewer sensors are required to
  achieve high accuracy. This is explained by the significant drop in RMSE
  values from $\ell\!=\!1$ to $\ell\!=\!3$.  
  This is also obvious by comparing the RMSE results
  in~\Cref{tbl:RMSE_Correlation} with ~\Cref{tbl:RMSE_No_Correlation}.

  %
  \begin{table}[!ht]
    \caption{Similar to~\Cref{tbl:Posterior_Trace_No_Correlation}.
      Here we allow observation correlations, with correlation length scale
      $\ell$ set to $\ell=1$ and $\ell=3$, respectively. 
    }\label{tbl:Posterior_Trace_Correlation}
    \centering\pgfkeys{/pgf/number format/fixed,
          /pgf/number format/precision=2}
  \resizebox{0.99\linewidth}{!}{%
    \begin{tabular}{|c||c||E|E|E||E|E|E||E|E|E||E|E|E|}
    \hline
        \multirow{3}{*}{\parbox{0.9in}{\centering Correlation\\Length-scale}} 
        & \multirow{3}{*}{\parbox{0.9in}{\centering Regularization Penalty}} 
          
        & \multicolumn{3}{c||}{Standard
          OED~(\ref{eqn:oed_relaxed_optimization_problem},~\ref{eqn:standard_prepost_fixed})}
        & \multicolumn{3}{c||}{Product Kernel~\eqref{eqn:cov-prod_kernel}}
        & \multicolumn{3}{c||}{EXP Kernel~\eqref{eqn:exp-prod_kernel}}
        & \multicolumn{3}{c|}{Sigmoid Kernel~\eqref{eqn:sigmoid-prod_kernel}}
        \\ \cline{3-14} 
        &
        & \multicolumn{1}{c|}{A-optimal} 
        & \multicolumn{2}{c||}{Thresholded Design} 
        & \multicolumn{1}{c|}{A-optimal} 
        & \multicolumn{2}{c||}{Thresholded Design} 
        & \multicolumn{1}{c|}{A-optimal} 
        & \multicolumn{2}{c||}{Thresholded Design} 
        & \multicolumn{1}{c|}{A-optimal} 
        & \multicolumn{2}{c|}{Thresholded Design} 
        \\ \cline{4-5} \cline{7-8} \cline{10-11} \cline{13-14} 
          &  
          & \multicolumn{1}{c|}{Design}
          & \multicolumn{1}{c|}{Max 8}
          & \multicolumn{1}{c||}{Median}   
          & \multicolumn{1}{c|}{Design}
          & \multicolumn{1}{c|}{Max 8}
          & \multicolumn{1}{c||}{Median}   
          & \multicolumn{1}{c|}{Design}
          & \multicolumn{1}{c|}{Max 8}
          & \multicolumn{1}{c||}{Median}   
          & \multicolumn{1}{c|}{Design}
          & \multicolumn{1}{c|}{Max 8}
          & \multicolumn{1}{c|}{Median}     \\ \hline \hline
          \multicolumn{1}{|c||}{\multirow{5}{*}{$\ell=1$}}  
              & $\alpha=0$ & 
                0.00679 & 0.00779 & 0.00711 & 0.00701 & 0.02777 & 0.00719 & 0.00701 & 0.02962 & 0.00718 & 0.00704 & 0.08983 & 0.00828 
                \\ \cline{2-14}
          \multicolumn{1}{|c||}{}
              & $\alpha=1e-4$ &  
                0.00683 & 0.00775 & 0.00703 & 0.00812 & 0.00938 & 0.00732 & 0.00808 & 0.01825 & 0.00719 & 0.00815 & 0.01042 & 0.00736
                \\ \cline{2-14}
          \multicolumn{1}{|c||}{}
              & $\alpha=3e-4$ &  
                0.00724 & 0.00857 & 0.00728 & 0.00908 & 0.00766 & 0.00722 & 0.00875 & 0.00831 & 0.00725 & 0.00902 & 0.00709 & 0.00713
                \\ \cline{2-14}
          \multicolumn{1}{|c||}{}
              & $\alpha=5e-4$ &  
                0.00762 & 0.00779 & 0.00719 & 0.00961 & 0.00766 & 0.00722 & 0.00969 & 0.00781 & 0.00727 & 0.01017 & 0.00777 & 0.00705
                \\ \cline{2-14}
          \multicolumn{1}{|c||}{}
              & $\alpha=1e-3$ &  
                0.00790 & 0.00779 & 0.00739 & 0.01123 & 0.00725 & 0.00676 & 0.01122 & 0.00728 & 0.00694 & 0.01123 & 0.00725 & 0.00668
                \\ \hline  \hline
          \multicolumn{1}{|c||}{\multirow{5}{*}{$\ell=3$}}  
              & $\alpha=0$ &  
                0.00329 & 0.00581 & 0.00321 & 0.00336 & 0.03241 & 0.00233 & 0.00353 & 0.04341 & 0.00229 & 0.00329 & 0.00647 & 0.00289
                \\ \cline{2-14}
          \multicolumn{1}{|c||}{}
              & $\alpha=1e-4$ &  
                0.00329 & 0.00581 & 0.00321 & 0.00499 & 0.02117 & 0.00247 & 0.00431 & 0.00835 & 0.00230 & 0.00765 & 0.00386 & 0.00241
                \\ \cline{2-14}
          \multicolumn{1}{|c||}{}
              & $\alpha=3e-4$ & 
                0.00330 & 0.00581 & 0.00321 & 0.00653 & 0.00399 & 0.00287 & 0.00631 & 0.00399 & 0.00263 & 0.00745 & 0.00326 & 0.00255
                \\ \cline{2-14}
          \multicolumn{1}{|c||}{}
              & $\alpha=5e-4$ &  
                0.00304 & 0.00345 & 0.00255 & 0.00732 & 0.00364 & 0.00269 & 0.00688 & 0.00371 & 0.00260 & 0.00920 & 0.00387 & 0.00234
                \\ \cline{2-14}
          \multicolumn{1}{|c||}{}
              & $\alpha=1e-3$ &  
                0.00364 & 0.00339 & 0.00290 & 0.01069 & 0.00373 & 0.00226 & 0.00652 & 0.00331 & 0.00241 & 0.01065 & 0.00373 & 0.00216 
                \\ \hline 
    \end{tabular}%
    } 
  \end{table}
  %

    %
    \begin{table}[!ht]
      \caption{Similar to~\Cref{tbl:RMSE_No_Correlation}.
        Here we allow observation correlations, with correlation length scale
        $\ell$ set to $\ell=1$ and $\ell=3$, respectively. 
        }\label{tbl:RMSE_Correlation}
      \centering\pgfkeys{/pgf/number format/fixed,
            /pgf/number format/fixed zerofill,
            /pgf/number format/precision=2}
    \resizebox{0.85\linewidth}{!}{%
      \begin{tabular}{|c||c||c|c||c|c||c|c||c|c||c|c|}
      \hline
        \multirow{3}{*}{\parbox{0.9in}{\centering Correlation\\Length-scale}} 
        & \multirow{3}{*}{\parbox{0.9in}{\centering Regularization Penalty}} 
        & \multicolumn{10}{c|}{RMSE}
          \\ \cline{3-12}
          &
          & \multicolumn{2}{c||}{All Sensors}
          & \multicolumn{2}{c||}{Standard
            OED~(\ref{eqn:oed_relaxed_optimization_problem},~\ref{eqn:standard_prepost_fixed})}
          & \multicolumn{2}{c||}{Product Kernel~\eqref{eqn:cov-prod_kernel}}
          & \multicolumn{2}{c||}{EXP Kernel~\eqref{eqn:exp-prod_kernel}}
          & \multicolumn{2}{c|}{Sigmoid Kernel~\eqref{eqn:sigmoid-prod_kernel}}
          \\ \cline{3-12}
            & 
            & \multicolumn{1}{c|}{Prior}
            & \multicolumn{1}{c||}{Posterior}   
            & \multicolumn{1}{c|}{Max 8}
            & \multicolumn{1}{c||}{Median}   
            & \multicolumn{1}{c|}{Max 8}
            & \multicolumn{1}{c||}{Median}   
            & \multicolumn{1}{c|}{Max 8}
            & \multicolumn{1}{c||}{Median}   
            & \multicolumn{1}{c|}{Max 8}
            & \multicolumn{1}{c|}{Median}     \\ \hline \hline
            \multirow{5}{*}{$\ell=1$}
            & $\alpha=0$
            &\multicolumn{1}{c|}{\multirow{5}{*}{0.4704}}&\multicolumn{1}{c||}{\multirow{5}{*}{0.0225}}& 
              0.04457 & 0.01799 & 0.05667 & 0.02262 & 0.02969 & 0.02249 & 0.08218 & 0.02327
              \\ \cline{2-2} \cline{5-12}
            & $\alpha=1e-4$ & &  & 
              0.05227 & 0.02434 & 0.02036 & 0.01358 & 0.01594 & 0.01315 & 0.05219 & 0.01491
              \\ \cline{2-2} \cline{5-12}
            & $\alpha=3e-4$ & &  & 
              0.04535 & 0.02130 & 0.02321 & 0.01960 & 0.02374 & 0.01401 & 0.04955 & 0.02212
              \\ \cline{2-2} \cline{5-12} 
            & $\alpha=5e-4$ & &  & 
              0.04457 & 0.03283 & 0.02321 & 0.01960 & 0.02417 & 0.01919 & 0.04694 & 0.02998
              \\ \cline{2-2} \cline{5-12} 
            & $\alpha=1e-3$ & &  & 
              0.04457 & 0.01959 & 0.03310 & 0.03169 & 0.05018 & 0.02870 & 0.03310 & 0.03056
              \\ \hline \hline 
            \multirow{5}{*}{$\ell=3$}
            & $\alpha=0$     &\multicolumn{1}{c|}{\multirow{5}{*}{0.4704}}&\multicolumn{1}{c||}{\multirow{5}{*}{0.0239}}& 
              0.05055 & 0.02920 & 0.08058 & 0.01356 & 0.10616 & 0.01053 & 0.02389 & 0.02499
              \\ \cline{2-2} \cline{5-12}
            & $\alpha=1e-4$  & & &
              0.05055 & 0.02920 & 0.14247 & 0.01687 & 0.01879 & 0.00903 & 0.01820 & 0.01177 
              \\ \cline{2-2} \cline{5-12}
            & $\alpha=3e-4$  & & &
              0.05055 & 0.02920 & 0.01711 & 0.01880 & 0.01711 & 0.02025 & 0.01820 & 0.01177
              \\ \cline{2-2} \cline{5-12}
            & $\alpha=5e-4$  & & &
              0.04985 & 0.01161 & 0.02042 & 0.01892 & 0.01832 & 0.01337 & 0.01179 & 0.01152
              \\ \cline{2-2} \cline{5-12}
            & $\alpha=1e-3$  & & &
              0.04763 & 0.02041 & 0.04279 & 0.01124 & 0.02722 & 0.00988 & 0.01279 & 0.00755
              \\ \hline
      \end{tabular}%
      } 
    \end{table}
    %

  We conclude this section by showing the solution of the inverse problem
  obtained by reconfiguring the observational setup based on the A-optimal
  design resulting from~\eqref{eqn:oed_relaxed_optimization_problem} in the
  presence of observation correlations;
  see~\Cref{fig:Restricted_IP_Correlation_Max8_Penalty_0.0001}.
  These results support our assertion that with increasing correlation
  length scale, fewer sensors are needed achieve high levels of
  accuracy. The results also show that the Schur-OED formulation outperforms
  the standard OED approach especially for larger correlation length scales. 
  \begin{figure}[!ht]\vspace{-5pt}
    \centering
      \includegraphics[width=0.49\linewidth]{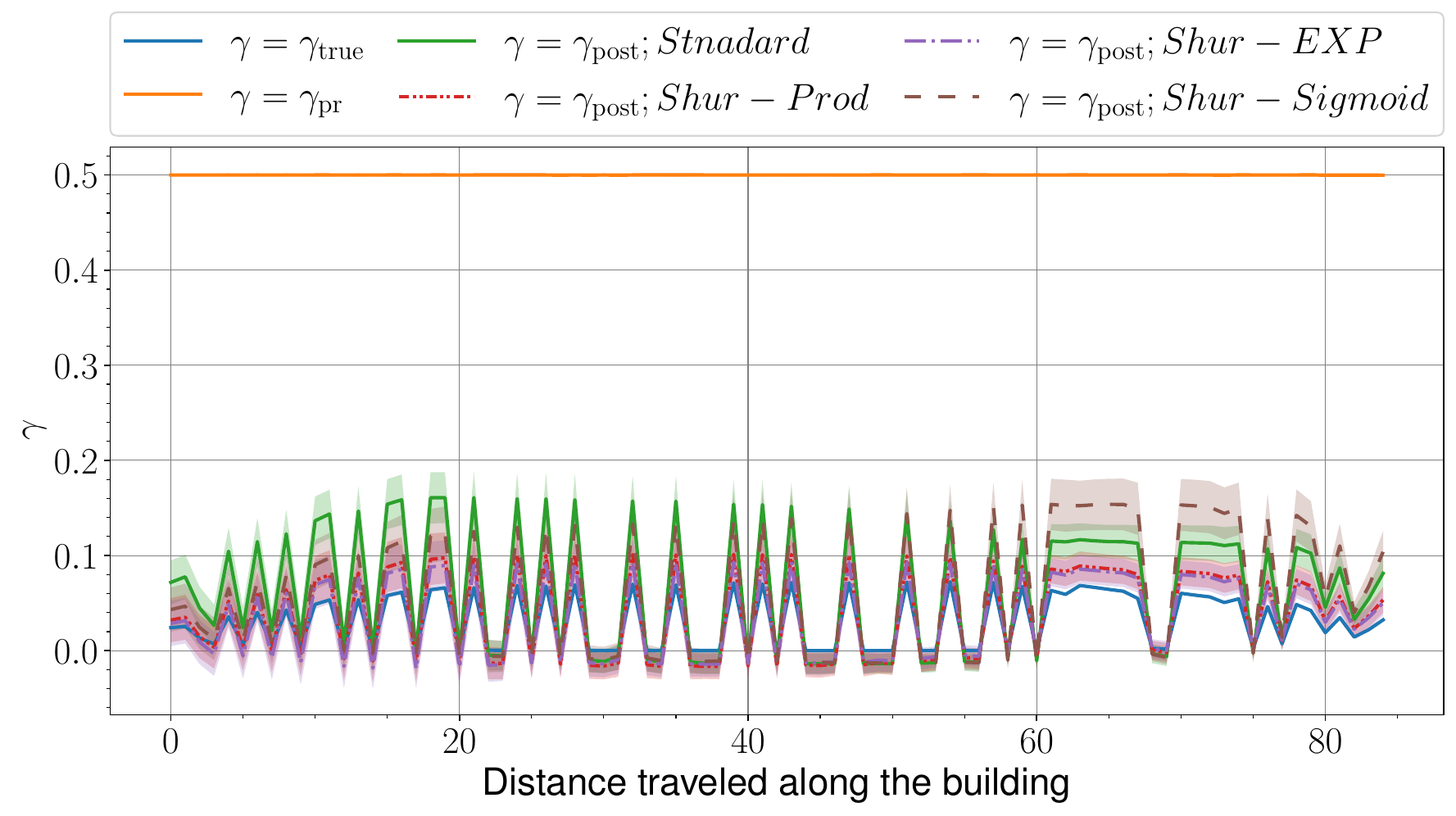}
      \hfill
      \includegraphics[width=0.49\linewidth]{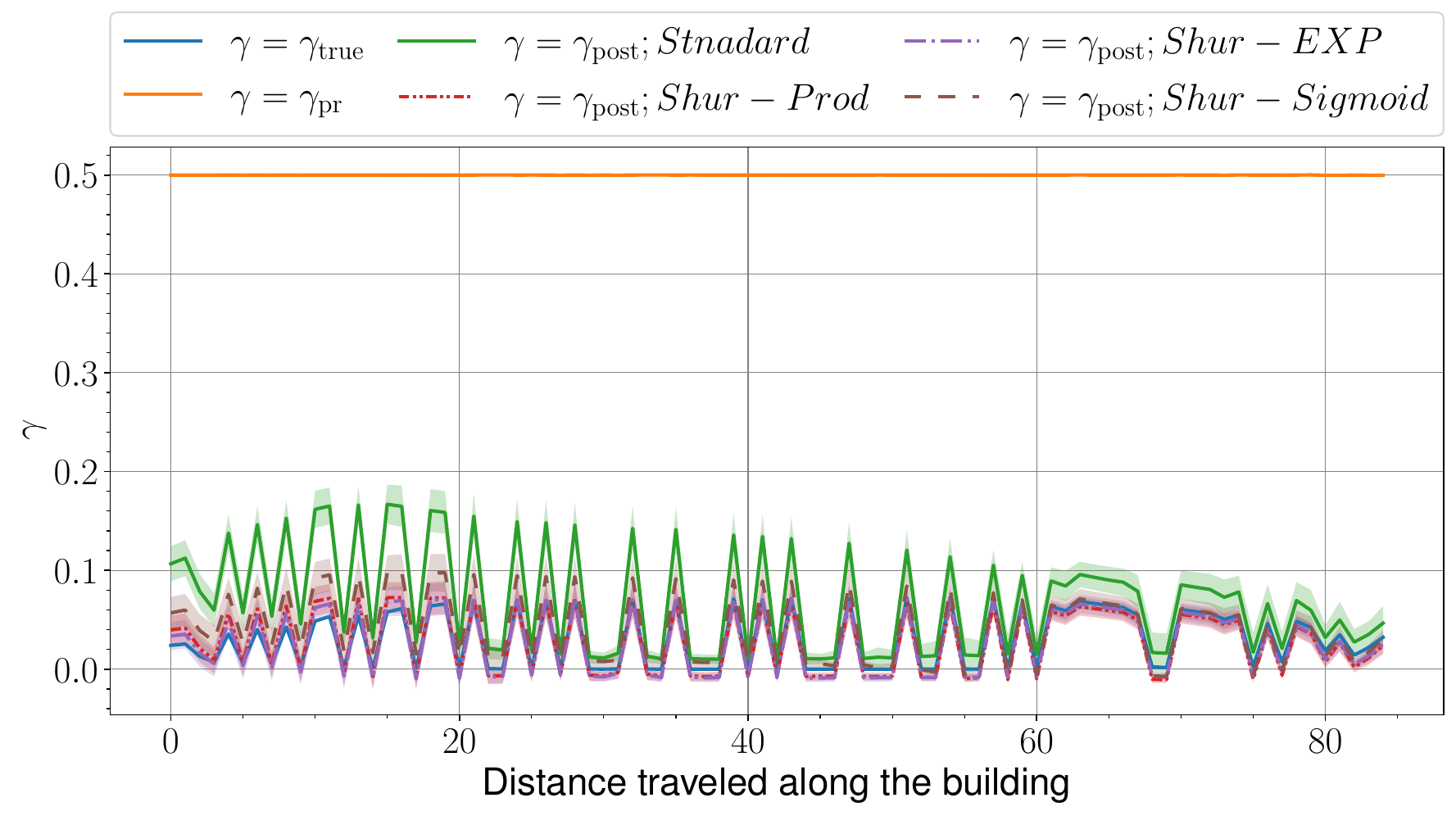}
      \caption{Similar to~\Cref{fig:Restricted_IP_No_Correlation_Max8_Penalty_0.0001}.
          Here, observations are assumed to be correlated with length scale
          $\ell$, with $\ell=1$ (left), and with $\ell=3$ (right), respectively.
          We show results obtained with $\alpha=1e-4$, and $\budget=8$ sensors are activated. 
        }
      \label{fig:Restricted_IP_Correlation_Max8_Penalty_0.0001}
    \end{figure}

  The numerical results presented in this section reveal that solving an OED problem for 
  sensor placement before carrying out Bayesian inversion is an essential step to 
  guarantee an optimal deployment of observational grid under a limited budget.
  While the traditional formulation of the OED problem is useful, its performance can be 
  enhanced by employing the generalized formulation suggested in this work, especially 
  in the presence of observation correlations. Specifically, while it can be
  tempting for simplicity to ignore observation correlations in the process
  of solving an OED problem, one will likely end up with a suboptimal design
  that can lead to erroneous solution of the inverse problem.

\section{Conclusion}
\label{sec:Conclusions}
  In this work we presented a generalized approach for optimal experimental
  design  for linear Bayesian inverse problems where the measurement 
  errors are generally correlated. 
  This study supplements the fast-evolving literature on OED for the Bayesian inverse problem; it provides an extended mathematical
  formulation of the most popular optimality criteria, as well as the associated
  gradients, essential for the numerical solution of the optimization
  problem.
  The proposed formulation follows a Hadamard product approach to formulate the 
  weighted likelihood, which is then used in the optimality criterion. 
  This approach provides a clear understanding of the effect of the design
  on the measurements and the covariances of observational errors and is 
  valid in both finite- and infinite-dimensional settings of Bayesian linear 
  inverse problems.
  We show that the traditional formulation of the OED 
  problem is a special case of the proposed approach in the case of uncorrelated
  observations, where the observation precision matrix is weighted by the
  relaxed design.
  The Hadamard product formulation is shown to be accurate and more flexible, especially for
  handling spatiotemporal correlations.
  We provide multiple candidates of the weighting function that evaluates the
  relative importance of observation covariances.
  Our numerical results show that the proposed formulation 
  achieves better results than by using the traditional formulation, that is,
  by pre- and postmultiplication of the observation error precision matrix with
  the relaxed design.
  All weighting functions investigated here achieved similar performance in
  the presence of observation correlations with larger length scale, that is, by
  allowing more sensors to be correlated.
  By using a logistic sigmoid function for covariance weighting, however, the
  OED optimization problem transforms into an unconstrained optimization
  problem, which is generally easier to solve than the traditional OED problem
  with box constraints.

  The main limitation of the proposed approach
  is the requirement of  differntiability of the regularization term.
  In this work we follow the common practice of approximating the
  sparsification-enforcing penalty $\ell_0$ with a penalty function based on
  $\ell_1$ norm, which is differentiable given that the weights fall in the
  interval $[0, 1]$.
  A recent approach that does not require differntiability of the objective function with
  respect to the design, and thus enables sparsification penalties such as
  $\ell_0$, is proposed in~\cite{attia2022stochastic}.
  However, this approach does not apply any relaxation to the design variables 
  and is thus out of the scope of this work.

  While we provided the mathematical formulation of the approach for A- and
  D-optimal designs, we focused our numerical experiments on A-optimality in the
  presence of spatial correlations.
  Extensions of the mathematical formulation and empirical studies of 
  D-optimal designs in the presence of correlations in space and time
  domains are still required.
  Algorithmic approaches such as the standard greedy swapping
  (exchange)
  algorithm~\cite{fedorov2013theory,pazman1974convergence,wynn1970sequential}
  can be used to seek a local optimum of the OED
  objective; however, as with the coordinate descent
  optimization approach, 
  it can be computationally expensive for increasing 
  cardinality of the design space.
  This work focused on extending the popular relaxation approach 
  for solving binary OED formulation where gradient-based optimization routines
  are utilized to find a local minimum of the OED objective.
  Nevertheless, in future work we plan to provide empirical comparisons to assess the
  quality and computational cost of various OED formulations and solution
  approaches.

\appendix

\section{Proofs of Theorems and Lemmas Discussed in~\Cref{sec:OED_Correlated_OBS}}
\label{app:Proofs}
  \begin{proof}[Proof of~\Cref{lemma:pseudo_inverse_prepost}]
    Let $\mat{A}:=\designmat \Cobsnoise\designmat\tran$ 
    and 
    $\mat{A}^{\dagger}:= \proj\tran
           \left( \proj \designmat
             \Cobsnoise \designmat\tran \proj\tran
           \right)\inv \proj
    $, where we dropped the dependency of $\proj$ and $\designmat$ for
    clarity.
    By the definition of $\proj$ and $\designmat$, it follows that 
    $\proj\tran\proj\designmat=\designmat=\designmat\tran=\designmat\tran\proj\tran\proj$,
    and hence the following hold:
    \begin{enumerate}
      \item[i] 
        $\mat{A} \mat{A}^{\dagger} 
        = 
        \proj\tran\proj\designmat \Cobsnoise\designmat\tran
        \proj\tran
           \left( \proj \designmat
             \Cobsnoise \designmat\tran\proj \tran
           \right)\inv \proj 
           = \proj\tran\proj$.

         \item[ii] $\mat{A}^{\dagger} \mat{A} = 
        \proj\tran
           \left( \proj \designmat
             \Cobsnoise \designmat\tran\proj\tran
           \right)\inv \proj 
            \designmat \Cobsnoise\designmat\tran
           \proj\tran \proj
           = \proj\tran\proj$.
    \end{enumerate}

    It follows from (i), (ii) that the four conditions of a pseudo inverse matrix are then satisfied:
    \begin{enumerate}
      \item $\mat{A}\mat{A}^{\dagger} \mat{A} = \proj\tran\proj\mat{A} 
        = \proj\tran \proj \designmat \Cobsnoise\designmat\tran = \designmat
          \Cobsnoise\designmat\tran = \mat{A}\,.$
      
        \item $\mat{A}^{\dagger} \mat{A} \mat{A}^{\dagger} 
          = \proj\tran\proj \mat{A}^{\dagger} 
          = \proj\tran\proj \proj\tran 
            \left( \proj \designmat \Cobsnoise \designmat\tran \proj\tran \right)\inv \proj
          = \proj\tran
            \left( \proj \designmat \Cobsnoise \designmat\tran \proj\tran \right)\inv \proj
          = \mat{A}^{\dagger}\,,$ where we used the fact that
          $\proj\proj\tran=\mat{I}$, the identity matrix.

        \item $\left(\mat{A} \mat{A}^{\dagger}\right)\tran
          = \left(\proj\tran \proj\right)\tran
          =\proj\tran \proj =
          \mat{A}\mat{A}^{\dagger}\,.$

        \item $\left(\mat{A}^{\dagger} \mat{A}\right)\tran
          = \left(\proj\tran \proj\right)\tran
          = \proj\tran \proj =
          \mat{A}^{\dagger} \mat{A}$ \,.
    \end{enumerate}
  \end{proof}
  \begin{proof}[Proof of~\Cref{lemma:pseudo_inverse_pointwise}]
    The proof follows easily by noting that
    \begin{equation*}
      \proj\tran(\design)\proj(\design)
        \left(\designmat(\design)\odot \Cobsnoise \right) 
        =
        \designmat(\design)\odot \Cobsnoise 
        = 
        \left(\designmat(\design)\odot \Cobsnoise \right) 
        \proj\tran(\design) \proj(\design)\,,
    \end{equation*}
    and by following the same steps in the proof
    of~\Cref{lemma:pseudo_inverse_prepost}.
  \end{proof}
  \begin{proof}[Proof of~\Cref{lemma:regularization_equivalence}]
    By definition of $\designmat$, if $\design\rightarrow\vec{1}$, then
    $\designmat\rightarrow\mat{1}$ 
    and thus $\wdesignmat\odot\Cobsnoise \rightarrow\Cobsnoise$, and
    $\left(\wdesignmat\odot\Cobsnoise\right)\inv
    \rightarrow\mat{\Gamma}\inv_{\rm noise}$, which proves the first part of
    the statement. 
    To prove the statement for $\design\rightarrow \vec{0}$, 
      by letting $\vec{s}:=\diag{\Cobsnoise}$,  and $\vec{d}(\design):=(d_1(\design_1),\ldots,d_{\Nsens}(\design_{\Nsens}))$
    with
    \begin{equation*}
        d_i(\design_i) := \begin{cases}
        0 \,; & \quad \design_i = 0  \\
        \frac{1 - \design_i^4 }{\design_i^2} \,; & \quad \design_i \neq 0  \,,
      \end{cases}
    \end{equation*}
    one can show that 
    \begin{equation}\label{eqn:noise_regularization}
        \designmat(\design) \odot \Cobsnoise = \Diag{\vec{s} \odot \vec{d}(\design)} +
      \Diag{\design} \Cobsnoise \Diag{\design}\,.
    \end{equation}

      Since the second term of the right hand side of~\eqref{eqn:noise_regularization} is positive semi-definite, by employing L\"{o}wner ordering~\cite{lowner1934monotone,Pukelsheim2006optimal}, it follows that 
        \begin{equation}
            \designmat(\design) \odot \Cobsnoise \succeq \Diag{\vec{s} \odot \vec{d}(\design)} \,,
        \end{equation}
        which implies 
        \begin{equation}
            \left( \designmat(\design) \odot \Cobsnoise \right)\inv 
                \preceq \left(\Diag{\vec{s} \odot \vec{d}(\design)}\right)\inv \,.
        \end{equation}
        Since the trace preserves the L\"{o}wner ordering, it follows that
        \begin{equation}
            0
            \leq 
            \Trace{
                \left(\designmat(\design) \odot \Cobsnoise \right)\inv 
            }
            \leq 
            \Trace{
                \left(\Diag{\vec{s} \odot \vec{d}(\design)}\right)\inv 
            }
            = \sum_{i=1}^{\Nsens}{\frac{\design_i^2}{1-\design_i^4}
                        \cdot \frac{1}{\left(\Cobsnoise\right)_{ii}} } \,.
        \end{equation}

        By letting $\design \rightarrow \vec{0}$ it follows that 
        $
            \Trace{
                \left(\designmat(\design) \odot \Cobsnoise \right)\inv 
            } \rightarrow 0
        $ 
        and hence $\left( \designmat(\design) \odot \Cobsnoise \right)\inv \rightarrow \vec{0}$ which immediately yields
        the desired conclusion.
  \end{proof}
  \begin{proof}[Proof of~\Cref{theorem:continuity}]
    We consider three possible cases.
    First, if $\design\rightarrow \vec{1}$, $w_{i}\rightarrow 1$ for all
    $i=1,\ldots,\Nsens$, then $\designmat(\design)\rightarrow\mat{1}$ and 
    $\wdesignmat(\design)=\left(\designmat(\design)\odot\Cobsnoise\right)\inv
    \rightarrow \Cobsnoise\inv$, resembling the case of activating all sensors.
      This also follows immediately from~\Cref{lemma:regularization_equivalence}.
    Second, if $\design\rightarrow\vec{0}$, then $w_{i, j}=\design_i\design_j
    \rightarrow 0$, and $\frac{1}{w_{i, i}}=\frac{1}{\design_i^2}
    \rightarrow \infty$.
    From~\Cref{lemma:regularization_equivalence} it follows in this case that 
      $\pseudoinv{\designmat(\design)\odot\Cobsnoise} \rightarrow \mat{0}$, 
    which properly represents the case of disabling all sensors. 

    Third, we study the case where $\design\rightarrow\binarydesign$ where
    $\binarydesign$ is a boundary 
    point other than $\vec{0}$ or $ \vec{1}$.
    Let the observation covariance matrix $\Cobsnoise$ and the weighting
    matrix $\designmat$ be represented as block matrices as follows:
    \begin{equation}
      \Cobsnoise = 
        \begin{bmatrix} \mat{A} & \mat{B} \\ \mat{B}\tran & \mat{D}
        \end{bmatrix}\,; \qquad
        \designmat =
        \begin{bmatrix} \mat{W_A} & \mat{W_B} \\ \mat{W}\tran_{\mat{B}} & \mat{W_D}
        \end{bmatrix}\,.  
    \end{equation}
    Without loss of generality, we assume that both
    $\mat{A},\, \mat{W_A}$ correspond to elements of $\binarydesign$ that are
    equal to $0$; permutation can be used to reach this form. 
    In this case, 
    \begin{equation}
      \pseudoinv{\Cobsnoise \odot \designmat(\binarydesign)} 
      = \begin{bmatrix} \mat{0} & \mat{0} \\ \mat{0} & \mat{D}\inv 
      \end{bmatrix}\,.
    \end{equation}

    Note that the weighted precision matrix 
    $\wdesignmat(\design)=\left(\designmat(\design)\odot\Cobsnoise\right)\inv$ for
    $\design\in(0, 1)^{\Nsens}$, that is, for designs in the interior of the
    design variable domain. Thus, we write the precision matrix using block
    inversion as follows
    \begin{equation}\label{eqn:block_inv}
      \begin{aligned}
        \wdesignmat(\design)
        &= \left(\designmat(\design) \odot \Cobsnoise \right)\inv 
        = 
        \begin{bmatrix} 
          \mat{A}\odot\mat{W_A} & \mat{B}\odot\mat{W_B} \\ 
            \mat{B}\tran\odot\mat{W}\tran_{\mat{B}} & \mat{D}\odot\mat{W_D}
        \end{bmatrix}\inv
        = 
        \begin{bmatrix} 
          \widetilde{\mat{A}}  &  \widetilde{\mat{B}}  \\ 
            \widetilde{\mat{B}}\tran  &  \widetilde{\mat{D}}  
        \end{bmatrix}\inv  \\
        &= 
        \begin{bmatrix}
          \left(\widetilde{\mat{A}} - \widetilde{\mat{B}} {\widetilde{\mat{D}}}\inv \widetilde{\mat{B}}\tran \right)\inv   
              & - \left(\widetilde{\mat{A}} - \widetilde{\mat{B}} {\widetilde{\mat{D}}}\inv \widetilde{\mat{B}}\tran \right)\inv  
              \widetilde{\mat{B}} {\widetilde{\mat{D}}} \inv
              \\ 
              -   \widetilde{\mat{D}}\inv \widetilde{\mat{B}}\tran 
              \left(\widetilde{\mat{A}} - \widetilde{\mat{B}} {\widetilde{\mat{D}}}\inv \widetilde{\mat{B}}\tran \right)\inv   
            &  
            \widetilde{\mat{D}}\inv + 
            \widetilde{\mat{D}}\inv \widetilde{\mat{B}}\tran 
              \left(\widetilde{\mat{A}} - \widetilde{\mat{B}} {\widetilde{\mat{D}}}\inv \widetilde{\mat{B}}\tran \right)\inv  
              \widetilde{\mat{B}} {\widetilde{\mat{D}}\inv} 
        \end{bmatrix}  \,.
      \end{aligned}
    \end{equation}

    It follows from the first case (also from~\Cref{lemma:regularization_equivalence}) that 
    $\widetilde{\mat{D}}\inv \rightarrow \mat{D}\inv$
    for $\design \rightarrow \binarydesign$.
    Note that $\left(\widetilde{\mat{A}} - \widetilde{\mat{B}} {\widetilde{\mat{D}}}\inv
        \widetilde{\mat{B}}\tran\right)\inv$ appears in all blocks of $\wdesignmat(\design)$. Thus,
    to achieve the desired result it suffices to show that
    $\left(\widetilde{\mat{A}} - \widetilde{\mat{B}} {\widetilde{\mat{D}}}\inv
        \widetilde{\mat{B}}\tran\right)\inv \rightarrow \mat{0}$ for $\design \rightarrow \binarydesign$, which implies
    $ \wdesignmat(\design) \rightarrow \pseudoinv{\Cobsnoise \odot \designmat(\binarydesign)}  $ for $\design\rightarrow\binarydesign$.
    To this end, let us write $\design:=(\design_{A}\tran, \design_{B}\tran)\tran$ and $\vec{s}:=(\vec{s}_{A}\tran, \vec{s}_{B}\tran)\tran$.
    Thus, from both~\eqref{eqn:noise_regularization} and ~\eqref{eqn:block_inv} it follows that 
    \begin{equation}\label{eqn:block_expansion}
      \begin{aligned}
        \designmat(\design) \!\odot\! \Cobsnoise 
        &= 
        \begin{bmatrix} 
          \mat{A}\!\odot\!\mat{W_A} & \mat{B}\!\odot\!\mat{W_B} \\ 
            \mat{B}\tran\!\odot\!\mat{W}\tran_{\mat{B}} & \mat{D}\!\odot\!\mat{W_D}
        \end{bmatrix}
        = 
        \begin{bmatrix} 
          \widetilde{\mat{A}}  &  \widetilde{\mat{B}}  \\ 
            \widetilde{\mat{B}}\tran  &  \widetilde{\mat{D}}  
        \end{bmatrix}  \\
        &= 
          \Diag{ \vec{s} \!\odot\! \vec{d}(\design) }   
        + 
          \Diag{\design} \Cobsnoise \Diag{\design} \\
        &= 
        \begin{bmatrix} 
            \Diag{ \vec{s}_{A} \!\odot\! \vec{d}(\design_{A}) }    &   \mat{0}  \\ 
            \mat{0} &   \Diag{ \vec{s}_{B} \!\odot\! \vec{d}(\design_{B}) }  
        \end{bmatrix}  
        + 
        \begin{bmatrix} 
            \mat{A} \!\odot\! \left(\design_{A}\design_{A}\tran\right) & \mat{B} \!\odot\! \left(\design_{A}\design{B}\tran\right)  \\ 
            \mat{B}\tran \!\odot\! \left(\design_{B}\design_{A}\tran\right) & \mat{D} \!\odot\! \left(\design_{B}\design{B}\tran\right)  
        \end{bmatrix}  
          \,.
      \end{aligned}
    \end{equation}
    Since for $\design \succ \vec{0}\,,\ \Cobsnoise \succ \mat{0}$ it holds that $\Diag{\design} \Cobsnoise \Diag{\design} \succ \mat{0} $, and given that 
    $\vec{0} \preceq \design \preceq \vec{1}$, it follows from~\eqref{eqn:block_expansion} that $\widetilde{\mat{D}} \succeq \mat{D}\odot\left(\design_{B}\design_{B}\tran \right) $.
    Thus, it holds that 
    \begin{equation}\label{eqn:weight_ineq}
        \begin{bmatrix} 
            \mat{A} \!\odot\! \left(\design_{A}\design_{A}\tran\right) & \mat{B} \!\odot\! \left(\design_{A}\design{B}\tran\right)  \\ 
            \mat{B}\tran \!\odot\! \left(\design_{B}\design_{A}\tran\right) & \widetilde{\mat{D}}   
        \end{bmatrix}  
        =
        \begin{bmatrix} 
            \mat{A} \!\odot\! \left(\design_{A}\design_{A}\tran\right) & \widetilde{\mat{B} }  \\ 
             \widetilde{\mat{B}}\tran  & \widetilde{\mat{D}}  \\ 
        \end{bmatrix} 
        \succeq
        \Diag{\design} \Cobsnoise \Diag{\design} 
        \succ \mat{0} \,.
    \end{equation}
    From~\eqref{eqn:weight_ineq} and by utilizing the Shur complement, it follows that 
    \begin{equation}\label{eqn:Shur_complement_firstblock}
        \mat{A} \!\odot\! \left(\design_{A}\design_{A}\tran\right) - \widetilde{\mat{B} } \widetilde{\mat{D}}\inv \widetilde{\mat{B}}\tran \succ \mat{0} \,.
    \end{equation}
    By definition (see~\eqref{eqn:block_expansion}) $\widetilde{\mat{A}}=\Diag{\vec{s}_A \odot \vec{d}(\design_{A})} + \mat{A}\odot\left(\design_{A}\design_{A}\tran \right)$.
    Thus,
    \begin{equation}
        \widetilde{\mat{A}} - \widetilde{\mat{B}} {\widetilde{\mat{D}}}\inv \widetilde{\mat{B}}\tran
        =
        \Diag{\vec{s}_A \!\odot\! \vec{d}(\design_{A})} + \mat{A}\!\odot\!\left(\design_{A}\design_{A}\tran \right)
            - \widetilde{\mat{B}} {\widetilde{\mat{D}}}\inv \widetilde{\mat{B}}\tran \\
            \stackrel{\eqref{eqn:Shur_complement_firstblock}}{\succ}  
            \Diag{\vec{s}_A \!\odot\! \vec{d}(\design_{A})} \succ \mat{0}  \,,
    \end{equation}
    which, by taking the inverse and applying the trace, implies
    $0  \leq 
        \Trace{\left(\widetilde{\mat{A}} - \widetilde{\mat{B}} {\widetilde{\mat{D}}}\inv \widetilde{\mat{B}}\tran \right)\inv}
        \leq
        \Trace{ \left( \Diag{\vec{s}_A \!\odot\!
            \vec{d}(\design_{A})}\right)\inv}$ and $\Trace{ \left( \Diag{\vec{s}_A \!\odot\!
            \vec{d}(\design_{A})}\right)\inv} \rightarrow 0$ for
        $ \design \rightarrow \binarydesign$.
    That proves that 
    $\left(\widetilde{\mat{A}} - \widetilde{\mat{B}} {\widetilde{\mat{D}}}\inv
        \widetilde{\mat{B}}\tran\right)\inv \rightarrow \mat{0}$ for $\design \rightarrow \binarydesign$.
    Note that although $\design$ whose limits are 
    considered here always remain in $(0,1)^{\Nsens}$, this is by no means restrictive, 
    since the continuous mapping $\design\rightarrow \pseudoinv{\designmat(\design)\odot \Cobsnoise}$
    with domain $(0,1)^{\Nsens}$ uniquely extends to a continuous mapping with domain 
    $[0,1]^{\Nsens}$ being the closure of $(0,1)^{\Nsens}$.
  \end{proof}
  \begin{proof}[Proof of~\Cref{lemma:gradient_continuity}]
    It is enough to discuss the proof assuming the weighting function~\eqref{eqn:cov-prod_kernel}.
    For $\design \in (0, 1]^{\Nsens}$, the weighted precision matrix is
    invertible, and continuity of the gradient follows immediately.
    For $\design_i\rightarrow 0$, 
    $\wdesignmat(\design) \rightarrow \Oh{\weightfunc_i^2}$ 
    and $\vec{\eta_i}\rightarrow \Oh{\weightfunc_i^{-3}}$.
    Thus, as $\design_i \rightarrow 0$,
    $\wdesignmat(\design) \vec{e}_i \vec{\eta_i}\tran \wdesignmat(\design) \rightarrow
    \Oh{\weightfunc_i^2\weightfunc_i^{-3}\weightfunc_i^2}=\Oh{\weightfunc}\rightarrow 0$.
    This matches the definition of the derivative defined
    by~\eqref{eqn:weights_derivative_space} for $\weightfunc_i=0$.
  \end{proof}
  %

\section{A-Optimality Criterion and Gradient}
\label{app:A_Optimality}
Here we detail the derivation of the formulae of the gradient of the
A-optimality criterion~\eqref{eqn:A_optimality_Schur} with respect to the relaxed
design. 
Recall that for each $i=1,2,\ldots,\Nsens$
\begin{equation}\label{eqn:A_opt_deriv_1}
  \begin{aligned}
    \del{ \Psi^{\GA}(\design) }{\design_i } 
      &= \Trace{ \del{\Predmat \Cparampostmat(\design) \Predmat\adj }{\design_i} }  
      = - \Trace{ \Predmat \Hessmat\inv(\design)  \Fadj
      \del{\wdesignmat(\design)}{\design_i} \F  \Hessmat\inv(\design) \Predmat\adj }  \,.
  \end{aligned}
\end{equation}
%

\subsection{Space correlations}\label{app:A_Optimality_Space}
  In the case of temporally uncorrelated observations, by
  utilizing~\eqref{eqn:weighted_obs_derivative_space}, then for each
  $i=1, 2,\ldots,\Nsens$ we  have
  \begin{equation} \label{eqn:A_opt_deriv_2}
    \begin{aligned}
      \del{ \Psi^{\GA}\!(\design) }{\design_i} 
        &= \Trace{\! \Predmat \Hessmat\inv\!(\design)  
          \Fadj \wdesignmat(\design) 
              \directsum{m=1}{\nobstimes}{\! 
                \vec{e}_i \!\left(\! \left(\mat{R}_m \vec{e}_i\right) 
                 \!\odot\! \vec{\eta}_i \!\right)\tran 
                \!+\! \left(\! \left(\mat{R}_m \vec{e}_i  \!\right) 
                 \!\odot\! \vec{\eta}_i\right) \vec{e}_i\tran 
              \!} 
             \wdesignmat(\design) \F 
          \Hessmat\inv\!(\design) \Predmat\adj \!}  \\
        &= 2\, \Trace{ \Predmat \Hessmat\inv(\design)  
          \Fadj  \wdesignmat(\design)
              \directsum{m=1}{\nobstimes}{
                \left(\left(\mat{R}_m \vec{e}_i\right) \odot \vec{\eta}_i\right) \vec{e}_i\tran   }
           \wdesignmat(\design)\F 
          \Hessmat\inv(\design) \Predmat\adj }  \,,
    \end{aligned}
  \end{equation}
  where the vector of weight derivatives $\vec{\eta}_i$ is given
  by~\eqref{eqn:weights_derivative_space}.
  Since observation errors are assumed to be temporally uncorrelated 
  (i.e., across observation time instances $t_m$), 
  then, given~\eqref{eqn:A_opt_deriv_2}, the gradient of 
  the A-optimality criterion can be written in terms of 
  cardinality vectors  $\vec{e}_i\in \Rnum^{\Nsens}$,
  as $\nabla_{\design}{ \Psi^{\GA}(\design) }= \sum_{i=1}^{\Nsens} \del{
    \Psi^{\GA}(\design) }{\design_i}\, \vec{e}_{i} $, 
  resulting in the following form: 
  \begin{equation}\label{eqn:A_opt_deriv_3}
    \begin{aligned}
      \nabla_{\design}{ \Psi^{\GA}(\design) }
        &= 2  \sum_{i=1}^{\Nsens} \vec{e}_i 
        \Trace{ \Predmat \Hessmat\inv(\design)  
          \Fadj \wdesignmat(\design) 
            \left( 
              \directsum{m=1}{\nobstimes}{\left(\left(\mat{R}_m \vec{e}_i\right) 
                \!\odot\! \vec{\eta}_i\right) \vec{e}_i\tran  }
            \right) 
           \wdesignmat(\design) \F 
          \Hessmat\inv(\design) \Predmat\adj }  \\
        &= 2 \sum_{j=1}^{\Nsens} \vec{e}_i 
          \Trace{ \Predmat \Hessmat\inv(\design)  \sum_{m=1}^{\nobstimes}{
            \Fadjind{0}{m} \mat{V}_m^{\dagger}(\design)
              \left(\left(\mat{R}_m \vec{e}_i\right) 
                \!\odot\! \vec{\eta}_i\right) \vec{e}_i\tran   
            \mat{V}_m^{\dagger}(\design) \Find{0}{m} 
            \Hessmat\inv(\design) \Predmat\adj }  
          }  \\
        &= 2 \sum_{m=1}^{\nobstimes} { \sum_{j=1}^{\Nsens} \vec{e}_i 
          \Trace{ 
            \vec{e}_i \tran  
            \mat{V}_m^{\dagger}(\design) \Find{0}{m} 
              \Hessmat\inv(\design) \Predmat\adj \Predmat \Hessmat\inv(\design)  
            \Fadjind{0}{m} \mat{V}_m^{\dagger}(\design)
            \left(\left(\mat{R}_m \vec{e}_i\right) 
                \!\odot\! \vec{\eta}_i\right) \vec{e}_i\tran 
            }  
          }  \\
        &= 2 \sum_{j=1}^{\Nsens} \vec{e}_i \vec{e}_i \tran 
            \sum_{m=1}^{\nobstimes} { 
            \mat{V}_m^{\dagger}(\design) \Find{0}{m} 
              \Hessmat\inv(\design) \Predmat\adj \Predmat \Hessmat\inv(\design)  
            \Fadjind{0}{m} \mat{V}_m^{\dagger}(\design)
            \left(\left(\mat{R}_m \vec{e}_i\right) 
                \!\odot\! \vec{\eta}_i\right) \vec{e}_i\tran 
          }  \,,
    \end{aligned}
  \end{equation}
  where as defined by~\eqref{eqn:spacial_weighted_noise},
  \begin{equation}\label{eqn:spacial_weighted_noise_supp}
    \wdesignmat(\design) 
      = \directsum{m=1}{\nobstimes}{ \mat{V}_m^{\dagger}(\design)  } \,,
      \quad
      \mat{V}_m(\design) 
      := 
          \mat{R}_m \odot 
          \left( \sum_{i,j=1}^{\Nsens}{ \varweightfunc(\design_i,\,\design_j) 
            \vec{e}_i \vec{e}_j\tran }   \right)
        \,.
  \end{equation}
  We used the circular property of matrix trace and the fact that, 
  for a symmetric matrix $\mat{A}$ and a vector $\vec{y}$ with conformable shapes. 
  Then 
  $\mat{A} \odot \left( \vec{e}_i \vec{y}\tran \right) 
    = \vec{e}_i \left( \left( \vec{e}_i \tran \mat{A}\right) \odot \vec{y}\tran \right)
    = \vec{e}_i \left( \left( \mat{A} \vec{e}_i \right) \odot \vec{y}\right)
    \tran  \,.
  $
  By utilizing the matrix of weights derivatives $\designmat^{\prime}$ defined
  by~\eqref{eqn:designmat_derivative}, we can refine \eqref{eqn:A_opt_deriv_3}  to
  \begin{equation}\label{eqn:A_opt_deriv_4}
    \nabla_{\design}{ \Psi^{\GA}(\design) }
      = 2 \sum_{m=1}^{\nobstimes} { 
      \diag{
        \mat{V}_m^{\dagger}(\design) \Find{0}{m} 
          \Hessmat\inv(\design) \Predmat\adj \Predmat \Hessmat\inv(\design)  
        \Fadjind{0}{m} \mat{V}_m^{\dagger}(\design)
          \left( \mat{R}_m \odot \designmat^{\prime} \right)
        } } \,.
  \end{equation}
  %

\subsection{Spatiotemporal correlations}\label{app:A_Optimality_SpaceTime}
  In the presence of spatiotemporal correlations, the derivative of the 
  A-optimality criterion is obtained as follows. For each
  $i=1, 2,\ldots,\Nsens$
  \begin{equation}
    \begin{aligned}
      \del{ \Psi^{\GA}(\design) }{\design_i} 
        &= - \Trace{ \Predmat \Hessmat\inv(\design)  \Fadj
        \del{\wdesignmat(\design)}{\design_i} \F  \Hessmat\inv(\design)  \Predmat\adj }  \\
        &= 2 \Trace{ \Predmat \Hessmat\inv(\design)  
        \Fadj \wdesignmat(\design)
        \left(
          \sum_{m=1}^{\nobstimes}
          \left(
              \left( \Cobsnoise \vec{e}_q \right) \odot \vec{\vartheta}_{i, m}
            \right) \vec{e}_q \tran
        \right)
        \wdesignmat(\design) \F  
        \Hessmat\inv(\design) \Predmat\adj }  \\
        &= 2 \Trace{ \vec{e}_q \tran
        \wdesignmat(\design) \F  
        \Hessmat\inv(\design) \Predmat\adj  \Predmat \Hessmat\inv(\design) 
        \Fadj \wdesignmat(\design)
          \sum_{m=1}^{\nobstimes}
          \left(
              \left( \Cobsnoise \vec{e}_q \right) \odot \vec{\vartheta}_{i, m}
            \right) }  \\
        &= 2 \vec{e}_q \tran
        \wdesignmat(\design) \F  
        \Hessmat\inv(\design) \Predmat\adj  \Predmat \Hessmat\inv(\design) 
        \Fadj \wdesignmat(\design)
          \sum_{m=1}^{\nobstimes}
          \left(
              \left( \Cobsnoise \vec{e}_q \right) \odot \vec{\vartheta}_{i, m}
            \right) 
        \,,
    \end{aligned}
  \end{equation}
  where $q={i+(m\!-\!1)\Nsens}$ and the vector of weights derivatives 
  $\vec{\vartheta}_{i,m}$ is defined by~\eqref{eqn:weights_derivative_spacetime}.
  The gradient of the A-optimality criterion can be written 
  as 
  \begin{equation}
    \begin{aligned}
      \nabla_{\design}{ \Psi^{\GA}(\design) }
        &
        = 2 \sum_{i=1}^{\Nsens} \vec{e}_{i}
          \vec{e}_{i+(n-1)\Nsens} \tran
            \wdesignmat(\design) \F  
            \Hessmat\inv(\design) \Predmat\adj  \Predmat \Hessmat\inv(\design)  
            \Fadj \wdesignmat(\design)
              \sum_{m=1}^{\nobstimes} \left(
              \left( \Cobsnoise \vec{e}_q \right) \!\odot\! \vec{\vartheta}_{i, m} \right) \\
        &= 2 \sum_{i=1}^{\Nsens} \sum_{m=1}^{\nobstimes}
          \vec{e}_{i} \vec{e}_q \tran
            \wdesignmat(\design) \F  
            \Hessmat\inv(\design) \Predmat\adj  \Predmat \Hessmat\inv(\design)  
            \Fadj \wdesignmat(\design)
              \left(
              \left( \Cobsnoise \vec{e}_q \right) \!\odot\! \vec{\vartheta}_{i, n}
            \right) 
        \,.
    \end{aligned}
  \end{equation}
  %

\subsection{A-optimal design with randomized trace estimator}
  \label{app:A_Optimality_Trace}
  A randomized approximation of the A-optimality criterion 
  $\widetilde{\Psi} ^{\GA} (\design) \approx \Psi^{\GA}(\design)$,
  as defined by~\eqref{eqn:randomized_A_optimality}, takes the form
  \begin{equation}
    \begin{aligned}
      \widetilde{\Psi}^{\GA} (\design) 
        &= \! \frac{1}{n_r} \sum_{r=1}^{ n_r } {\vec{z}_r\tran
        \Cpredpostmat(\design) \vec{z}_r } 
        = \! \frac{1}{n_r} \sum_{r=1}^{ n_r } {\vec{z}_r\tran
        \Hessmat\inv(\design) \vec{z}_r } \\
        &=\! \frac{1}{n_r} \sum_{r=1}^{ n_r } \vec{z}_r \tran 
          \Predmat \left( \F\adj \wdesignmat(\design) \F
            \!+\! \Cparampriormat\inv \right)\inv \Predmat\adj \vec{z}_r 
            \,,
    \end{aligned}
  \end{equation}
  where $\vec{z}_r \in \Rnum^{\Npred} $ and $\wdesignmat(\design) $ is given
  by~\eqref{eqn:Schur_weighted_joint_likelihood}.
  The gradient of this criterion follows
  directly as 
  \begin{equation}
    \del{ \widetilde{\Psi}^{\GA} (\design) }{\design_i}
      = -\frac{1}{n_r} \sum_{r=1}^{ n_r } {\vec{z}_r\tran \Predmat \Hessmat(\design)\inv 
        \Fadj \del{\wdesignmat(\design) }{\design_i} \F
        \Hessmat(\design)\inv \Predmat\adj \vec{z}_r} \,.
  \end{equation}

  This formula of the gradient can be refined given the exact formulation of 
  the design matrix $\designmat(\design)$
  and its derivative $\del{\wdesignmat(\design)}{\design_i}$, 
  as discussed before.
  If we assume the observation errors are temporally uncorrelated,
  following the same procedure as in~\Cref{app:A_Optimality_Space}, 
  then we have
  \begin{equation}
    \begin{aligned}
      \del{ \widetilde{\Psi}^{\GA}(\design) } {\design_i}
        &= \frac{1}{n_r} \sum_{r=1}^{ n_r } {\vec{z}_r\tran \Predmat \Hessmat(\design)\inv 
          \Fadj \wdesignmat(\design) 
            \directsum{m=1}{\nobstimes}{ \mat{R}_m \!\odot\! 
              \left( \vec{e}_i \left( \vec{\eta}_i\right) \tran + \vec{\eta}_i \vec{e}_i \tran \right)  
            } 
          \wdesignmat(\design) \F
          \Hessmat(\design)\inv \Predmat\adj \vec{z}_r} \\
        &= \frac{2}{n_r}\, \sum_{r=1}^{ n_r } \sum_{m=1}^{\nobstimes}
          \left( \vec{\xi}\adj_{r,m}\right)\tran 
          \vec{e}_i\left( 
            \left( \mat{R}_m\vec{e}_i \right) \odot \vec{\eta}_i   
          \right)  \tran
            \vec{\xi}_{r,m}  \,,
    \end{aligned}
  \end{equation}
  where $\vec{\eta}_i$ is given by~\eqref{eqn:weights_derivative_space}, the vectors $\vec{\xi}_{r,m}$ and $\vec{\xi}\adj_{r,m}$ are given by
  \begin{equation}
    \vec{\xi}_{r,m} = \mat{V}_m^{\dagger}(\design) \F_{0, m} \Hessmat(\design)\inv \Predmat\adj \vec{z}_r; \quad
    \left( \vec{\xi}\adj_{r,m}\right)\tran = \vec{z}_r\tran \Predmat \Hessmat(\design)\inv 
           \Fadjind{0}{m} \mat{V}_m^{\dagger}(\design) \,,
  \end{equation}
  and $\mat{V}_m(\design)$ is given by~\eqref{eqn:spacial_weighted_noise_supp}.
  The full gradient, written in terms of its components, in this case is
  \begin{equation}
    \begin{aligned}
      \nabla_{\design}{\widetilde{\Psi}^{\GA}(\design)}
        &= \frac{2}{n_r}\sum_{i=1}^{\Nsens} \vec{e}_i
          \sum_{r=1}^{ n_r } \sum_{m=1}^{\nobstimes}
            \left( \vec{\xi}\adj_{r,m}\right)\tran 
            \vec{e}_i\left( 
              \left( \mat{R}_m\vec{e}_i \right) \odot \vec{\eta}_i   
            \right) \tran 
              \vec{\xi}_{r,m}   \\
        &= \frac{2}{n_r} 
        \sum_{r=1}^{ n_r } \sum_{m=1}^{\nobstimes}
            \vec{\xi}\adj_{r,m} \odot \left( 
              \left( \mat{R}_m \odot \designmat^{\prime} \right)  
            \vec{\xi}_{r,m}\right)  \,,
    \end{aligned}
  \end{equation}
  with $\designmat^{\prime}$ defined by~\eqref{eqn:designmat_derivative}.
  In the presence of spatiotemporal correlations, the
  gradient is found as follows:
  \begin{equation}
    \begin{aligned}
      \del{ \widetilde{\Psi}^{\GA}(\design) } {\design_i}
        &= \frac{1}{n_r} \sum_{r=1}^{ n_r } 
          \vec{z}_r\tran \Predmat \Hessmat\inv\!(\design)
          \Fadj \wdesignmat(\design) \! \left(\! 
            \Cobsnoise \!\odot \!
              \sum_{m=1}^{\nobstimes}\!
                \left( 
                  \vec{e}_q \vec{\vartheta}_{i, m}\tran 
                    \!+\! \vec{\vartheta}_{i, m} \vec{e}_q \tran
                \right)\! 
          \right) \! \wdesignmat(\design) \F
          \Hessmat\inv\!(\design) \Predmat\adj \vec{z}_r \\
        &= \frac{2}{n_r} \, \sum_{r=1}^{ n_r } {\vec{z}_r\tran \Predmat
        \Hessmat\inv\!(\design) 
          \Fadj \wdesignmat(\design) \left( 
            \Cobsnoise \!\odot \!
              \sum_{m=1}^{\nobstimes}
                \left( 
                  \vec{e}_q \vec{\vartheta}_{i, m}\tran 
                \right)  
          \right) \wdesignmat(\design)  \F 
          \Hessmat\inv\!(\design) \Predmat\adj \vec{z}_r} \\
        &= \frac{2}{n_r}\, \sum_{r=1}^{ n_r } 
          \sum_{m=1}^{\nobstimes}
            \vec{\psi}\adj 
              \vec{e}_q \left( 
                \left( \Cobsnoise \vec{e}_q\right) \odot \vec{\vartheta}_{i, m} 
              \right) \tran  
          \vec{\psi}_r \,,
    \end{aligned}
  \end{equation}
  where $q={i+(m\!-\!1)\Nsens}$,  
  $\vec{\psi}_r = \wdesignmat(\design) \F \Hessmat(\design)\inv \Predmat\adj \vec{z}_r$, and 
  $\vec{\psi}_r\adj = \vec{z}_r\tran \Predmat \Hessmat(\design)\inv \Fadj \wdesignmat(\design) $.
  The full gradient follows as
  \begin{equation}
    \begin{aligned}
      \nabla_{\design}{\widetilde{\Psi}^{\GA}(\design)}
        &= \frac{2}{n_r}\, \sum_{i=1}^{\Nsens} \vec{e}_i
        \sum_{r=1}^{ n_r } 
          \sum_{m=1}^{\nobstimes}
            \vec{\psi}_r\adj 
              \vec{e}_{i+(m-1)\Nsens} \left( 
                \left( \Cobsnoise \vec{e}_{i+(m-1)\Nsens}\right) \odot \vec{\vartheta}_{i, m} 
              \right) \tran  
          \vec{\psi}_r  \\
        &= \frac{2}{n_r}\, \sum_{r=1}^{ n_r }
        \sum_{i=1}^{\Nsens}  
          \sum_{m=1}^{\nobstimes}
            \vec{e}_i \vec{\psi}_r\adj 
              \vec{e}_{i+(m-1)\Nsens} \left( 
                \left( \Cobsnoise \vec{e}_{i+(m-1)\Nsens}\right) \odot \vec{\vartheta}_{i, m} 
              \right) \tran  
          \vec{\psi}_r \,. \\
    \end{aligned}
  \end{equation}
  %

\section{D-Optimality Criterion and Gradient}
\label{app:D_Optimality}
  When the D-optimality is set as the OED criterion, 
  that is, by defining the optimal design using~\eqref{eqn:D_optimality_Schur}, 
  the derivative of the optimization objective w.r.t the design variables $\design_i,\,i=1,2,\ldots,\Nsens\,,$ is
  \begin{equation}\label{eqn:D_opt_deriv_2}
    \begin{aligned}
      \del{ \Psi^{\GD}(\design) }{\design_i } 
        &= \Trace{ \Cpredpost \inv(\design)\,  
          \del{ \Predmat \Cparampostmat(\design) \Predmat\adj }{\design_i } } \\ 
        &= - \Trace{ \Cpredpost \inv(\design) \,
          \Predmat \Hessmat\inv(\design)  \Fadj \del{\wdesignmat(\design)}{\design_i} \F  \Hessmat\inv(\design) \Predmat\adj  } \,.
    \end{aligned}
  \end{equation}
  %
  
  \subsection{Space correlations}
  \label{app:D_Optimality_Space}
    For the sake of derivation, we define the Cholesky factorization of the 
    prediction covariance as 
    $\Cpredpost(\design) = \Cpredpost^{1/2} \Cpredpost^{T/2}$, with $\Cpredpost^{1/2}$ 
    being the lower triangular factor. 
    Moreover, the dimensionality of the prediction QoI is generally small, and such 
    factorization if needed is inexpensive.
    Then, $\Cpredpost\inv(\design) = \Cpredpost^{-1/2} \Cpredpost^{-T/2}$ 
    with $\Cpredpost^{-1/2}=\left(\Cpredpost^{1/2}\right)\inv$.
    The gradient of the D-optimality criterion follows as 
    \begin{equation}\label{eqn:D_opt_deriv_3}
      \begin{aligned}
        \nabla_{\design}{ \Psi^{\GD}(\design) }
          &= - \sum_{i=1}^{\Nsens} \vec{e}_i \Trace{ \Cpredpost\inv\!(\design)
          \Predmat \Hessmat\inv\!(\design)  \Fadj 
              \del{\wdesignmat(\design)}{\design_i} 
            \F  \Hessmat\inv\!(\design) \Predmat\adj  } \\
          &= \!\sum_{i=1}^{\Nsens} \vec{e}_i \Trace{ \Cpredpost^{-\frac{T}{2}}  
            \Predmat \Hessmat\inv\!(\design)  
            \Fadj \wdesignmat(\design) 
                \directsum{m=1}{\nobstimes}{\!\mat{R}_m \!\odot\! \left(\! \vec{e}_i
               \! \left( \vec{\eta}_i\right)\! \tran  \!+\! \vec{\eta}_i
                \vec{e}_i\tran\! \right) \! } 
            \wdesignmat(\design) \F 
            \Hessmat\inv\!(\design) \Predmat\adj \Cpredpost^{-\frac{1}{2}}  } \\
          &= 2 \sum_{i=1}^{\Nsens} \vec{e}_i \Trace{ \Cpredpost^{-\frac{T}{2}}  
            \Predmat \Hessmat\inv\!(\design)  
            \Fadj \wdesignmat(\design) 
              \directsum{m=1}{\nobstimes}{\mat{R}_m \!\odot\! \left( \vec{\eta}_i \vec{e}_i \tran \right)  } 
            \wdesignmat(\design) \F 
            \Hessmat\inv\!(\design) \Predmat\adj \Cpredpost^{-\frac{1}{2}} }  \\
          &= 2 \sum_{i=1}^{\Nsens} \vec{e}_i \sum_{m=1}^{\nobstimes}
          \Trace{ \Cpredpost^{-\frac{T}{2}}  
            \Predmat \Hessmat\inv\!(\design)  
            \Fadjind{0}{m} \mat{V}_m^{\dagger}\!(\design) 
              \left( \left( \mat{R}_m \vec{e}_i \right) \!\odot\! \vec{\eta}_i \right) \vec{e}_i \tran 
            \mat{V}_m^{\dagger}\!(\design) \Find{0}{m} 
            \Hessmat\inv\!(\design) \Predmat\adj \Cpredpost^{-\frac{1}{2}} }  \\
          &= 2 \sum_{i=1}^{\Nsens} \vec{e}_i \vec{e}_i \tran
          \sum_{m=1}^{\nobstimes}
            \mat{V}_m^{\dagger}(\design) \Find{0}{m} 
            \Hessmat\inv\!(\design) \Predmat\adj  \Cpredpost\inv\!(\design) 
            \Predmat \Hessmat\inv(\design)  
            \Fadjind{0}{m} \mat{V}_m^{\dagger}(\design) 
              \left( \left( \mat{R}_m \vec{e}_i \right) \!\odot\! \vec{\eta}_i \right)  \\
          &= 2 \sum_{m=1}^{\nobstimes} \diag{
            \mat{V}_m^{\dagger}(\design) \Find{0}{m} 
            \Hessmat\inv\!(\design) \Predmat\adj \Cpredpost\inv\!(\design)
            \Predmat \Hessmat\inv(\design)  
            \Fadjind{0}{m} \mat{V}_m^{\dagger}(\design) 
              \left( \mat{R}_m \!\odot\! \designmat^{\prime} \right) } \,,
      \end{aligned}
    \end{equation}
    where we used the circular property of the matrix trace and  the fact that the trace
    is invariant under matrix transposition and $\mat{V}_m^{\dagger}(\design)$
    is given by~\eqref{eqn:spacial_weighted_noise_supp}.

  \subsection{Spatiotemporal correlations}
  \label{app:D_Optimality_SpaceTime}
  The gradient of the D-optimality criterion in this case is
  \begin{equation}
    \begin{aligned}
      \nabla_{\design}{ \Psi^{\GD}(\design) }
        &= - \sum_{i=1}^{\Nsens} \vec{e}_i 
          \Trace{ \Cpredpost\inv \Predmat \Hessmat\inv\!(\design)  \Fadj 
          \del{\wdesignmat(\design)}{\design_i} 
          \F  \Hessmat\inv(\design) \Predmat\adj  } \\
        &=\! \sum_{i=1}^{\Nsens} \vec{e}_i 
          \Trace{\! \Cpredpost\inv \Predmat \Hessmat\inv\!(\design)  
          \Fadj \wdesignmat(\design)
          \!
            \left(\! 
              \Cobsnoise \!\odot \!
                \sum_{m=1}^{\nobstimes}
                  \left( 
                    \vec{e}_q \vec{\vartheta}_{i, m}\tran 
                      \!+\! \vec{\vartheta}_{i, m} \vec{e}_q \tran
                  \right)  \!
            \right)\!  
          \wdesignmat(\design) \F 
          \Hessmat\inv\!(\design) \Predmat\adj  \!} \\
        &= 2\, \sum_{i=1}^{\Nsens} \vec{e}_i 
          \Trace{ \Cpredpost\inv \Predmat \Hessmat\inv(\design) 
          \Fadj \wdesignmat(\design)
            \left( 
              \Cobsnoise \!\odot\! 
                \sum_{m=1}^{\nobstimes}
                  \left( 
                    \vec{e}_q \vec{\vartheta}_{i, m}\tran 
                  \right)  
            \right)  
          \wdesignmat(\design) \F 
          \Hessmat\inv(\design) \Predmat\adj  } \\
        &= 2\, \sum_{i=1}^{\Nsens} \vec{e}_i \sum_{m=1}^{\nobstimes}
          \Trace{ \Cpredpost\inv \Predmat \Hessmat\inv(\design) 
          \Fadj \wdesignmat(\design)
            \left( 
              \Cobsnoise \odot 
                \left( 
                  \vec{e}_q \vec{\vartheta}_{i, m}\tran 
                \right)  
            \right)  
          \wdesignmat(\design) \F 
          \Hessmat\inv(\design) \Predmat\adj  } \\
        &= 2\, \sum_{i=1}^{\Nsens} \vec{e}_i \sum_{m=1}^{\nobstimes}
          \Trace{ \Cpredpost\inv \Predmat \Hessmat\inv(\design)  
          \Fadj \wdesignmat(\design)
            \vec{e}_q \left( 
              \Cobsnoise \vec{e}_q \odot \vec{\vartheta}_{i, m} 
            \right) \tran  
          \wdesignmat(\design) \F 
          \Hessmat\inv(\design) \Predmat\adj  } \\
        &= 2\, \sum_{i=1}^{\Nsens} \sum_{m=1}^{\nobstimes}
          \vec{e}_i \left( 
              \Cobsnoise \vec{e}_q \odot \vec{\vartheta}_{i, m} 
            \right) \tran  
          \wdesignmat(\design) \F 
          \Hessmat\inv(\design) \Predmat\adj  \Cpredpost\inv \Predmat \Hessmat\inv(\design)  
          \Fadj \wdesignmat(\design)
            \vec{e}_q  
          \,,
    \end{aligned}
  \end{equation}
  where $\vec{e}_q \equiv
  \vec{e}_{i+(m-1)\Nsens}\in\Rnum^{\Nobs};\,i=1,2,\ldots,\Nsens,\,m=1,2,\ldots,\nobstimes$.

\section{Additional Numerical Results}
\label{supp:subsec:numerical_results}
In this section we give additional empirical results to complement the work
presented in~\Cref{sec:numerical_experiments}.

\subsection{Reduced-order approximation of Hessian}
\label{supp:subsec:reduced_Hessian}
Here we give more details about the approach used to develop a reduced-order
approximation of the Hessian operator in the Bayesian inverse problem.
This complements the discussion in~\Cref{sec:numerical_experiments}

Repeated evaluation of Hessian matrix-vector products is computationally demanding.
We use the two-pass algorithm described in~\cite{saibaba2016randomizedHermitCOPY} to generate 
a randomized reduced-order approximation of the Hessian $\Hessmat$.
The number of eigenvalues is set to $80$, and the oversampling parameter is $p=20$.
\Cref{fig:Hessian_Approx} shows the leading $\budget=80$ eigenvalues, with approximation error.
\begin{figure}[ht]
  \centering
  \includegraphics[width=0.70\linewidth]{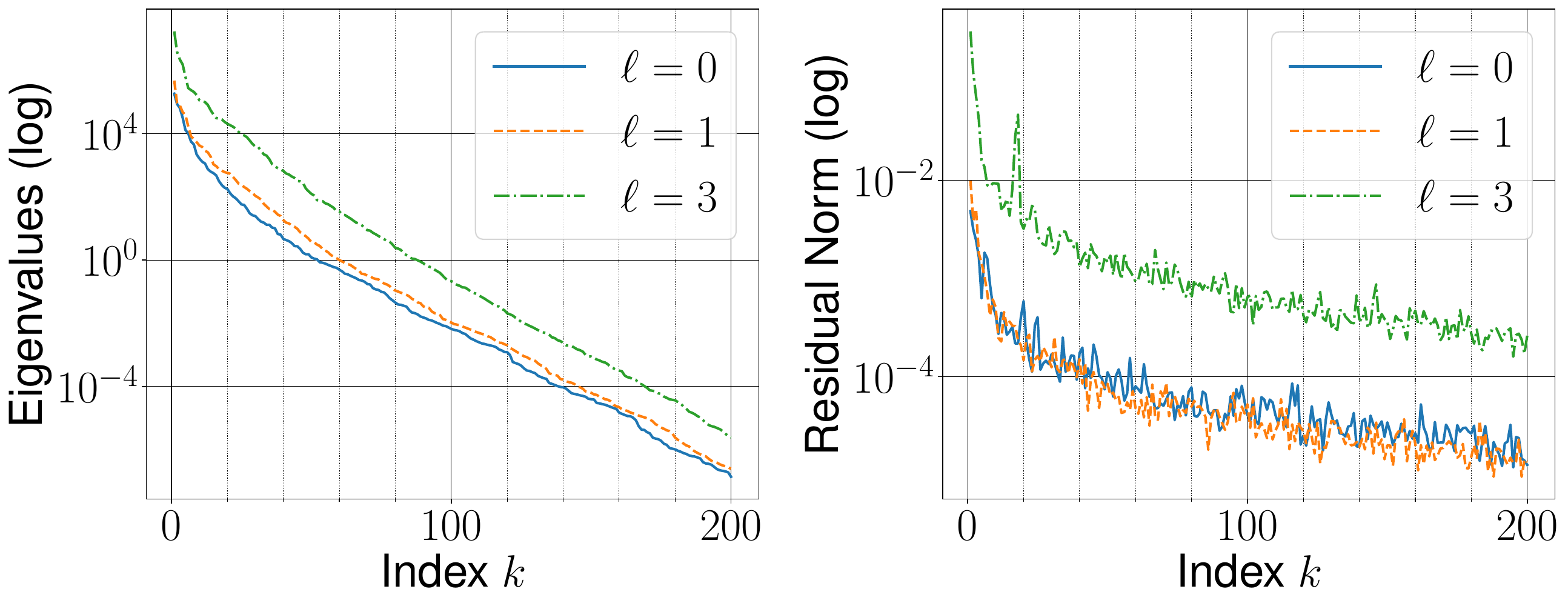}
  \caption{Leading eigenvalues, on a logarithmic scale, of the Hessian $\Hessmat$
    obtained by the two-pass algorithm~\cite{saibaba2016randomizedHermitCOPY},
    along with the residual norms. 
    Results are obtained with all sensors activated, i.e., $\design=\vec{1}$.
    }
  \label{fig:Hessian_Approx} 
\end{figure}
Fast decay in eigenvalues, with over 99\% of the variance explained by the leading $80$ 
eigenvalues in all three cases, supports the accuracy of the reduced-order
approximation  of the Hessian.


\subsection{Randomized estimator of the A-optimality criterion}
\label{supp:subsec:randomized_trace_estimation}
    In the numerical experiments in ~\Cref{sec:numerical_experiments} we 
    used the Hutchinson randomized trace estimator 
    to  approximate the A-optimality criterion, which enabled us to carry 
    out several comparative experiments efficiently, 
    and we set the sample size to $n_r=25$, which we believed would achieve a  highly accurate
    estimate of the optimality criterion.
    To test the validity of this assertion, in our settings we compared the exact 
    value of the posterior covariance trace with the randomized approximation. 
    Results indicating the accuracy of the posterior trace approximation by 
    randomization are shown in~\Cref{fig:Posterior_Trace_Randomization}.
    The value of the Hutchinson randomized trace estimator was evaluated for the
    three experimental setups discussed in this section.
    In each case, the trace estimate of the posterior covariance matrix trace was
    evaluated by using several choices of the sample size $n_r$.
    For each experimental setting and for each choice of the sample size, 
    the trace approximation was carried out $100$ times, each with a new sample.  
    The red stars show the true value of the posterior covariance trace in
    each case. 
    These results show that even using a sample of size $n_r=1$ generated from
    Rademacher distribution, we obtain a good approximation of the trace of the
    posterior covariance matrix, albeit exhibiting high variability around the
    true value. A much better estimate of the true value can be obtained by
    increasing the sample size $n_r$.
    Thus, in our experiments we followed this approach to  approximate the A-optimality
    criterion, which enabled us to carry out several comparative experiments
    efficiently, and we set the sample size to $n_r=25$. 
    \begin{figure}[ht]
      \centering
      \includegraphics[width=0.55\linewidth]{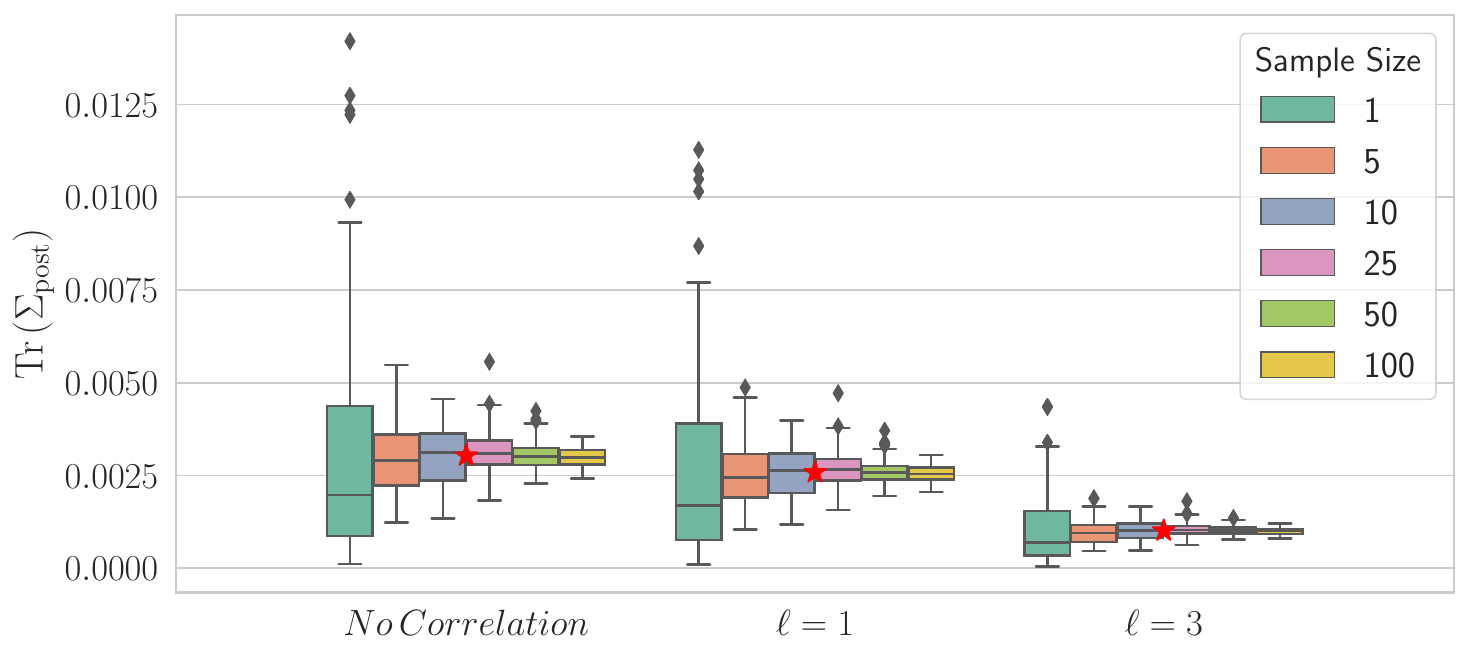}
      \caption{
        Results showing accuracy of posterior trace approximation by
        randomization using the Hutchinson trace estimator.
        Box plots are shown for the three cases studied in this section. 
        The first case considers the setup with no observation correlations, 
        and the other cases correspond to the experiments where observation correlations 
        are synthesized with length-scale $\ell$ set to $1$ and $3$, respectively. 
        The true value of the posterior covariance  trace $\Trace{\Cpredpost}$ 
        is plotted as a red star for each case.
        In each case the randomized trace is calculated by using various choices
        of the sample size  $n_r$, where the  random vectors are sampled
        from Rademacher distribution. The trace approximation is carried out
        $100$ times for each choice of the sample size and for each case to generate the boxplots. 
      }
      \label{fig:Posterior_Trace_Randomization}
    \end{figure}

\section*{Acknowledgments}
This work was partially supported by the U.S. Department of Energy, Office of Science, Advanced Scientific Computing Research Program under contract DE-AC02-06CH11357 and Laboratory Directed Research and Development (LDRD) funding from Argonne National Laboratory.
We thank three anonymous referees and the associate editor for their detailed and
insightful comments that helped us improve our manuscript.


\bibliographystyle{siamplain}
\bibliography{references}

\iftrue
\null 
  \begin{flushright}
  \scriptsize \framebox{\parbox{5.3in}{
  The submitted manuscript has been created by UChicago Argonne, LLC,
  Operator of Argonne National Laboratory (``Argonne"). Argonne, a
  U.S. Department of Energy Office of Science laboratory, is operated
  under Contract No. DE-AC02-06CH11357. The U.S. Government retains for
  itself, and others acting on its behalf, a paid-up nonexclusive,
  irrevocable worldwide license in said article to reproduce, prepare
  derivative works, distribute copies to the public, and perform
  publicly and display publicly, by or on behalf of the Government.
  The Department of
  Energy will provide public access to these results of federally sponsored research in accordance
  with the DOE Public Access Plan. http://energy.gov/downloads/doe-public-access-plan. }}
  \normalsize
  \end{flushright}
\fi

\end{document}